\documentclass[a4paper]{amsart}
\usepackage{a4wide,amsfonts,amssymb,amsmath,amscd,color,tikz,tikz-3dplot,tikz-cd,todonotes,mathtools}
\usepackage[english]{babel}
\allowdisplaybreaks

\newtheorem{lem}{Lemma}[section]

\newtheorem{theo}[lem]{Theorem}
\newtheorem{cor}[lem]{Corollary}
\newtheorem{rem}[lem]{Remark}

\newcommand{\Abb}[5]{\begin{array}{ccccc}#1&:&#2&\longrightarrow&#3\\{}&{}&#4&\longmapsto&#5\end{array}}
\def\bs{\boldsymbol}
\def\mr{\mathring}
\def\ol{\overline}
\def\To{\longrightarrow}

\def\incl{\hookrightarrow}
\DeclareMathOperator{\reals}{\mathbb{R}}
\DeclareMathOperator{\nat}{\mathbb{N}}
\def\ga{\Gamma}
\def\gat{\ga_{\!t}}
\def\gan{\ga_{\!n}}
\def\om{\Omega}
\def\rt{\reals^{3}}
\DeclareMathOperator{\RT}{\mathbb{RT}}
\def\Pone{\mathbb{P}^{1}}
\DeclareMathOperator{\mcB}{\mathcal{B}}
\DeclareMathOperator{\mcH}{\mathcal{H}}
\DeclareMathOperator{\sfL}{\mathsf{L}}
\DeclareMathOperator{\sfH}{\mathsf{H}}
\DeclareMathOperator{\mbH}{\bs{\mathsf{H}}}
\DeclareMathOperator{\sfC}{\mathsf{C}}
\newcommand{\Harm}[2]{\mcH^{#1}_{#2}}

\renewcommand{\L}[2]{\sfL^{#1}_{#2}}
\renewcommand{\H}[2]{\sfH^{#1}_{#2}}
\newcommand{\bH}[2]{\mbH^{#1}_{#2}}
\newcommand{\C}[2]{\sfC^{#1}_{#2}}
\newcommand{\B}[1]{\mcB^{#1}}
\newcommand{\vB}[2]{B^{#1}_{#2}}
\newcommand{\eps}{\varepsilon}
\renewcommand{\S}{\mathbb{S}}
\newcommand{\T}{\mathbb{T}}
\newcommand{\PotP}{\mathcal{P}}
\newcommand{\PotQ}{\mathcal{Q}}
\newcommand{\PotN}{\mathcal{N}}
\DeclareMathOperator{\Lin}{Lin}
\DeclareMathOperator{\A}{A}
\DeclareMathOperator{\p}{\partial}
\DeclareMathOperator{\id}{id}
\DeclareMathOperator{\sym}{sym}
\DeclareMathOperator{\skw}{skw}
\DeclareMathOperator{\tr}{tr}
\DeclareMathOperator{\dev}{dev}
\DeclareMathOperator{\spn}{spn}

\DeclareMathOperator{\ed}{d}

\DeclareMathOperator{\grad}{grad}
\DeclareMathOperator{\rot}{rot}
\DeclareMathOperator{\divergence}{div}
\def\div{\divergence}
\DeclareMathOperator{\Grad}{Grad}
\DeclareMathOperator{\bGrad}{\bs\Grad}
\DeclareMathOperator{\Rot}{Rot}
\DeclareMathOperator{\bRot}{\bs\Rot}
\DeclareMathOperator{\Div}{Div}
\DeclareMathOperator{\bDiv}{\bs\Div}
\DeclareMathOperator{\symGrad}{symGrad}
\DeclareMathOperator{\devGrad}{devGrad}
\DeclareMathOperator{\devRot}{devRot}
\DeclareMathOperator{\symRot}{symRot}
\DeclareMathOperator{\Gradgrad}{Gradgrad}
\DeclareMathOperator{\bGradgrad}{\bs\Gradgrad}
\DeclareMathOperator{\RotRot}{RotRot}

\DeclareMathOperator{\divDiv}{divDiv}
\DeclareMathOperator{\bdivDiv}{\bs\divDiv}
\newcommand{\RotS}{\Rot_{\S}}

\newcommand{\DivS}{\Div_{\S}}
\newcommand{\DivT}{\Div_{\T}}
\newcommand{\symRotT}{\symRot_{\T}}
\newcommand{\RotRottS}{\RotRot_{\S}^{\!\top}}
\newcommand{\divDivS}{\divDiv_{\S}}
\newcommand{\SGradgrad}{\prescript{}{\S}{\Gradgrad}}
\newcommand{\TGrad}{\prescript{}{\T}{\Grad}}
\newcommand{\TRotS}{\prescript{}{\T}{\Rot}_{\S}}
\newcommand{\SRotT}{\prescript{}{\S}{\Rot}_{\T}}

\newcommand{\divDivSgat}{\divDiv_{\S,\gat}}
\newcommand{\rdivDivSgat}{\mr\divDiv_{\S,\gat}}
\newcommand{\divDivSgan}{\divDiv_{\S,\gan}}
\newcommand{\DivTgat}{\Div_{\T,\gat}}
\newcommand{\rDivTgat}{\mr\Div_{\T,\gat}}
\newcommand{\DivTgan}{\Div_{\T,\gan}}
\newcommand{\SGradgradgat}{\prescript{}{\S}{\Gradgrad}_{\gat}}
\newcommand{\rSGradgradgat}{\prescript{}{\S}{\mr\Gradgrad}_{\gat}}
\newcommand{\TGradgat}{\prescript{}{\T}{\Grad}_{\gat}}
\newcommand{\rTGradgat}{\prescript{}{\T}{\mr\Grad}_{\gat}}
\newcommand{\TGradgan}{\prescript{}{\T}{\Grad}_{\gan}}
\newcommand{\TRotSgat}{\prescript{}{\T}{\Rot}_{\S,\gat}}
\newcommand{\rTRotSgat}{\prescript{}{\T}{\mr\Rot}_{\S,\gat}}
\newcommand{\TRotSgan}{\prescript{}{\T}{\Rot}_{\S,\gan}}
\newcommand{\SRotTgat}{\prescript{}{\S}{\Rot}_{\T,\gat}}
\newcommand{\rSRotTgat}{\prescript{}{\S}{\mr\Rot}_{\T,\gat}}
\newcommand{\SRotTgan}{\prescript{}{\S}{\Rot}_{\T,\gan}}
\newcommand{\bdivDivSgat}{\bdivDiv_{\S,\gat}}
\newcommand{\bdivDivSgan}{\bdivDiv_{\S,\gan}}
\newcommand{\bDivTgat}{\bDiv_{\T,\gat}}
\newcommand{\bDivTgan}{\bDiv_{\T,\gan}}
\newcommand{\SbGradgradgat}{\prescript{}{\S}{\bGradgrad}_{\gat}}
\newcommand{\SbGradgradgan}{\prescript{}{\S}{\bGradgrad}_{\gan}}
\newcommand{\TbGradgat}{\prescript{}{\T}{\bGrad}_{\gat}}
\newcommand{\TbGradgan}{\prescript{}{\T}{\bGrad}_{\gan}}
\newcommand{\TbRotSgat}{\prescript{}{\T}{\bRot}_{\S,\gat}}
\newcommand{\TbRotSgan}{\prescript{}{\T}{\bRot}_{\S,\gan}}
\newcommand{\SbRotTgat}{\prescript{}{\S}{\bRot}_{\T,\gat}}
\newcommand{\SbRotTgan}{\prescript{}{\S}{\bRot}_{\T,\gan}}
\newcommand{\norm}[1]{|#1|}

\newcommand{\scp}[2]{\langle#1,#2\rangle}

\title[Hilbert Complexes with Mixed Boundary Conditions -- Part 3: Biharmonic Complexes]
{Hilbert Complexes with Mixed Boundary Conditions\\
Part 3: Biharmonic Complexes}
\author{Dirk Pauly}
\author{Michael Schomburg}
\address{Institut f\"ur Analysis, Technische Universit\"at Dresden, Germany}
\email[Dirk Pauly]{dirk.pauly@tu-dresden.de}
\email[Michael Schomburg]{michael.schomburg@tu-dresden.de}
\date{\today; {\it Corresponding Author}: Dirk Pauly}
\thanks{{\it Acknowledgements.} We thank the anonymous referee for careful reading 
and helpful suggestions to improve the quality of the paper}
\thanks{{\it Conflict of interest statement.} This work does not have any conflicts of interest}
\thanks{{\it Funding Statement.} There are no funders to report for this submission}

\begin{document}


\begin{abstract}
We show that the biharmonic Hilbert complex
with mixed boundary conditions on bounded strong Lipschitz domains
is closed and compact. The crucial results are compact embeddings 
which follow by abstract arguments using functional analysis
together with particular regular decompositions.
Higher Sobolev order results are also proved.
This paper extends recent results on the de Rham and elasticity Hilbert complexes
with mixed boundary conditions from \cite{PS2021b,PS2021d}
and results on the biharmonic Hilbert complex 
with empty or full boundary conditions from \cite{PZ2020a}.
\end{abstract}


\maketitle
\setcounter{tocdepth}{3}
\tableofcontents


\section{Introduction}
\label{sec:intro}%

In \cite{PS2021b} we investigated the de Rham Hilbert complex
with mixed boundary conditions on bounded strong Lipschitz domains
\begin{equation*}
\def\arrowlength{5ex}
\def\arrowdistance{0}
\begin{tikzcd}[column sep=\arrowlength]
\cdots 
\arrow[r, rightarrow, shift left=\arrowdistance, "\cdots"] 
& 
\L{q-1,2}{}(\om) 
\ar[r, rightarrow, shift left=\arrowdistance, "\ed^{q-1}"] 
& 
[1em]
\L{q,2}{}(\om) 
\arrow[r, rightarrow, shift left=\arrowdistance, "\ed^{q}"] 
& 
\L{q+1,2}{}(\om) 
\arrow[r, rightarrow, shift left=\arrowdistance, "\cdots"] 
&
\cdots,
\end{tikzcd}
\end{equation*}
whose 3D-version for vector proxies reads
\begin{equation*}
\def\arrowlength{5ex}
\def\arrowdistance{0}
\begin{tikzcd}[column sep=\arrowlength]
\cdots 
\arrow[r, rightarrow, shift left=\arrowdistance, "\cdots"] 
& 
\L{2}{}(\om) 
\ar[r, rightarrow, shift left=\arrowdistance, "\ed^{0}\widehat{=}\grad"] 
& 
[2.5em]
\L{2}{}(\om)
\arrow[r, rightarrow, shift left=\arrowdistance, "\ed^{1}\widehat{=}\rot"] 
& 
[2.5em]
\L{2}{}(\om)
\arrow[r, rightarrow, shift left=\arrowdistance, "\ed^{2}\widehat{=}\div"] 
& 
[2.5em]
\L{2}{}(\om)
\arrow[r, rightarrow, shift left=\arrowdistance, "\cdots"] 
&
\cdots.
\end{tikzcd}
\end{equation*}
In \cite{PS2021d} we extended our studies and results to the elasticity complex
\begin{equation*}
\def\arrowlength{5ex}
\def\arrowdistance{0}
\begin{tikzcd}[column sep=\arrowlength]
\cdots 
\arrow[r, rightarrow, shift left=\arrowdistance, "\cdots"] 
& 
\L{2}{}(\om) 
\ar[r, rightarrow, shift left=\arrowdistance, "\symGrad"] 
& 
[2.5em]
\L{2}{\S}(\om)
\arrow[r, rightarrow, shift left=\arrowdistance, "\RotRottS"] 
& 
[2.5em]
\L{2}{\S}(\om)
\arrow[r, rightarrow, shift left=\arrowdistance, "\DivS"] 
& 
[1em]
\L{2}{}(\om)
\arrow[r, rightarrow, shift left=\arrowdistance, "\cdots"] 
&
\cdots.
\end{tikzcd}
\end{equation*}

In this contribution, the third part of the series,
we shall investigate the two biharmonic Hilbert complexes with mixed boundary conditions
on a bounded strong Lipschitz domain $\om\subset\rt$
\begin{equation*}
\def\arrowlength{5ex}
\def\arrowdistance{0}
\begin{tikzcd}[column sep=\arrowlength]
\cdots 
\arrow[r, rightarrow, shift left=\arrowdistance, "\cdots"] 
& 
\L{2}{}(\om) 
\ar[r, rightarrow, shift left=\arrowdistance, "\Gradgrad"] 
& 
[2.5em]
\L{2}{\S}(\om)
\arrow[r, rightarrow, shift left=\arrowdistance, "\RotS"] 
& 
[2.5em]
\L{2}{\T}(\om)
\arrow[r, rightarrow, shift left=\arrowdistance, "\DivT"] 
& 
[1em]
\L{2}{}(\om)
\arrow[r, rightarrow, shift left=\arrowdistance, "\cdots"] 
&
\cdots,
\end{tikzcd}
\end{equation*}
\begin{equation*}
\def\arrowlength{5ex}
\def\arrowdistance{0}
\begin{tikzcd}[column sep=\arrowlength]
\cdots 
\arrow[r, rightarrow, shift left=\arrowdistance, "\cdots"] 
& 
\L{2}{}(\om) 
\ar[r, rightarrow, shift left=\arrowdistance, "\devGrad"] 
& 
[2.5em]
\L{2}{\T}(\om)
\arrow[r, rightarrow, shift left=\arrowdistance, "\symRotT"] 
& 
[2.5em]
\L{2}{\S}(\om)
\arrow[r, rightarrow, shift left=\arrowdistance, "\divDivS"] 
& 
[1em]
\L{2}{}(\om)
\arrow[r, rightarrow, shift left=\arrowdistance, "\cdots"] 
&
\cdots.
\end{tikzcd}
\end{equation*}
Note that these two complexes are formally dual (adjoint) to each other.

As explained in detail in \cite{PS2021b,PS2021d},
all these Hilbert complexes share the same geometric structure 
\begin{equation*}
\def\arrowlength{5ex}
\def\arrowdistance{0}
\begin{tikzcd}[column sep=\arrowlength]
\cdots 
\arrow[r, rightarrow, shift left=\arrowdistance, "\cdots"] 
& 
\H{}{0} 
\ar[r, rightarrow, shift left=\arrowdistance, "\A_{0}"] 
& 
\H{}{1}
\arrow[r, rightarrow, shift left=\arrowdistance, "\A_{1}"] 
& 
\H{}{2}
\arrow[r, rightarrow, shift left=\arrowdistance, "\cdots"] 
&
\cdots,
\end{tikzcd}
\qquad
R(\A_{0})\subset N(\A_{1}),
\end{equation*}
where $\A_{0}$ and $\A_{1}$ are densely defined and closed (unbounded) linear operators
between Hilbert spaces $\H{}{\ell}$.
The corresponding domain Hilbert complex is denoted by
\begin{equation*}
\def\arrowlength{5ex}
\def\arrowdistance{0}
\begin{tikzcd}[column sep=\arrowlength]
\cdots 
\arrow[r, rightarrow, shift left=\arrowdistance, "\cdots"] 
& 
D(\A_{0})
\ar[r, rightarrow, shift left=\arrowdistance, "\A_{0}"] 
& 
D(\A_{1})
\arrow[r, rightarrow, shift left=\arrowdistance, "\A_{1}"] 
& 
\H{}{2}
\arrow[r, rightarrow, shift left=\arrowdistance, "\cdots"] 
&
\cdots.
\end{tikzcd}
\end{equation*}

The goal of this article is to show that the previous biharmonic Hilbert complexes are compact,
which is proved by using regular decompositions of the domains of definition 
of the respective operators as crucial tool.
We shall follow in close lines the rationale from \cite{PS2021b,PS2021d}.
Along the way we show the existence of regular potentials and decompositions,
compact embeddings, Helmholtz decompositions, 
closed ranges, Friedrichs/Poincar\'e type estimates, 
and bases of the corresponding cohomology groups 
(generalised Dirichlet/Neumann tensors).
Due to the similarity of results
we shall only state those which are most important.
In the appendix we will present some of the crucial proofs
which differ from the proofs of the previously investigated complexes.

\section{Biharmonic Complexes I}
\label{bih:sec:bih1}%

Throughout this paper, let
$\om\subset\rt$ be a \emph{bounded strong Lipschitz domain}
with boundary $\ga$, decomposed into two parts $\gat$ and $\gan:=\ga\setminus\ol{\gat}$
with some \emph{relatively open and strong Lipschitz boundary part} $\gat\subset\ga$.
More precisely, we assume generally that $(\om,\gat)$ is a \emph{bounded strong Lipschitz pair}.
We shall consequently use the notations, methods, and results from 
our corresponding papers for the de Rham complex \cite{PS2021b},
for the elasticity complex \cite{PZ2021a,PS2021d},
and for the biharmonic complexes \cite{PZ2020a}.
In particular, we recall \cite[Section 2, Section 3]{PS2021b}
including the notion of \emph{extendable domains}.
The standard Lebesgue and Sobolev spaces (scalar or tensor valued)
are denoted by $\L{2}{}(\om)$ and $\H{k}{}(\om)$ with $k\in\nat_{0}$.

We recall that weak and strong boundary conditions coincide
for the standard Sobolev spaces with mixed boundary conditions, i.e.,
\begin{align}
\label{eq:weakeqstrongsimple}
\bH{k}{\gat}(\om)=\H{k}{\gat}(\om),
\end{align}
cf. \cite[Lemma 3.2, Theorem 4.6]{PS2021b}.
Below we shall show that ``\emph{strong $=$ weak}'' holds generally also
for the biharmonic complex.
Note that $\H{k}{\emptyset}(\om)=\H{k}{}(\om)$ and $\H{0}{\gat}(\om)=\L{2}{}(\om)$.

We introduce as usual $\Grad$, $\Rot$, and $\Div$ as `row-wise' incarnations 
of the classical operators $\grad$, $\rot$, and $\div$
from the de Rham complex.

\subsection{Operators}
\label{bih:sec:bihop}%

Let
$\Gradgrad$, $\Rot$, $\Div$, $\devGrad$, $\symRot$, and $\divDiv$
be realised as densely defined (unbounded) linear operators
\begin{align*}
\rSGradgradgat:D(\rSGradgradgat)\subset\L{2}{}(\om)&\to\L{2}{\S}(\om);
&
u&\mapsto\Gradgrad u,\\
\rTRotSgat:D(\rTRotSgat)\subset\L{2}{\S}(\om)&\to\L{2}{\T}(\om);
&
S&\mapsto\Rot S,\\
\rDivTgat:D(\rDivTgat)\subset\L{2}{\T}(\om)&\to\L{2}{}(\om);
&
T&\mapsto\Div T,\\
\rTGradgat:D(\rTGradgat)\subset\L{2}{}(\om)&\to\L{2}{\T}(\om);
&
v&\mapsto\devGrad v,\\
\rSRotTgat:D(\rSRotTgat)\subset\L{2}{\T}(\om)&\to\L{2}{\S}(\om);
&
T&\mapsto\symRot T,\\
\rdivDivSgat:D(\rdivDivSgat)\subset\L{2}{\S}(\om)&\to\L{2}{}(\om);
&
S&\mapsto\divDiv S,
\end{align*}
where
$\sym S:=\frac{1}{2}(S+S^{\top})$ and 
$\dev T:=T-\frac{1}{3}(\tr T)\id$,
with domains of definition
\begin{align*}
D(\rSGradgradgat)&:=\C{\infty}{\gat}(\om),
&
D(\rTRotSgat)&:=\C{\infty}{\S,\gat}(\om),
&
D(\rDivTgat)&:=\C{\infty}{\T,\gat}(\om),\\
D(\rTGradgat)&:=\C{\infty}{\gat}(\om),
&
D(\rSRotTgat)&:=\C{\infty}{\T,\gat}(\om),
&
D(\rdivDivSgat)&:=\C{\infty}{\S,\gat}(\om),
\end{align*}
satisfying the complex properties
\begin{align*}
\rTRotSgat\,\rSGradgradgat&\subset0,
&
\rDivTgat\,\rTRotSgat\subset0,\\
\rSRotTgat\,\rTGradgat&\subset0,
&
\rdivDivSgat\,\rSRotTgat\subset0.
\end{align*}
For elementary properties of these operators see, e.g., \cite{PZ2021a},
in particular, we have a collection of formulas presented in Lemma \ref{bih:lem:PZformulalem}.
Here, we introduce the Lebesgue Hilbert spaces
and the test spaces of symmetric and deviatoric tensor fields
\begin{align*}
\L{2}{\S}(\om)
&:=\big\{S\in\L{2}{}(\om):\skw S=0\big\},
&
\C{\infty}{\S,\gat}(\om)
&:=\C{\infty}{\gat}(\om)\cap\L{2}{\S}(\om),\\
\L{2}{\T}(\om)
&:=\big\{S\in\L{2}{}(\om):\tr T=0\big\},
&
\C{\infty}{\T,\gat}(\om)
&:=\C{\infty}{\gat}(\om)\cap\L{2}{\T}(\om),
\end{align*}
respectively. We get the first and second biharmonic complexes on smooth tensor fields
\begin{equation*}
\def\arrowlength{5ex}
\def\arrowdistance{0}
\begin{tikzcd}[column sep=\arrowlength]
\cdots 
\arrow[r, rightarrow, shift left=\arrowdistance, "\cdots"] 
& 
\L{2}{}(\om) 
\ar[r, rightarrow, shift left=\arrowdistance, "\rSGradgradgat"] 
& 
[2.5em]
\L{2}{\S}(\om)
\arrow[r, rightarrow, shift left=\arrowdistance, "\rTRotSgat"] 
& 
[2.5em]
\L{2}{\T}(\om)
\arrow[r, rightarrow, shift left=\arrowdistance, "\rDivTgat"] 
& 
[1em]
\L{2}{}(\om)
\arrow[r, rightarrow, shift left=\arrowdistance, "\cdots"] 
&
\cdots,
\end{tikzcd}
\end{equation*}
\begin{equation*}
\def\arrowlength{5ex}
\def\arrowdistance{0}
\begin{tikzcd}[column sep=\arrowlength]
\cdots 
\arrow[r, rightarrow, shift left=\arrowdistance, "\cdots"] 
& 
\L{2}{}(\om) 
\ar[r, rightarrow, shift left=\arrowdistance, "\rTGradgat"] 
& 
[2.5em]
\L{2}{\T}(\om)
\arrow[r, rightarrow, shift left=\arrowdistance, "\rSRotTgat"] 
& 
[2.5em]
\L{2}{\S}(\om)
\arrow[r, rightarrow, shift left=\arrowdistance, "\rdivDivSgat"] 
& 
[1em]
\L{2}{}(\om)
\arrow[r, rightarrow, shift left=\arrowdistance, "\cdots"] 
&
\cdots.
\end{tikzcd}
\end{equation*}
For a more algebraically structured introduction of the latter operators
suggested by Rainer Picard see Appendix \ref{bih:app:sec:revop}.
The closures 
\begin{align*}
\SGradgradgat&:=\ol{\rSGradgradgat},
&
\TRotSgat&:=\ol{\rTRotSgat},
&
\DivTgat&:=\ol{\rDivTgat},\\
\TGradgat&:=\ol{\rTGradgat},
&
\SRotTgat&:=\ol{\rSRotTgat},
&
\divDivSgat&:=\ol{\rdivDivSgat}
\end{align*}
and Hilbert space adjoints 
are given by the densely defined and closed linear operators
\begin{align*}
\SGradgradgat:D(\SGradgradgat)\subset\L{2}{}(\om)&\to\L{2}{\S}(\om);
&
u&\mapsto\Gradgrad u,\\
\SGradgradgat^{*}=\bdivDivSgan:D(\bdivDivSgan)\subset\L{2}{\S}(\om)&\to\L{2}{}(\om);
&
S&\mapsto\divDiv S,\\
\TRotSgat:D(\TRotSgat)\subset\L{2}{\S}(\om)&\to\L{2}{\T}(\om);
&
S&\mapsto\Rot S,\\
\TRotSgat^{*}=\SbRotTgan:D(\SbRotTgan)\subset\L{2}{\T}(\om)&\to\L{2}{\S}(\om);
&
T&\mapsto\symRot T,\\
\DivTgat:D(\DivTgat)\subset\L{2}{\T}(\om)&\to\L{2}{}(\om);
&
T&\mapsto\Div T,\\
\DivTgat^{*}=-\TbGradgan:D(\TbGradgan)\subset\L{2}{}(\om)&\to\L{2}{\T}(\om);
&
v&\mapsto-\devGrad v,\\
\TGradgat:D(\TGradgat)\subset\L{2}{}(\om)&\to\L{2}{\T}(\om);
&
v&\mapsto\devGrad v,\\
\TGradgat^{*}=-\bDivTgan:D(\bDivTgan)\subset\L{2}{\T}(\om)&\to\L{2}{}(\om);
&
T&\mapsto-\Div T,\\
\SRotTgat:D(\SRotTgat)\subset\L{2}{\T}(\om)&\to\L{2}{\S}(\om);
&
T&\mapsto\symRot T,\\
\SRotTgat^{*}=\TbRotSgan:D(\TbRotSgan)\subset\L{2}{\S}(\om)&\to\L{2}{\T}(\om);
&
S&\mapsto\Rot S,\\
\divDivSgat:D(\divDivSgat)\subset\L{2}{\S}(\om)&\to\L{2}{}(\om);
&
S&\mapsto\divDiv S,\\
\divDivSgat^{*}=\SbGradgradgan:D(\SbGradgradgan)\subset\L{2}{}(\om)&\to\L{2}{\S}(\om);
&
u&\mapsto\Gradgrad u,
\end{align*}
with domains of definition
\begin{align*}
D(\SGradgradgat)&=\H{}{\gat}(\Gradgrad,\om),
&
D(\bdivDivSgan)&=\bH{}{\S,\gan}(\divDiv,\om),\\
D(\TRotSgat)&=\H{}{\S,\gat}(\Rot,\om),
&
D(\SbRotTgan)&=\bH{}{\T,\gan}(\symRot,\om),\\
D(\DivTgat)&=\H{}{\T,\gat}(\Div,\om),
&
D(\TbGradgan)&=\bH{}{\gan}(\devGrad,\om),\\
D(\TGradgat)&=\H{}{\gat}(\devGrad,\om),
&
D(\bDivTgan)&=\bH{}{\T,\gan}(\Div,\om),\\
D(\SRotTgat)&=\H{}{\T,\gat}(\symRot,\om),
&
D(\TbRotSgan)&=\bH{}{\S,\gan}(\Rot,\om),\\
D(\divDivSgat)&=\H{}{\S,\gat}(\divDiv,\om),
&
D(\SbGradgradgan)&=\bH{}{\gan}(\Gradgrad,\om).
\end{align*}
We shall introduce the latter Sobolev spaces in the next section.

\subsection{Sobolev Spaces}
\label{bih:sec:sobolev}%

Let
\begin{align*}
\H{}{}(\Gradgrad,\om)
&:=\big\{u\in\L{2}{}(\om):\Gradgrad u\in\L{2}{}(\om)\big\},\\
\H{}{\S}(\Rot,\om)
&:=\big\{S\in\L{2}{\S}(\om):\Rot S\in\L{2}{}(\om)\big\},\\
\H{}{\T}(\Div,\om)
&:=\big\{T\in\L{2}{\T}(\om):\Div T\in\L{2}{}(\om)\big\},\\
\H{}{}(\devGrad,\om)
&:=\big\{v\in\L{2}{}(\om):\devGrad v\in\L{2}{}(\om)\big\},\\
\H{}{\T}(\symRot,\om)
&:=\big\{T\in\L{2}{\T}(\om):\symRot T\in\L{2}{}(\om)\big\},\\
\H{}{\S}(\divDiv,\om)
&:=\big\{S\in\L{2}{\S}(\om):\divDiv S\in\L{2}{}(\om)\big\}.
\end{align*}
Note that $S\in\H{}{\S}(\Rot,\om)$ implies $\Rot S\in\L{2}{\T}(\om)$,
cf.~Lemma \ref{bih:lem:PZformulalem},
and that we have by Ne\v{c}as' inequality 
and a Korn type inequality for $\dev$ the regularities 
\begin{align}
\label{eq:necasreg}
\H{}{}(\Gradgrad,\om)=\H{2}{}(\om),\qquad
\H{}{}(\devGrad,\om)=\H{1}{}(\om)
\end{align}
with equivalent norms, see, e.g., 
\cite[Lemma 8.2]{PW2021a} and \cite[Lemma 3.2]{PZ2020a}.
Moreover, we define boundary conditions in the \emph{strong sense} 
as closures of respective test fields, i.e.,
\begin{align*}
\H{}{\gat}(\Gradgrad,\om)
&:=\ol{\C{\infty}{\gat}(\om)}^{\H{2}{}(\om)}
=\H{2}{\gat}(\om),\\
\H{}{\S,\gat}(\Rot,\om)
&:=\ol{\C{\infty}{\S,\gat}(\om)}^{\H{}{\S}(\Rot,\om)},\\
\H{}{\T,\gat}(\Div,\om)
&:=\ol{\C{\infty}{\T,\gat}(\om)}^{\H{}{\T}(\Div,\om)},\\
\H{}{\gat}(\devGrad,\om)
&:=\ol{\C{\infty}{\gat}(\om)}^{\H{1}{}(\om)}
=\H{1}{\gat}(\om),\\
\H{}{\T,\gat}(\symRot,\om)
&:=\ol{\C{\infty}{\T,\gat}(\om)}^{\H{}{\T}(\symRot,\om)},\\
\H{}{\S,\gat}(\divDiv,\om)
&:=\ol{\C{\infty}{\S,\gat}(\om)}^{\H{}{\S}(\divDiv,\om)}.
\end{align*}
For $\gat=\emptyset$ we may skip the index $\emptyset$
which is justified by density.
Spaces with vanishing differential operator coincide with kernels
and are denoted by an additional index $0$ 
at the lower right corner, e.g., 
$$\H{}{\S,\gat,0}(\Rot,\om)=N(\TRotSgat),\qquad
\H{}{\T,\gat,0}(\Div,\om)=N(\DivTgat).$$

We need also the Sobolev spaces with boundary conditions defined 
in the \emph{weak sense}, i.e.,
\begin{align*}
\bH{}{\gat}(\Gradgrad,\om)
&:=\big\{u\in\H{2}{}(\om):
\scp{\Gradgrad u}{\Phi}_{\L{2}{\S}(\om)}=\scp{u}{\divDiv\Phi}_{\L{2}{}(\om)}\\
&\hspace{60mm}\forall\,\Phi\in\C{\infty}{\S,\gan}(\om)\big\},\\
\bH{}{\S,\gat}(\Rot,\om)
&:=\big\{S\in\H{}{\S}(\Rot,\om):
\scp{\Rot S}{\Psi}_{\L{2}{\T}(\om)}=\scp{S}{\symRot\Psi}_{\L{2}{\S}(\om)}\\
&\hspace{60mm}\forall\,\Psi\in\C{\infty}{\T,\gan}(\om)\big\},\\
\bH{}{\T,\gat}(\Div,\om)
&:=\big\{T\in\H{}{\T}(\Div,\om):
\scp{\Div T}{\phi}_{\L{2}{}(\om)}=-\scp{T}{\devGrad\phi}_{\L{2}{\T}(\om)}\\
&\hspace{60mm}\forall\,\phi\in\C{\infty}{\gan}(\om)\big\},\\
\bH{}{\gat}(\devGrad,\om)
&:=\big\{v\in\H{1}{}(\om):
\scp{\devGrad v}{\Psi}_{\L{2}{\T}(\om)}=-\scp{v}{\Div\Psi}_{\L{2}{}(\om)}\\
&\hspace{60mm}\forall\,\Psi\in\C{\infty}{\T,\gan}(\om)\big\},\\
\bH{}{\T,\gat}(\symRot,\om)
&:=\big\{T\in\H{}{\T}(\symRot,\om):
\scp{\symRot T}{\Phi}_{\L{2}{\S}(\om)}=\scp{T}{\Rot\Phi}_{\L{2}{\T}(\om)}\\
&\hspace{60mm}\forall\,\Phi\in\C{\infty}{\S,\gan}(\om)\big\},\\
\bH{}{\S,\gat}(\divDiv,\om)
&:=\big\{S\in\H{}{\S}(\divDiv,\om):
\scp{\divDiv S}{\phi}_{\L{2}{}(\om)}=\scp{S}{\Gradgrad\phi}_{\L{2}{\S}(\om)}\\
&\hspace{60mm}\forall\,\phi\in\C{\infty}{\gan}(\om)\big\}.
\end{align*}
Note that ``\emph{strong $\subset$ weak}'' holds, i.e., 
$\H{}{\cdots}(\cdots,\om)\subset\bH{}{\cdots}(\cdots,\om)$, e.g., 
$$\H{}{\S,\gat}(\Rot,\om)\subset\bH{}{\S,\gat}(\Rot,\om),\qquad
\H{}{\T,\gat}(\Div,\om)\subset\bH{}{\T,\gat}(\Div,\om),$$
and that the complex properties hold in both the strong and the weak case, e.g.,
$$\devGrad\H{}{\gat}(\om)\subset\H{}{\T,\gat,0}(\symRot,\om),\qquad
\Rot\bH{}{\S,\gat}(\Rot,\om)\subset\bH{}{\T,\gat,0}(\Div,\om),$$
which follows immediately by the definitions.
In Remark \ref{bih:rem:weakeqstrongbih} below we comment on the question 
whether ``\emph{strong $=$ weak}'' holds in general.

\subsection{Higher Order Sobolev Spaces}
\label{bih:sec:highsobolev}%

For $k\in\nat_{0}$ we define higher order Sobolev spaces by
\begin{align*}
\H{k}{\S}(\om)
&:=\H{k}{}(\om)\cap\L{2}{\S}(\om),\\
\H{k}{\T}(\om)
&:=\H{k}{}(\om)\cap\L{2}{\T}(\om),\\
\H{k}{\S,\gat}(\om)
&:=\ol{\C{\infty}{\S,\gat}(\om)}^{\H{k}{}(\om)}
=\H{k}{\gat}(\om)\cap\L{2}{\S}(\om),\\
\H{k}{\T,\gat}(\om)
&:=\ol{\C{\infty}{\T,\gat}(\om)}^{\H{k}{}(\om)}
=\H{k}{\gat}(\om)\cap\L{2}{\T}(\om),\\
\H{k}{}(\Gradgrad,\om)
&:=\big\{u\in\H{k}{}(\om):\Gradgrad u\in\H{k}{}(\om)\big\},\\
\H{k}{\gat}(\Gradgrad,\om)
&:=\big\{u\in\H{k}{\gat}(\om)\cap\H{}{\gat}(\Gradgrad,\om):\Gradgrad u\in\H{k}{\gat}(\om)\big\},\\
\H{k}{\S}(\Rot,\om)
&:=\big\{S\in\H{k}{\S}(\om):\Rot S\in\H{k}{}(\om)\big\},\\
\H{k}{\S,\gat}(\Rot,\om)
&:=\big\{S\in\H{k}{\gat}(\om)\cap\H{}{\S,\gat}(\Rot,\om):\Rot S\in\H{k}{\gat}(\om)\big\},\\
\H{k}{\T}(\Div,\om)
&:=\big\{T\in\H{k}{\T}(\om):\Div T\in\H{k}{}(\om)\big\},\\
\H{k}{\T,\gat}(\Div,\om)
&:=\big\{T\in\H{k}{\gat}(\om)\cap\H{}{\T,\gat}(\Div,\om):\Div T\in\H{k}{\gat}(\om)\big\},\\
\H{k}{}(\devGrad,\om)
&:=\big\{v\in\H{k}{}(\om):\devGrad v\in\H{k}{}(\om)\big\},\\
\H{k}{\gat}(\devGrad,\om)
&:=\big\{v\in\H{k}{\gat}(\om)\cap\H{}{\gat}(\devGrad,\om):\devGrad v\in\H{k}{\gat}(\om)\big\},\\
\H{k}{\T}(\symRot,\om)
&:=\big\{T\in\H{k}{\T}(\om):\symRot T\in\H{k}{}(\om)\big\},\\
\H{k}{\T,\gat}(\symRot,\om)
&:=\big\{T\in\H{k}{\gat}(\om)\cap\H{}{\T,\gat}(\symRot,\om):\symRot T\in\H{k}{\gat}(\om)\big\},\\
\H{k}{\S}(\divDiv,\om)
&:=\big\{S\in\H{k}{\S}(\om):\divDiv S\in\H{k}{}(\om)\big\},\\
\H{k}{\S,\gat}(\divDiv,\om)
&:=\big\{S\in\H{k}{\gat}(\om)\cap\H{}{\S,\gat}(\divDiv,\om):\divDiv S\in\H{k}{\gat}(\om)\big\}.
\end{align*}
For the first reading we recommend to only regard the case $k=0$
from  Section \ref{bih:sec:sobolev}.

Note that, e.g., for the latter $\divDiv$-Sobolev spaces,
we have $\H{k}{\S,\emptyset}(\divDiv,\om)=\H{k}{\S}(\divDiv,\om)$
and $\H{0}{\S,\emptyset}(\divDiv,\om)=\H{}{\S}(\divDiv,\om)$
as well as $\H{0}{\S,\gat}(\divDiv,\om)=\H{}{\S,\gat}(\divDiv,\om)$.
For $\gat\neq\emptyset$ it holds
$$\H{k}{\S,\gat}(\divDiv,\om)
=\big\{S\in\H{k}{\S,\gat}(\om):\divDiv S\in\H{k}{\gat}(\om)\big\},\qquad
k\geq2,$$
but for $\gat\neq\emptyset$ and $k=0$ and $k=1$
\begin{align*}
\H{0}{\S,\gat}(\divDiv,\om)
&=\H{}{\S,\gat}(\divDiv,\om)\\
&\subsetneq\big\{S\in\underbrace{\H{0}{\S,\gat}(\om)}_{=\L{2}{\S}(\om)}:
\divDiv S\in\underbrace{\H{0}{\gat}(\om)}_{=\L{2}{}(\om)}\big\}
=\H{}{\S}(\divDiv,\om),\\
\H{1}{\S,\gat}(\divDiv,\om)
&\subsetneq\big\{S\in\H{1}{\S,\gat}(\om):\divDiv S\in\H{1}{\gat}(\om)\big\},
\end{align*}
respectively. As before, we introduce the kernels
\begin{align*}
\H{k}{\S,\gat,0}(\divDiv,\om)
&:=\H{k}{\gat}(\om)\cap\H{}{\S,\gat,0}(\divDiv,\om)
=\H{k}{\S,\gat}(\divDiv,\om)\cap\H{}{\S,0}(\divDiv,\om)\\
&\;=\big\{S\in\H{k}{\S,\gat}(\divDiv,\om):\divDiv S=0\big\}.
\end{align*}
The corresponding remarks and definitions extend also
to the $\H{k}{\gat}(\Gradgrad,\om)$, $\H{k}{\S,\gat}(\Rot,\om)$, $\H{k}{\T,\gat}(\Div,\om)$,  
$\H{k}{\gat}(\devGrad,\om)$, and $\H{k}{\T,\gat}(\symRot,\om)$-spaces.
In particular, we have for $\gat\neq\emptyset$ and $k\geq1$ and, e.g.,
$\H{k}{\S,\gat}(\Rot,\om)$, the observations 
\begin{align*}
\H{k}{\S,\gat}(\Rot,\om)
&=\big\{S\in\H{k}{\S,\gat}(\om):\Rot S\in\H{k}{\gat}(\om)\big\},\\
\H{0}{\S,\gat}(\Rot,\om)
&=\H{}{\S,\gat}(\Rot,\om)
\subsetneq\big\{S\in\underbrace{\H{0}{\S,\gat}(\om)}_{=\L{2}{\S}(\om)}:
\Rot S\in\underbrace{\H{0}{\gat}(\om)}_{=\L{2}{}(\om)}\big\}
=\H{}{\S}(\Rot,\om),\\
\H{k}{\S,\gat,0}(\Rot,\om)
&=\H{k}{\gat}(\om)\cap\H{}{\S,\gat,0}(\Rot,\om)
=\H{k}{\S,\gat}(\Rot,\om)\cap\H{}{\S,0}(\Rot,\om)\\
&=\big\{S\in\H{k}{\S,\gat}(\Rot,\om):\Rot S=0\big\}.
\end{align*}

Analogously, we define the Sobolev spaces 
$\bH{k}{\gat}(\Gradgrad,\om)$, $\bH{k}{\S,\gat}(\Rot,\om)$, $\bH{k}{\T,\gat}(\Div,\om)$,  
$\bH{k}{\gat}(\devGrad,\om)$, $\bH{k}{\T,\gat}(\symRot,\om)$, and $\bH{k}{\S,\gat}(\divDiv,\om)$
using the respective Sobolev spaces with weak boundary conditions $\bH{\cdots}{\cdots}(\cdots,\om)$
in the definitions, e.g.,
\begin{align*}
\bH{k}{\T,\gat}(\symRot,\om)
&:=\big\{T\in\bH{k}{\gat}(\om)\cap\bH{}{\T,\gat}(\symRot,\om):\symRot T\in\bH{k}{\gat}(\om)\big\}\\
&\;=\big\{T\in\H{k}{\gat}(\om)\cap\bH{}{\T,\gat}(\symRot,\om):\symRot T\in\H{k}{\gat}(\om)\big\},
\end{align*} 
where we have used \eqref{eq:weakeqstrongsimple}.
Note that again ``\emph{strong $\subset$ weak}'' holds, i.e., 
$\H{\cdots}{\cdots}(\cdots,\om)\subset\bH{\cdots}{\cdots}(\cdots,\om)$, e.g., 
$\H{k}{\S,\gat}(\Rot,\om)\subset\bH{k}{\S,\gat}(\Rot,\om)$, 
and that the complex properties hold in both the strong and the weak case, e.g.,
$$\Gradgrad\H{k+2}{\gat}(\om)\subset\H{k}{\S,\gat,0}(\Rot,\om),\qquad
\symRot\bH{k}{\T,\gat}(\symRot,\om)\subset\bH{k}{\S,\gat,0}(\divDiv,\om).$$

In the forthcoming sections we shall also investigate
whether indeed ``\emph{strong $=$ weak}'' holds.
We start with a simple implication from \eqref{eq:weakeqstrongsimple}.

\begin{cor}
\label{bih:cor:weakeqstrongderham}
$\bH{k}{\S,\gat}(\om)=\H{k}{\S,\gat}(\om)$ and 
$\bH{k}{\T,\gat}(\om)=\H{k}{\T,\gat}(\om)$, i.e.,
weak and strong boundary conditions coincide
for the standard Sobolev spaces of symmetric and deviatoric tensor fields 
with mixed boundary conditions, respectively.
\end{cor}

As in \eqref{eq:necasreg} and with Corollary \ref{bih:cor:weakeqstrongderham}
we get the following.

\begin{lem}[higher order weak and strong partial boundary conditions coincide]
\label{bih:lem:weakeqstrongbih}
\mbox{}
\begin{itemize}
\item[\bf(i)]
For $k\geq0$ it holds
\begin{align*}
\H{k}{\gat}(\devGrad,\om)&=\H{k+1}{\gat}(\om)=\bH{k+1}{\gat}(\om),\\
\H{k}{\gat}(\Gradgrad,\om)&=\H{k+2}{\gat}(\om)=\bH{k+2}{\gat}(\om).
\end{align*}
\item[\bf(ii)]
For $k\geq1$ it holds
\begin{align*}
\bH{k}{\gat}(\devGrad,\om)
&=\big\{v\in\H{k}{\gat}(\om):\devGrad v\in\H{k}{\gat}(\om)\big\}
=\H{k}{\gat}(\devGrad,\om)
=\H{k+1}{\gat}(\om),\\
\bH{k}{\S,\gat}(\Rot,\om)
&=\big\{S\in\H{k}{\S,\gat}(\om):\Rot S\in\H{k}{\gat}(\om)\big\}
=\H{k}{\S,\gat}(\Rot,\om),\\
\bH{k}{\T,\gat}(\symRot,\om)
&=\big\{T\in\H{k}{\T,\gat}(\om):\symRot T\in\H{k}{\gat}(\om)\big\}
=\H{k}{\T,\gat}(\symRot,\om),\\
\bH{k}{\T,\gat}(\Div,\om)
&=\big\{T\in\H{k}{\T,\gat}(\om):\Div T\in\H{k}{\gat}(\om)\big\}
=\H{k}{\T,\gat}(\Div,\om).
\end{align*}
\item[\bf(iii)]
For $k\geq2$ it holds
\begin{align*}
\bH{k}{\gat}(\Gradgrad,\om)
&=\big\{u\in\H{k}{\gat}(\om):\Gradgrad u\in\H{k}{\gat}(\om)\big\}
=\H{k}{\gat}(\Gradgrad,\om),\\
\bH{k}{\S,\gat}(\divDiv,\om)
&=\big\{S\in\H{k}{\S,\gat}(\om):\divDiv S\in\H{k}{\gat}(\om)\big\}
=\H{k}{\S,\gat}(\divDiv,\om).
\end{align*}
\end{itemize}
\end{lem}

\begin{rem}[weak and strong partial boundary conditions coincide]
\label{bih:rem:weakeqstrongbih}
In \cite{PZ2020a,PW2021a} we could prove the corresponding results ``\emph{strong $=$ weak}'' 
for the whole two biharmonic complexes
but only with empty or full boundary conditions ($\ga_{t}=\emptyset$ or $\ga_{t}=\ga$).
Therefore, in these special cases, the adjoints are well-defined on the spaces with strong boundary conditions as well.

Lemma \ref{bih:lem:weakeqstrongbih} shows that
for higher values of $k$ indeed ``\emph{strong $=$ weak}'' holds.
Thus to show ``\emph{strong $=$ weak}'' in general 
we only have to prove that equality holds in the remains cases
$k=0$ and $k=1$, i.e.,
we only have to show
\begin{align*}
\bH{}{\gat}(\devGrad,\om)
&\subset\H{}{\gat}(\devGrad,\om),
&
\bH{}{\gat}(\Gradgrad,\om)
&\subset\H{}{\gat}(\Gradgrad,\om),\\
\bH{}{\T,\gat}(\Div,\om)
&\subset\H{}{\T,\gat}(\Div,\om),
&
\bH{1}{\gat}(\Gradgrad,\om)
&\subset\H{1}{\gat}(\Gradgrad,\om),\\
\bH{}{\S,\gat}(\Rot,\om)
&\subset\H{}{\S,\gat}(\Rot,\om),
&
\bH{}{\S,\gat}(\divDiv,\om)
&\subset\H{}{\S,\gat}(\divDiv,\om),\\
\bH{}{\T,\gat}(\symRot,\om)
&\subset\H{}{\T,\gat}(\symRot,\om),
&
\bH{1}{\S,\gat}(\divDiv,\om)
&\subset\H{1}{\S,\gat}(\divDiv,\om).
\end{align*}

The most delicate situation appears due to the second order nature of $\divDivS$.
In Corollary \ref{bih:cor:weakstrongbih} we shall show 
using regular decompositions that these results 
(weak and strong boundary conditions coincide for the biharmonic complexes for all $k\geq0$)
indeed hold true.
\end{rem}

\subsection{More Sobolev Spaces}
\label{bih:sec:sobolevmore}%

For $k\in\nat$ we introduce also slightly less regular higher order Sobolev spaces by
\begin{align*}
\H{k,k-1}{\gat}(\Gradgrad,\om)
&:=\big\{u\in\H{k}{\gat}(\om)\cap\H{}{\gat}(\Gradgrad,\om):\Gradgrad u\in\H{k-1}{\gat}(\om)\big\},\\
\bH{k,k-1}{\gat}(\Gradgrad,\om)
&:=\big\{u\in\bH{k}{\gat}(\om)\cap\bH{}{\gat}(\Gradgrad,\om):\Gradgrad u\in\bH{k-1}{\gat}(\om)\big\},\\
\H{k,k-1}{\S,\gat}(\divDiv,\om)
&:=\big\{S\in\H{k}{\gat}(\om)\cap\H{}{\S,\gat}(\divDiv,\om):\divDiv S\in\H{k-1}{\gat}(\om)\big\},\\
\bH{k,k-1}{\S,\gat}(\divDiv,\om)
&:=\big\{S\in\bH{k}{\gat}(\om)\cap\bH{}{\S,\gat}(\divDiv,\om):\divDiv S\in\bH{k-1}{\gat}(\om)\big\},
\end{align*}
and we extend all conventions of our notations.
These spaces can be ignored at the first reading.

We have for the kernels of $\divDivS$
\begin{align*}
\H{k,k-1}{\S,\gat,0}(\divDiv,\om)
&=\H{k}{\S,\gat,0}(\divDiv,\om),
&
\bH{k,k-1}{\S,\gat,0}(\divDiv,\om)
&=\bH{k}{\S,\gat,0}(\divDiv,\om),
\end{align*}
and by Ne\v{c}as' inequality, cf.~\eqref{eq:necasreg},
\begin{align*}
\H{k,k-1}{\gat}(\Gradgrad,\om)
&=\H{k+1}{\gat}(\om)
=\bH{k+1}{\gat}(\om)
\subset\bH{k,k-1}{\gat}(\Gradgrad,\om).
\end{align*}
The intersection with $\H{}{\gat}(\Gradgrad,\om)$, $\bH{}{\gat}(\Gradgrad,\om)$ 
and $\H{}{\S,\gat}(\divDiv,\om)$, $\bH{}{\S,\gat}(\divDiv,\om)$, respectively, 
is only needed if $k=1$.
As before, we observe 
$\H{k,k-1}{\S,\gat}(\divDiv,\om)\subset\bH{k,k-1}{\S,\gat}(\divDiv,\om)$,
i.e., ``\emph{strong $\subset$ weak}'',
and in both cases (weak and strong) the complex properties hold, e.g.,
$\Gradgrad\H{k,k-1}{\gat}(\Gradgrad,\om)\subset\H{k-1}{\S,\gat,0}(\Rot,\om)$.

Similar to Lemma \ref{bih:lem:weakeqstrongbih} we have the following.

\begin{lem}[higher order weak and strong partial boundary conditions coincide]
\label{bih:lem:weakeqstrongela2}
For $k\geq2$
{\small
\begin{align*}
\bH{k,k-1}{\gat}(\Gradgrad,\om)
&=\big\{u\in\H{k}{\gat}(\om):\Gradgrad u\in\H{k-1}{\gat}(\om)\big\}
=\H{k,k-1}{\gat}(\Gradgrad,\om)
=\H{k+1}{\gat}(\om),\\
\bH{k,k-1}{\S,\gat}(\divDiv,\om)
&=\big\{S\in\H{k}{\S,\gat}(\om):\divDiv S\in\H{k-1}{\gat}(\om)\big\}
=\H{k,k-1}{\S,\gat}(\divDiv,\om).
\end{align*}}
\end{lem}

\subsection{Some Biharmonic Complexes}
\label{bih:sec:bihcomplexes}%

By definition we have densely defined and closed (unbounded) linear operators 
defining six dual pairs 
\begin{align*}
(\SGradgradgat,\SGradgradgat^{*})
&=(\SGradgradgat,\bdivDivSgan),\\
(\TRotSgat,\TRotSgat^{*})
&=(\TRotSgat,\SbRotTgan),\\
(\DivTgat,\DivTgat^{*})
&=(\DivTgat,-\TbGradgan),\\
(\TGradgat,\TGradgat^{*})
&=(\TGradgat,-\bDivTgan),\\
(\SRotTgat,\SRotTgat^{*})
&=(\SRotTgat,\TbRotSgan),\\
(\divDivSgat,\divDivSgat^{*})
&=(\divDivSgat,\SbGradgradgan).
\end{align*}
\cite[Remark 2.5, Remark 2.6]{PS2021b} show the complex properties 
\begin{align*}
\TRotSgat\SGradgradgat&\subset0,
&
\DivTgat\TRotSgat&\subset0,\\
\SRotTgat\TGradgat&\subset0,
&
\divDivSgat\SRotTgat&\subset0,\\
\bdivDivSgan\SbRotTgan&\subset0,
&
-\SbRotTgan\TbGradgan&\subset0,\\
-\bDivTgan\TbRotSgan&\subset0,
&
\TbRotSgan\SbGradgradgan&\subset0.
\end{align*}
Hence we get the two primal and dual biharmonic Hilbert complexes
\begin{equation}
\label{bih:bihcomplex1a}
\def\arrowlength{16ex}
\def\arrowdistance{.8}
\begin{tikzcd}[column sep=\arrowlength]
\cdots
\arrow[r, rightarrow, shift left=\arrowdistance, "\cdots"] 
\arrow[r, leftarrow, shift right=\arrowdistance, "\cdots"']
& 
[-5em]
\L{2}{}(\om) 
\ar[r, rightarrow, shift left=\arrowdistance, "\SGradgradgat"] 
\ar[r, leftarrow, shift right=\arrowdistance, "\bdivDivSgan"']
&
[-1em]
\L{2}{\S}(\om) 
\ar[r, rightarrow, shift left=\arrowdistance, "\TRotSgat"] 
\ar[r, leftarrow, shift right=\arrowdistance, "\SbRotTgan"']
& 
[-1em]
\L{2}{\T}(\om) 
\arrow[r, rightarrow, shift left=\arrowdistance, "\DivTgat"] 
\arrow[r, leftarrow, shift right=\arrowdistance, "-\TbGradgan"']
& 
[-1em]
\L{2}{}(\om) 
\arrow[r, rightarrow, shift left=\arrowdistance, "\cdots"] 
\arrow[r, leftarrow, shift right=\arrowdistance, "\cdots"']
&
[-5em]
\cdots,
\end{tikzcd}
\end{equation}
\begin{equation}
\label{bih:bihcomplex1b}
\def\arrowlength{16ex}
\def\arrowdistance{.8}
\begin{tikzcd}[column sep=\arrowlength]
\cdots
\arrow[r, rightarrow, shift left=\arrowdistance, "\cdots"] 
\arrow[r, leftarrow, shift right=\arrowdistance, "\cdots"']
& 
[-5em]
\L{2}{}(\om) 
\ar[r, rightarrow, shift left=\arrowdistance, "\TGradgat"] 
\ar[r, leftarrow, shift right=\arrowdistance, "-\bDivTgan"']
&
[-1em]
\L{2}{\T}(\om) 
\ar[r, rightarrow, shift left=\arrowdistance, "\SRotTgat"] 
\ar[r, leftarrow, shift right=\arrowdistance, "\TbRotSgan"']
& 
[-1em]
\L{2}{\S}(\om) 
\arrow[r, rightarrow, shift left=\arrowdistance, "\divDivSgat"] 
\arrow[r, leftarrow, shift right=\arrowdistance, "\SbGradgradgan"']
& 
[-1em]
\L{2}{}(\om) 
\arrow[r, rightarrow, shift left=\arrowdistance, "\cdots"] 
\arrow[r, leftarrow, shift right=\arrowdistance, "\cdots"']
&
[-5em]
\cdots.
\end{tikzcd}
\end{equation}
The long primal and dual biharmonic Hilbert complexes,
cf.~\cite[(12)]{PS2021b}, read
\begin{equation}
\label{bih:bihcomplex2a}
\footnotesize
\def\arrowlength{16ex}
\def\arrowdistance{.8}
\begin{tikzcd}[column sep=\arrowlength]
\Pone_{\gat}
\arrow[r, rightarrow, shift left=\arrowdistance, "\iota_{\Pone_{\gat}}"] 
\arrow[r, leftarrow, shift right=\arrowdistance, "\pi_{\Pone_{\gat}}"']
& 
[-2em]
\L{2}{}(\om) 
\ar[r, rightarrow, shift left=\arrowdistance, "\SGradgradgat"] 
\ar[r, leftarrow, shift right=\arrowdistance, "\bdivDivSgan"']
&
[-1em]
\L{2}{\S}(\om) 
\ar[r, rightarrow, shift left=\arrowdistance, "\TRotSgat"] 
\ar[r, leftarrow, shift right=\arrowdistance, "\SbRotTgan"']
& 
[0em]
\L{2}{\T}(\om) 
\arrow[r, rightarrow, shift left=\arrowdistance, "\DivTgat"] 
\arrow[r, leftarrow, shift right=\arrowdistance, "-\TbGradgan"']
& 
[-1em]
\L{2}{}(\om) 
\arrow[r, rightarrow, shift left=\arrowdistance, "\pi_{\RT_{\gan}}"] 
\arrow[r, leftarrow, shift right=\arrowdistance, "\iota_{\RT_{\gan}}"']
&
[-2em]
\RT_{\gan}
\end{tikzcd}
\end{equation}
\begin{equation}
\label{bih:bihcomplex2b}
\footnotesize
\def\arrowlength{16ex}
\def\arrowdistance{.8}
\begin{tikzcd}[column sep=\arrowlength]
\RT_{\gat}
\arrow[r, rightarrow, shift left=\arrowdistance, "\iota_{\RT_{\gat}}"] 
\arrow[r, leftarrow, shift right=\arrowdistance, "\pi_{\RT_{\gat}}"']
& 
[-2em]
\L{2}{}(\om) 
\ar[r, rightarrow, shift left=\arrowdistance, "\TGradgat"] 
\ar[r, leftarrow, shift right=\arrowdistance, "-\bDivTgan"']
&
[-1em]
\L{2}{\T}(\om) 
\ar[r, rightarrow, shift left=\arrowdistance, "\SRotTgat"] 
\ar[r, leftarrow, shift right=\arrowdistance, "\TbRotSgan"']
& 
[0em]
\L{2}{\S}(\om) 
\arrow[r, rightarrow, shift left=\arrowdistance, "\divDivSgat"] 
\arrow[r, leftarrow, shift right=\arrowdistance, "\SbGradgradgan"']
& 
[-1em]
\L{2}{}(\om) 
\arrow[r, rightarrow, shift left=\arrowdistance, "\pi_{\Pone_{\gan}}"] 
\arrow[r, leftarrow, shift right=\arrowdistance, "\iota_{\Pone_{\gan}}"']
&
[-2em]
\Pone_{\gan}
\end{tikzcd}
\end{equation}
with the additional complex properties 
\begin{align*}
R(\iota_{\RT_{\gat}})&=N(\TGradgat)=\RT_{\gat},
&
\ol{R(\DivTgan)}&=\RT_{\gat}^{\bot_{\L{2}{}(\om)}},\\
R(\iota_{\Pone_{\gat}})&=N(\SGradgradgat)=\Pone_{\gat},
&
\ol{R(\divDivSgan)}&=(\Pone_{\gat})^{\bot_{\L{2}{}(\om)}},
\end{align*}
where
\begin{align*}
\RT_{\Sigma}
=\begin{cases}
\{0\}&\text{if }\Sigma\neq\emptyset,\\
\RT&\text{if }\Sigma=\emptyset,
\end{cases}\qquad\text{with}\qquad
\RT:=\big\{\reals^{3}\ni x\mapsto ax+q:a\in\reals,\,q\in\reals^{3}\big\},\\
\Pone_{\Sigma}
=\begin{cases}
\{0\}&\text{if }\Sigma\neq\emptyset,\\
\Pone&\text{if }\Sigma=\emptyset,
\end{cases}\qquad\text{with}\qquad
\Pone:=\big\{\reals^{3}\ni x\mapsto q\cdot x+a:a\in\reals,\,q\in\reals^{3}\big\}
\end{align*}
denote the global Raviart-Thomas fields 
and the global polynomials of degree less or equal to $1$ in $\om$, respectively.
We have $\dim\RT=\dim\Pone=4$.
Note that, e.g., by Lemma \ref{bih:lem:weakeqstrongbih} (i) it holds
$$N(\SGradgradgat)
=\big\{u\in\H{2}{\gat}(\om):\Gradgrad u=0\big\}.$$

More generally, in addition to \eqref{bih:bihcomplex2a} and \eqref{bih:bihcomplex2b},
we shall discuss for $k\in\nat_{0}$ the higher Sobolev order 
(long primal and formally dual) biharmonic Hilbert complexes
(omitting $\om$ in the notation)

\begin{equation*}
\footnotesize
\def\arrowlength{16ex}
\def\arrowdistance{0}
\begin{tikzcd}[column sep=\arrowlength]
\Pone_{\gat}
\arrow[r, rightarrow, shift left=\arrowdistance, "\iota_{\Pone_{\gat}}"] 
& 
[-3em]
\H{k}{\gat}
\ar[r, rightarrow, shift left=\arrowdistance, "\SGradgradgat^{k}"] 
&
[-1em]
\H{k}{\S,\gat}
\ar[r, rightarrow, shift left=\arrowdistance, "\TRotSgat^{k}"] 
& 
[-1em]
\H{k}{\T,\gat}
\arrow[r, rightarrow, shift left=\arrowdistance, "\DivTgat^{k}"] 
& 
[-3em]
\H{k}{\gat}
\arrow[r, rightarrow, shift left=\arrowdistance, "\pi_{\RT_{\gan}}"] 
&
[-3em]
\RT_{\gan},
\end{tikzcd}
\end{equation*}
\vspace*{-4mm}
\begin{equation*}
\footnotesize
\def\arrowlength{16ex}
\def\arrowdistance{0}
\begin{tikzcd}[column sep=\arrowlength]
\Pone_{\gat}
\arrow[r, leftarrow, shift right=\arrowdistance, "\pi_{\Pone_{\gat}}"]
& 
[-3em]
\H{k}{\gan}
\ar[r, leftarrow, shift right=\arrowdistance, "\bdivDivSgan^{k}"]
&
[-2em]
\H{k}{\S,\gan}
\ar[r, leftarrow, shift right=\arrowdistance, "\SbRotTgan^{k}"]
& 
[0em]
\H{k}{\T,\gan}
\arrow[r, leftarrow, shift right=\arrowdistance, "-\TbGradgan^{k}"]
& 
[0em]
\H{k}{\gan}
\arrow[r, leftarrow, shift right=\arrowdistance, "\iota_{\RT_{\gan}}"]
&
[-3em]
\RT_{\gan}
\end{tikzcd}
\end{equation*}
and
\begin{equation*}
\footnotesize
\def\arrowlength{16ex}
\def\arrowdistance{0}
\begin{tikzcd}[column sep=\arrowlength]
\RT_{\gat}
\arrow[r, rightarrow, shift left=\arrowdistance, "\iota_{\RT_{\gat}}"] 
& 
[-3em]
\H{k}{\gat}
\ar[r, rightarrow, shift left=\arrowdistance, "\TGradgat^{k}"] 
&
[-1em]
\H{k}{\T,\gat}
\ar[r, rightarrow, shift left=\arrowdistance, "\SRotTgat^{k}"] 
& 
[-1em]
\H{k}{\S,\gat}
\arrow[r, rightarrow, shift left=\arrowdistance, "\divDivSgat^{k}"] 
& 
[-3em]
\H{k}{\gat}
\arrow[r, rightarrow, shift left=\arrowdistance, "\pi_{\Pone_{\gan}}"] 
&
[-3em]
\Pone_{\gan},
\end{tikzcd}
\end{equation*}
\vspace*{-4mm}
\begin{equation*}
\footnotesize
\def\arrowlength{16ex}
\def\arrowdistance{0}
\begin{tikzcd}[column sep=\arrowlength]
\RT_{\gat}
\arrow[r, leftarrow, shift right=\arrowdistance, "\pi_{\RT_{\gat}}"]
& 
[-3em]
\H{k}{\gan}
\ar[r, leftarrow, shift right=\arrowdistance, "-\bDivTgan^{k}"]
&
[-2em]
\H{k}{\T,\gan}
\ar[r, leftarrow, shift right=\arrowdistance, "\TbRotSgan^{k}"]
& 
[0em]
\H{k}{\S,\gan}
\arrow[r, leftarrow, shift right=\arrowdistance, "\SbGradgradgan^{k}"]
& 
[0em]
\H{k}{\gan}
\arrow[r, leftarrow, shift right=\arrowdistance, "\iota_{\Pone_{\gan}}"]
&
[-3em]
\Pone_{\gan}
\end{tikzcd}
\end{equation*}
with associated domain complexes
\begin{equation*}
\footnotesize
\def\arrowlength{16ex}
\def\arrowdistance{0}
\begin{tikzcd}[column sep=\arrowlength]
\Pone_{\gat}
\arrow[r, rightarrow, shift left=\arrowdistance, "\iota_{\Pone_{\gat}}"] 
& 
[-4em]
\H{k}{\gat}(\Gradgrad)
\ar[r, rightarrow, shift left=\arrowdistance, "\SGradgradgat^{k}"] 
&
[-1em]
\H{k}{\S,\gat}(\Rot)
\ar[r, rightarrow, shift left=\arrowdistance, "\TRotSgat^{k}"] 
& 
[-1em]
\H{k}{\T,\gat}(\Div)
\arrow[r, rightarrow, shift left=\arrowdistance, "\DivTgat^{k}"] 
& 
[-3em]
\H{k}{\gat}
\arrow[r, rightarrow, shift left=\arrowdistance, "\pi_{\RT_{\gan}}"] 
&
[-3em]
\RT_{\gan},
\end{tikzcd}
\end{equation*}
\vspace*{-4mm}
\begin{equation*}
\footnotesize
\def\arrowlength{16ex}
\def\arrowdistance{0}
\begin{tikzcd}[column sep=\arrowlength]
\Pone_{\gat}
\arrow[r, leftarrow, shift right=\arrowdistance, "\pi_{\Pone_{\gat}}"]
& 
[-4em]
\H{k}{\gan}
\ar[r, leftarrow, shift right=\arrowdistance, "\bdivDivSgan^{k}"]
&
[-2em]
\bH{k}{\S,\gan}(\divDiv)
\ar[r, leftarrow, shift right=\arrowdistance, "\SbRotTgan^{k}"]
& 
[-2em]
\bH{k}{\T,\gan}(\symRot)
\arrow[r, leftarrow, shift right=\arrowdistance, "-\TbGradgan^{k}"]
& 
[-2em]
\bH{k}{\gan}(\devGrad)
\arrow[r, leftarrow, shift right=\arrowdistance, "\iota_{\RT_{\gan}}"]
&
[-3em]
\RT_{\gan}
\end{tikzcd}
\end{equation*}
and
\begin{equation*}
\footnotesize
\def\arrowlength{16ex}
\def\arrowdistance{0}
\begin{tikzcd}[column sep=\arrowlength]
\RT_{\gat}
\arrow[r, rightarrow, shift left=\arrowdistance, "\iota_{\RT_{\gat}}"] 
& 
[-3em]
\H{k}{\gat}(\devGrad)
\ar[r, rightarrow, shift left=\arrowdistance, "\TGradgat^{k}"] 
&
[-2em]
\H{k}{\T,\gat}(\symRot)
\ar[r, rightarrow, shift left=\arrowdistance, "\SRotTgat^{k}"] 
& 
[-2em]
\H{k}{\S,\gat}(\divDiv)
\arrow[r, rightarrow, shift left=\arrowdistance, "\divDivSgat^{k}"] 
& 
[-2em]
\H{k}{\gat}
\arrow[r, rightarrow, shift left=\arrowdistance, "\pi_{\Pone_{\gan}}"] 
&
[-4em]
\Pone_{\gan},
\end{tikzcd}
\end{equation*}
\vspace*{-4mm}
\begin{equation*}
\footnotesize
\def\arrowlength{16ex}
\def\arrowdistance{0}
\begin{tikzcd}[column sep=\arrowlength]
\RT_{\gat}
\arrow[r, leftarrow, shift right=\arrowdistance, "\pi_{\RT_{\gat}}"]
& 
[-3em]
\H{k}{\gan}
\ar[r, leftarrow, shift right=\arrowdistance, "-\bDivTgan^{k}"]
&
[-2em]
\bH{k}{\T,\gan}(\Div)
\ar[r, leftarrow, shift right=\arrowdistance, "\TbRotSgan^{k}"]
& 
[-2em]
\bH{k}{\S,\gan}(\Rot)
\arrow[r, leftarrow, shift right=\arrowdistance, "\SbGradgradgan^{k}"]
& 
[0em]
\bH{k}{\gan}(\Gradgrad)
\arrow[r, leftarrow, shift right=\arrowdistance, "\iota_{\Pone_{\gan}}"]
&
[-4em]
\Pone_{\gan}.
\end{tikzcd}
\end{equation*}
Additionally, for $k\geq1$ we will also discuss the following variants of the biharmonic complexes 
\begin{equation*}
\footnotesize
\def\arrowlength{16ex}
\def\arrowdistance{0}
\begin{tikzcd}[column sep=\arrowlength]
\Pone_{\gat}
\arrow[r, rightarrow, shift left=\arrowdistance, "\iota_{\Pone_{\gat}}"] 
& 
[-3em]
\H{k}{\gat}
\ar[r, rightarrow, shift left=\arrowdistance, "\SGradgradgat^{k,k-1}"] 
&
[-1em]
\H{k-1}{\S,\gat}
\ar[r, rightarrow, shift left=\arrowdistance, "\TRotSgat^{k-1}"] 
& 
[-1em]
\H{k-1}{\T,\gat}
\arrow[r, rightarrow, shift left=\arrowdistance, "\DivTgat^{k-1}"] 
& 
[-3em]
\H{k-1}{\gat}
\arrow[r, rightarrow, shift left=\arrowdistance, "\pi_{\RT_{\gan}}"] 
&
[-3em]
\RT_{\gan},
\end{tikzcd}
\end{equation*}
\vspace*{-4mm}
\begin{equation*}
\footnotesize
\def\arrowlength{16ex}
\def\arrowdistance{0}
\begin{tikzcd}[column sep=\arrowlength]
\Pone_{\gat}
\arrow[r, leftarrow, shift right=\arrowdistance, "\pi_{\Pone_{\gat}}"]
& 
[-3em]
\H{k-1}{\gan}
\ar[r, leftarrow, shift right=\arrowdistance, "\bdivDivSgan^{k,k-1}"]
&
[-2em]
\H{k}{\S,\gan}
\ar[r, leftarrow, shift right=\arrowdistance, "\SbRotTgan^{k}"]
& 
[0em]
\H{k}{\T,\gan}
\arrow[r, leftarrow, shift right=\arrowdistance, "-\TbGradgan^{k}"]
& 
[0em]
\H{k}{\gan}
\arrow[r, leftarrow, shift right=\arrowdistance, "\iota_{\RT_{\gan}}"]
&
[-3em]
\RT_{\gan}
\end{tikzcd}
\end{equation*}
and
\begin{equation*}
\footnotesize
\def\arrowlength{16ex}
\def\arrowdistance{0}
\begin{tikzcd}[column sep=\arrowlength]
\RT_{\gat}
\arrow[r, rightarrow, shift left=\arrowdistance, "\iota_{\RT_{\gat}}"] 
& 
[-3em]
\H{k}{\gat}
\ar[r, rightarrow, shift left=\arrowdistance, "\TGradgat^{k}"] 
&
[-1em]
\H{k}{\T,\gat}
\ar[r, rightarrow, shift left=\arrowdistance, "\SRotTgat^{k}"] 
& 
[-1em]
\H{k}{\S,\gat}
\arrow[r, rightarrow, shift left=\arrowdistance, "\divDivSgat^{k,k-1}"] 
& 
[-3em]
\H{k-1}{\gat}
\arrow[r, rightarrow, shift left=\arrowdistance, "\pi_{\Pone_{\gan}}"] 
&
[-3em]
\Pone_{\gan},
\end{tikzcd}
\end{equation*}
\vspace*{-4mm}
\begin{equation*}
\footnotesize
\def\arrowlength{16ex}
\def\arrowdistance{0}
\begin{tikzcd}[column sep=\arrowlength]
\RT_{\gat}
\arrow[r, leftarrow, shift right=\arrowdistance, "\pi_{\RT_{\gat}}"]
& 
[-3em]
\H{k-1}{\gan}
\ar[r, leftarrow, shift right=\arrowdistance, "-\bDivTgan^{k-1}"]
&
[-2em]
\H{k-1}{\T,\gan}
\ar[r, leftarrow, shift right=\arrowdistance, "\TbRotSgan^{k-1}"]
& 
[0em]
\H{k-1}{\S,\gan}
\arrow[r, leftarrow, shift right=\arrowdistance, "\SbGradgradgan^{k,k-1}"]
& 
[0em]
\H{k}{\gan}
\arrow[r, leftarrow, shift right=\arrowdistance, "\iota_{\Pone_{\gan}}"]
&
[-3em]
\Pone_{\gan}
\end{tikzcd}
\end{equation*}
with associated domain complexes
\begin{equation*}
\footnotesize
\def\arrowlength{16ex}
\def\arrowdistance{0}
\begin{tikzcd}[column sep=\arrowlength]
\Pone_{\gat}
\arrow[r, rightarrow, shift left=\arrowdistance, "\iota_{\Pone_{\gat}}"] 
& 
[-4em]
\H{k,k-1}{\gat}(\Gradgrad)
\ar[r, rightarrow, shift left=\arrowdistance, "\SGradgradgat^{k,k-1}"] 
&
[-1em]
\H{k-1}{\S,\gat}(\Rot)
\ar[r, rightarrow, shift left=\arrowdistance, "\TRotSgat^{k-1}"] 
& 
[-2em]
\H{k-1}{\T,\gat}(\Div)
\arrow[r, rightarrow, shift left=\arrowdistance, "\DivTgat^{k-1}"] 
& 
[-3em]
\H{k-1}{\gat}
\arrow[r, rightarrow, shift left=\arrowdistance, "\pi_{\RT_{\gan}}"] 
&
[-3em]
\RT_{\gan},
\end{tikzcd}
\end{equation*}
\vspace*{-4mm}
\begin{equation*}
\footnotesize
\def\arrowlength{16ex}
\def\arrowdistance{0}
\begin{tikzcd}[column sep=\arrowlength]
\Pone_{\gat}
\arrow[r, leftarrow, shift right=\arrowdistance, "\pi_{\Pone_{\gat}}"]
& 
[-4em]
\H{k-1}{\gan}
\ar[r, leftarrow, shift right=\arrowdistance, "\bdivDivSgan^{k,k-1}"]
&
[-2em]
\bH{k,k-1}{\S,\gan}(\divDiv)
\ar[r, leftarrow, shift right=\arrowdistance, "\SbRotTgan^{k}"]
& 
[-2em]
\bH{k}{\T,\gan}(\symRot)
\arrow[r, leftarrow, shift right=\arrowdistance, "-\TbGradgan^{k}"]
& 
[-2em]
\bH{k}{\gan}(\devGrad)
\arrow[r, leftarrow, shift right=\arrowdistance, "\iota_{\RT_{\gan}}"]
&
[-4em]
\RT_{\gan}
\end{tikzcd}
\end{equation*}
and
\begin{equation*}
\footnotesize
\def\arrowlength{16ex}
\def\arrowdistance{0}
\begin{tikzcd}[column sep=\arrowlength]
\RT_{\gat}
\arrow[r, rightarrow, shift left=\arrowdistance, "\iota_{\RT_{\gat}}"] 
& 
[-4em]
\H{k}{\gat}(\devGrad)
\ar[r, rightarrow, shift left=\arrowdistance, "\TGradgat^{k}"] 
&
[-2em]
\H{k}{\T,\gat}(\symRot)
\ar[r, rightarrow, shift left=\arrowdistance, "\SRotTgat^{k}"] 
& 
[-2em]
\H{k,k-1}{\S,\gat}(\divDiv)
\arrow[r, rightarrow, shift left=\arrowdistance, "\divDivSgat^{k,k-1}"] 
& 
[-2em]
\H{k-1}{\gat}
\arrow[r, rightarrow, shift left=\arrowdistance, "\pi_{\Pone_{\gan}}"] 
&
[-4em]
\Pone_{\gan},
\end{tikzcd}
\end{equation*}
\vspace*{-4mm}
\begin{equation*}
\footnotesize
\def\arrowlength{16ex}
\def\arrowdistance{0}
\begin{tikzcd}[column sep=\arrowlength]
\RT_{\gat}
\arrow[r, leftarrow, shift right=\arrowdistance, "\pi_{\RT_{\gat}}"]
& 
[-3em]
\H{k-1}{\gan}
\ar[r, leftarrow, shift right=\arrowdistance, "-\bDivTgan^{k-1}"]
&
[-2em]
\bH{k-1}{\T,\gan}(\Div)
\ar[r, leftarrow, shift right=\arrowdistance, "\TbRotSgan^{k-1}"]
& 
[-2em]
\bH{k-1}{\S,\gan}(\Rot)
\arrow[r, leftarrow, shift right=\arrowdistance, "\SbGradgradgan^{k,k-1}"]
& 
[0em]
\bH{k}{\gan}(\Gradgrad)
\arrow[r, leftarrow, shift right=\arrowdistance, "\iota_{\Pone_{\gan}}"]
&
[-4em]
\Pone_{\gan}.
\end{tikzcd}
\end{equation*}

Here we have introduced the densely defined and closed linear operators
\begin{align*}
\SGradgradgat^{k}:D(\SGradgradgat^{k})\subset\H{k}{\gat}(\om)&\to\H{k}{\S,\gat}(\om);
&
u&\mapsto\Gradgrad u,\\
\bdivDivSgan^{k}:D(\bdivDivSgan^{k})\subset\H{k}{\S,\gan}(\om)&\to\H{k}{\gan}(\om);
&
S&\mapsto\divDiv S,\\
\TRotSgat^{k}:D(\TRotSgat^{k})\subset\H{k}{\S,\gat}(\om)&\to\H{k}{\T,\gat}(\om);
&
S&\mapsto\Rot S,\\
\SbRotTgan^{k}:D(\SbRotTgan^{k})\subset\H{k}{\T,\gan}(\om)&\to\H{k}{\S,\gan}(\om);
&
T&\mapsto\symRot T,\\
\DivTgat^{k}:D(\DivTgat^{k})\subset\H{k}{\T,\gat}(\om)&\to\H{k}{\gat}(\om);
&
T&\mapsto\Div T,\\
\TbGradgan^{k}:D(\TbGradgan^{k})\subset\H{k}{\gan}(\om)&\to\H{k}{\T,\gan}(\om);
&
v&\mapsto\devGrad v,\\
\TGradgat^{k}:D(\TGradgat^{k})\subset\H{k}{\gat}(\om)&\to\H{k}{\T,\gat}(\om);
&
v&\mapsto\devGrad v,\\
\bDivTgan^{k}:D(\bDivTgan^{k})\subset\H{k}{\T,\gan}(\om)&\to\H{k}{\gan}(\om);
&
T&\mapsto\Div T,\\
\SRotTgat^{k}:D(\SRotTgat^{k})\subset\H{k}{\T,\gat}(\om)&\to\H{k}{\S,\gat}(\om);
&
T&\mapsto\symRot T,\\
\TbRotSgan^{k}:D(\TbRotSgan^{k})\subset\H{k}{\S,\gan}(\om)&\to\H{k}{\T,\gan}(\om);
&
S&\mapsto\Rot S,\\
\divDivSgat^{k}:D(\divDivSgat^{k})\subset\H{k}{\S,\gat}(\om)&\to\H{k}{\gat}(\om);
&
S&\mapsto\divDiv S,\\
\SbGradgradgan^{k}:D(\SbGradgradgan^{k})\subset\H{k}{\gan}(\om)&\to\H{k}{\S,\gan}(\om);
&
u&\mapsto\Gradgrad u,
\end{align*}
with domains of definition
\begin{align*}
D(\SGradgradgat^{k})&=\H{k}{\gat}(\Gradgrad,\om),
&
D(\bdivDivSgan^{k})&=\bH{k}{\S,\gan}(\divDiv,\om),\\
D(\TRotSgat^{k})&=\H{k}{\S,\gat}(\Rot,\om),
&
D(\SbRotTgan^{k})&=\bH{k}{\T,\gan}(\symRot,\om),\\
D(\DivTgat^{k})&=\H{k}{\T,\gat}(\Div,\om),
&
D(\TbGradgan^{k})&=\bH{k}{\gan}(\devGrad,\om),\\
D(\TGradgat^{k})&=\H{k}{\gat}(\devGrad,\om),
&
D(\bDivTgan^{k})&=\bH{k}{\T,\gan}(\Div,\om),\\
D(\SRotTgat^{k})&=\H{k}{\T,\gat}(\symRot,\om),
&
D(\TbRotSgan^{k})&=\bH{k}{\S,\gan}(\Rot,\om),\\
D(\divDivSgat^{k})&=\H{k}{\S,\gat}(\divDiv,\om),
&
D(\SbGradgradgan^{k})&=\bH{k}{\gan}(\Gradgrad,\om).
\end{align*}
Moreover,
\begin{align*}
\SGradgradgat^{k,k-1}:D(\SGradgradgat^{k,k-1})\subset\H{k}{\gat}(\om)&\to\H{k-1}{\S,\gat}(\om);
&
u&\mapsto\Gradgrad u,\\
\SbGradgradgan^{k,k-1}:D(\SbGradgradgan^{k,k-1})\subset\H{k}{\gan}(\om)&\to\H{k-1}{\S,\gan}(\om);
&
u&\mapsto\Gradgrad u,\\
\divDivSgat^{k,k-1}:D(\divDivSgat^{k,k-1})\subset\H{k}{\S,\gat}(\om)&\to\H{k-1}{\gat}(\om);
&
S&\mapsto\divDiv S,\\
\bdivDivSgan^{k,k-1}:D(\bdivDivSgan^{k,k-1})\subset\H{k}{\S,\gan}(\om)&\to\H{k-1}{\gan}(\om);
&
S&\mapsto\divDiv S,
\end{align*}
with domains of definition
\begin{align*}
D(\SGradgradgat^{k,k-1})&=\H{k,k-1}{\gat}(\Gradgrad,\om),
&
D(\divDivSgat^{k,k-1})&=\H{k,k-1}{\S,\gat}(\divDiv,\om),\\
D(\SbGradgradgan^{k,k-1})&=\bH{k,k-1}{\gan}(\Gradgrad,\om),
&
D(\bdivDivSgan^{k,k-1})&=\bH{k,k-1}{\S,\gan}(\divDiv,\om).
\end{align*}

\subsection{Dirichlet/Neumann Fields}
\label{bih:sec:dirneu}%

We also introduce the cohomology spaces of 
biharmonic Dirichlet/Neumann tensor fields (generalised harmonic tensors)
\begin{align*}
\Harm{}{\S,\gat,\gan,\eps}(\om)
&:=N(\TRotSgat)\cap N(\divDivSgan\eps)
=\H{}{\S,\gat,0}(\Rot,\om)\cap\eps^{-1}\H{}{\S,\gan,0}(\divDiv,\om),\\
\Harm{}{\T,\gan,\gat,\mu}(\om)
&:=N(\SRotTgan)\cap N(\DivTgat\mu)
=\H{}{\T,\gan,0}(\symRot,\om)\cap\mu^{-1}\H{}{\T,\gat,0}(\Div,\om).
\end{align*}
Here, $\eps:\L{2}{\S}(\om)\to\L{2}{\S}(\om)$
is a symmetric and positive 
topological isomorphism (symmetric and positive bijective bounded linear operator),
which introduces a new inner product
$$\scp{\,\cdot\,}{\,\cdot\,}_{\L{2}{\S,\eps}(\om)}
:=\scp{\eps\,\cdot\,}{\,\cdot\,}_{\L{2}{\S}(\om)},$$
where $\L{2}{\S,\eps}(\om):=\L{2}{\S}(\om)$ (as linear space)
equipped with the inner product $\scp{\,\cdot\,}{\,\cdot\,}_{\L{2}{\S,\eps}(\om)}$.
Such \emph{weights} $\eps$ and also $\mu:\L{2}{\T}(\om)\to\L{2}{\T}(\om)$ 
are called \emph{admissible}.
Typical examples are given by
symmetric, $\L{\infty}{}$-bounded, and uniformly positive definite tensor fields
$\eps,\mu:\om\to\reals^{(3\times3)\times(3\times3)}$
with appropriate algebraic properties.

\section{Biharmonic Complexes II}
\label{bih:sec:bih2}%

\subsection{Regular Potentials and Decompositions I}
\label{bih:sec:regpotdeco1}%

\subsubsection{Extendable Domains}
\label{bih:sec:regpotdecoextdom}%

The next theorem is a crucial result.
Its proof is based on \cite[Theorem 3.10]{PZ2020a}, where the stated results 
for $\gat=\ga$ and $\gat=\emptyset$ have been shown,
and the arguments used in, e.g., \cite[Lemma 4.4]{PS2021b}
for partial boundary conditions.
See Appendix \ref{sec:someproof} for a detailed proof.

\begin{theo}[regular potential operators for extendable domains]
\label{bih:highorderregpotextdombih}
Let $(\om,\gat)$ be an extendable bounded strong Lipschitz pair
and let $k\geq0$. Then there exist bounded linear regular potential operators
\begin{align*}
\PotP_{\SGradgrad,\gat}^{k}:
\bH{k}{\S,\gat,0}(\Rot,\om)
&\To\H{k+2}{\gat}(\om)\cap\H{k+2}{}(\rt),\\
\PotP_{\TRotS,\gat}^{k}:
\bH{k}{\T,\gat,0}(\Div,\om)
&\To\H{k+1}{\S,\gat}(\om)\cap\H{k+1}{}(\rt),\\
\PotP_{\DivT,\gat}^{k}:
\H{k}{\gat}(\om)\cap(\RT_{\gan})^{\bot_{\L{2}{}(\om)}}
&\To\H{k+1}{\T,\gat}(\om)\cap\H{k+1}{}(\rt),\\
\PotP_{\TGrad,\gat}^{k}:
\bH{k}{\T,\gat,0}(\symRot,\om)
&\To\H{k+1}{\gat}(\om)\cap\H{k+1}{}(\rt),\\
\PotP_{\SRotT,\gat}^{k}:
\bH{k}{\S,\gat,0}(\divDiv,\om)
&\To\H{k+1}{\T,\gat}(\om)\cap\H{k+1}{}(\rt),\\
\PotP_{\divDivS,\gat}^{k}:
\H{k}{\gat}(\om)\cap(\Pone_{\gan})^{\bot_{\L{2}{}(\om)}}
&\To\H{k+2}{\S,\gat}(\om)\cap\H{k+2}{}(\rt).
\end{align*}
In particular, $\PotP_{\dots}^{\dots}$ are right inverses for 
$\SGradgrad$, $\TRotS$, $\DivT$, $\TGrad$, $\SRotT$, and $\divDivS$, respectively, i.e.,
\begin{align*}
\Gradgrad\PotP_{\SGradgrad,\gat}^{k}
&=\id_{\bH{k}{\S,\gat,0}(\Rot,\om)},
&
\devGrad\PotP_{\TGrad,\gat}^{k}
&=\id_{\bH{k}{\T,\gat,0}(\symRot,\om)},\\
\Rot\PotP_{\TRotS,\gat}^{k}
&=\id_{\bH{k}{\T,\gat,0}(\Div,\om)},
&
\symRot\PotP_{\SRotT,\gat}^{k}
&=\id_{\bH{k}{\S,\gat,0}(\divDiv,\om)},\\
\Div\PotP_{\DivT,\gat}^{k}
&=\id_{\H{k}{\gat}(\om)\cap(\RT_{\gan})^{\bot_{\L{2}{}(\om)}}},
&
\divDiv\PotP_{\divDivS,\gat}^{k}
&=\id_{\H{k}{\gat}(\om)\cap(\Pone_{\gan})^{\bot_{\L{2}{}(\om)}}}.
\end{align*}
Without loss of generality, $\PotP_{\dots}^{\dots}$ 
map to tensor fields with a fixed compact support in $\rt$.
\end{theo}

\begin{rem}
Note that $\A_{n}\PotP_{\A_{n}}=\id_{R(\A_{n})}$ is a general property 
of a (bounded regular) potential operator $\PotP_{\A_{n}}:R(\A_{n})\to\H{+}{n}$
with $\H{+}{n}\subset D(\A_{n})$, cf.~\cite[Section 2.3]{PS2021b}.
\end{rem}

As a simple consequence of the complex properties,
the general results for regular potentials and decompositions 
from, e.g., \cite[Section 2.3]{PS2021b}, and  Theorem \ref{bih:highorderregpotextdombih} 
we obtain a few corollaries.

\begin{cor}[regular potentials for extendable domains]
\label{bih:highorderregpotextdombihcor}
Let $(\om,\gat)$ be an extendable bounded strong Lipschitz pair
and let $k\geq0$. Then the regular potentials representations 
\begin{align*}
\bH{k}{\S,\gat,0}(\Rot,\om)
=\H{k}{\S,\gat,0}(\Rot,\om)
&=\Gradgrad\H{k}{\gat}(\Gradgrad,\om)
=\Gradgrad\H{k+2}{\gat}(\om)\\
&=\Gradgrad\H{k+1,k}{\gat}(\Gradgrad,\om)\\
&=R(\SGradgradgat^{k})
=R(\SGradgradgat^{k+1,k}),\\
\bH{k}{\T,\gat,0}(\Div,\om)
=\H{k}{\T,\gat,0}(\Div,\om)
&=\Rot\H{k}{\S,\gat}(\Rot,\om)
=\Rot\H{k+1}{\S,\gat}(\om)\\
&=R(\TRotSgat^{k}),\\
\H{k}{\gat}(\om)\cap(\RT_{\gan})^{\bot_{\L{2}{}(\om)}}
&=\Div\H{k}{\T,\gat}(\Div,\om)
=\Div\H{k+1}{\T,\gat}(\om)\\
&=R(\DivTgat^{k}),\\
\bH{k}{\T,\gat,0}(\symRot,\om)
=\H{k}{\T,\gat,0}(\symRot,\om)
&=\devGrad\H{k}{\gat}(\devGrad,\om)
=\devGrad\H{k+1}{\gat}(\om)\\
&=R(\TGradgat^{k}),\\
\bH{k}{\S,\gat,0}(\divDiv,\om)
=\H{k}{\S,\gat,0}(\divDiv,\om)
&=\symRot\H{k}{\T,\gat}(\symRot,\om)
=\symRot\H{k+1}{\T,\gat}(\om)\\
&=R(\SRotTgat^{k}),\\
\H{k}{\gat}(\om)\cap(\Pone_{\gan})^{\bot_{\L{2}{}(\om)}}
&=\divDiv\H{k}{\S,\gat}(\divDiv,\om)
=\divDiv\H{k+2}{\S,\gat}(\om)\\
&=\divDiv\H{k+1,k}{\S,\gat}(\divDiv,\om)\\
&=R(\divDivSgat^{k})
=R(\divDivSgat^{k+1,k})
\end{align*}
hold, and the potentials can be chosen such that they depend continuously on the data.
In particular, the latter spaces are closed subspaces of 
$\H{k}{\S}(\om)$, $\H{k}{\T}(\om)$, and $\H{k}{}(\om)$, respectively.
\end{cor}

\begin{cor}[regular decompositions for extendable domains]
\label{bih:highorderregdecoextdombihcor}
Let $(\om,\gat)$ be an extendable bounded strong Lipschitz pair
and let $k\geq0$. Then the bounded regular decompositions 
\begin{align*}
\bH{k}{\S,\gat}(\Rot,\om)
&=\H{k+1}{\S,\gat}(\om)
+\Gradgrad\H{k+2}{\gat}(\om)
=R(\PotP_{\TRotS,\gat}^{k})
\dotplus\H{k}{\S,\gat,0}(\Rot,\om)\\
&=R(\PotP_{\TRotS,\gat}^{k})
\dotplus\Gradgrad\H{k+2}{\gat}(\om)\\
&=R(\PotP_{\TRotS,\gat}^{k})
\dotplus\Gradgrad R(\PotP_{\SGradgrad,\gat}^{k}),\\
\bH{k}{\T,\gat}(\Div,\om)
&=\H{k+1}{\T,\gat}(\om)
+\Rot\H{k+1}{\S,\gat}(\om)
=R(\PotP_{\DivT,\gat}^{k})
\dotplus\H{k}{\T,\gat,0}(\Div,\om)\\
&=R(\PotP_{\DivT,\gat}^{k})
\dotplus\Rot\H{k+1}{\S,\gat}(\om)\\
&=R(\PotP_{\DivT,\gat}^{k})
\dotplus\Rot R(\PotP_{\TRotS,\gat}^{k}),\\
\bH{k}{\T,\gat}(\symRot,\om)
&=\H{k+1}{\T,\gat}(\om)
+\devGrad\H{k+1}{\gat}(\om)
=R(\PotP_{\SRotT,\gat}^{k})
\dotplus\H{k}{\T,\gat,0}(\symRot,\om)\\
&=R(\PotP_{\SRotT,\gat}^{k})
\dotplus\devGrad\H{k+1}{\gat}(\om)\\
&=R(\PotP_{\SRotT,\gat}^{k})
\dotplus\devGrad R(\PotP_{\TGrad,\gat}^{k}),\\
\bH{k}{\S,\gat}(\divDiv,\om)
&=\H{k+2}{\S,\gat}(\om)
+\symRot\H{k+1}{\T,\gat}(\om)
=R(\PotP_{\divDivS,\gat}^{k})
\dotplus\H{k}{\S,\gat,0}(\divDiv,\om)\\
&=R(\PotP_{\divDivS,\gat}^{k})
\dotplus\symRot\H{k+1}{\T,\gat}(\om)\\
&=R(\PotP_{\divDivS,\gat}^{k})
\dotplus\symRot R(\PotP_{\SRotT,\gat}^{k})
\end{align*}
hold with bounded linear regular decomposition operators
\begin{align*}
\PotQ_{\TRotS,\gat}^{k,1}:=\PotP_{\TRotS,\gat}^{k}\Rot:
\bH{k}{\S,\gat}(\Rot,\om)&\to\H{k+1}{\S,\gat}(\om),\\
\PotQ_{\TRotS,\gat}^{k,0}:=\PotP_{\SGradgrad,\gat}^{k}(1-\PotQ_{\TRotS,\gat}^{k,1}):
\bH{k}{\S,\gat}(\Rot,\om)&\to\H{k+2}{\gat}(\om),\\
\PotQ_{\DivT,\gat}^{k,1}:=\PotP_{\DivT,\gat}^{k}\Div:
\bH{k}{\T,\gat}(\Div,\om)&\to\H{k+1}{\T,\gat}(\om),\\
\PotQ_{\DivT,\gat}^{k,0}:=\PotP_{\TRotS,\gat}^{k}(1-\PotQ_{\DivT,\gat}^{k,1}):
\bH{k}{\T,\gat}(\Div,\om)&\to\H{k+1}{\S,\gat}(\om),\\
\PotQ_{\SRotT,\gat}^{k,1}:=\PotP_{\SRotT,\gat}^{k}\symRot:
\bH{k}{\T,\gat}(\symRot,\om)&\to\H{k+1}{\T,\gat}(\om),\\
\PotQ_{\SRotT,\gat}^{k,0}:=\PotP_{\TGrad,\gat}^{k}(1-\PotQ_{\SRotT,\gat}^{k,1}):
\bH{k}{\T,\gat}(\symRot,\om)&\to\H{k+1}{\gat}(\om),\\
\PotQ_{\divDivS,\gat}^{k,1}:=\PotP_{\divDivS,\gat}^{k}\divDiv:
\bH{k}{\S,\gat}(\divDiv,\om)&\to\H{k+2}{\S,\gat}(\om),\\
\PotQ_{\divDivS,\gat}^{k,0}:=\PotP_{\SRotT,\gat}^{k}(1-\PotQ_{\divDivS,\gat}^{k,1}):
\bH{k}{\S,\gat}(\divDiv,\om)&\to\H{k+1}{\T,\gat}(\om)
\end{align*}
satisfying
\begin{align*}
\PotQ_{\TRotS,\gat}^{k,1}
+\Gradgrad\PotQ_{\TRotS,\gat}^{k,0}
&=\id_{\bH{k}{\S,\gat}(\Rot,\om)},\\
\PotQ_{\DivT,\gat}^{k,1}
+\Rot\PotQ_{\DivT,\gat}^{k,0}
&=\id_{\bH{k}{\T,\gat}(\Div,\om)},\\
\PotQ_{\SRotT,\gat}^{k,1}
+\devGrad\PotQ_{\SRotT,\gat}^{k,0}
&=\id_{\bH{k}{\T,\gat}(\symRot,\om)},\\
\PotQ_{\divDivS,\gat}^{k,1}
+\symRot\PotQ_{\divDivS,\gat}^{k,0}
&=\id_{\bH{k}{\S,\gat}(\divDiv,\om)}.
\end{align*}
\end{cor}

\begin{rem}
Note that for (bounded linear) potential operators $\PotP_{\A_{n-1}}$ and $\PotP_{\A_{n}}$ the identity
\begin{align*}
\PotQ_{\A_{n}}^{1}+\A_{n-1}\PotQ_{\A_{n}}^{0}&=\id_{D(\A_{n})}\quad\text{with}
&
\PotQ_{\A_{n}}^{1}:=\PotP_{\A_{n}}\A_{n}:D(\A_{n})&\to\H{+}{n},\\
&&
\PotQ_{\A_{n}}^{0}:=\PotP_{\A_{n-1}}(1-\PotQ_{\A_{n}}^{1}):D(\A_{n})&\to\H{+}{n-1}
\end{align*}
is a general structure of a (bounded) regular decomposition. Moreover:
\begin{itemize}
\item[\bf(i)]
$R(\PotQ_{\A_{n}}^{1})=R(\PotP_{\A_{n}})$ 
and $R(\PotQ_{\A_{n}}^{0})=R(\PotP_{\A_{n-1}})$.
\item[\bf(ii)]
$N(\A_{n})$ is invariant under $\PotQ_{\A_{n}}^{1}$,
as $\A_{n}=\A_{n}\PotQ_{\A_{n}}^{1}$ holds by the complex property.
\item[\bf(iii)]
$\PotQ_{\A_{n}}^{1}$ and $\A_{n-1}\PotQ_{\A_{n}}^{0}=1-\PotQ_{\A_{n}}^{1}$
are projections.
\item[\bf(iv)]
There exists $c>0$ such that for all $x\in D(\A_{n})$
$$\norm{\PotQ_{\A_{n}}^{1}x}_{\H{+}{n}}\leq c\norm{\A_{n}x}_{\H{}{n+1}}.$$
\item[\bf(iv')]
In particular, $\PotQ_{\A_{n}}^{1}|_{N(\A_{n})}=0$.
\end{itemize}
\end{rem}

\begin{cor}[weak and strong partial boundary conditions coincide for extendable domains]
\label{bih:weakstrongextdombihcor}
Let $(\om,\gat)$ be an extendable bounded strong Lipschitz pair
and let $k\geq0$. Then weak and strong boundary conditions coincide, i.e.,
\begin{align*}
\bH{k}{\gat}(\Gradgrad,\om)
&=\H{k}{\gat}(\Gradgrad,\om)
=\H{k+2}{\gat}(\om)
=\bH{k+2}{\gat}(\om),\\
\bH{k}{\S,\gat}(\Rot,\om)
&=\H{k}{\S,\gat}(\Rot,\om),\\
\bH{k}{\T,\gat}(\Div,\om)
&=\H{k}{\T,\gat}(\Div,\om),\\
\bH{k}{\gat}(\devGrad,\om)
&=\H{k}{\gat}(\devGrad,\om)
=\H{k+1}{\gat}(\om)
=\bH{k+1}{\gat}(\om),\\
\bH{k}{\T,\gat}(\symRot,\om)
&=\H{k}{\T,\gat}(\symRot,\om),\\
\bH{k}{\S,\gat}(\divDiv,\om)
&=\H{k}{\S,\gat}(\divDiv,\om).
\end{align*}
\end{cor}

Similar versions of Corollary \ref{bih:highorderregdecoextdombihcor} 
and Corollary \ref{bih:weakstrongextdombihcor}
are available for the non-standard Sobolev spaces
of the form $\H{k,k-1}{\dots}(\cdots,\om)$,
cf.~Section \ref{bih:sec:sobolevmore}.
Note that 
\begin{align}
\label{bih:Gradgradkkmoregformula}
\bH{k,k-1}{\gat}(\Gradgrad,\om)=\H{k+1}{\gat}(\om)
\end{align}
as
$\bH{k,k-1}{\gat}(\Gradgrad,\om)
\subset\bH{k-1}{\gat}(\Gradgrad,\om)
=\H{k+1}{\gat}(\om)
\subset\bH{k,k-1}{\gat}(\Gradgrad,\om)$.

\begin{cor}[Corollary \ref{bih:highorderregdecoextdombihcor} and Corollary \ref{bih:weakstrongextdombihcor}
for non-standard Sobolev spaces]
\label{bih:highorderregdecoextdombihcornonstandard}
Let $(\om,\gat)$ be an extendable bounded strong Lipschitz pair
and let $k\geq1$. Then the bounded regular decompositions 
\begin{align*}
\bH{k,k-1}{\S,\gat}(\divDiv,\om)
&=\H{k+1}{\S,\gat}(\om)
+\symRot\H{k+1}{\T,\gat}(\om)
=R(\PotP_{\divDivS,\gat}^{k-1})
\dotplus\H{k}{\S,\gat,0}(\divDiv,\om)\\
&=R(\PotP_{\divDivS,\gat}^{k-1})
\dotplus\symRot\H{k+1}{\T,\gat}(\om)\\
&=R(\PotP_{\divDivS,\gat}^{k-1})
\dotplus\symRot R(\PotP_{\SRotT,\gat}^{k})
=\H{k,k-1}{\S,\gat}(\divDiv,\om)
\end{align*}
hold with bounded linear regular decomposition operators
\begin{align*}
\PotQ_{\divDivS,\gat}^{k,k-1,1}:=\PotP_{\divDivS,\gat}^{k-1}\divDiv:
\bH{k,k-1}{\S,\gat}(\divDiv,\om)&\to\H{k+1}{\S,\gat}(\om),\\
\PotQ_{\divDivS,\gat}^{k,k-1,0}:=\PotP_{\SRotT,\gat}^{k}(1-\PotQ_{\divDivS,\gat}^{k,k-1,1}):
\bH{k,k-1}{\S,\gat}(\divDiv,\om)&\to\H{k+1}{\T,\gat}(\om)
\end{align*}
satisfying
$\PotQ_{\divDivS,\gat}^{k,k-1,1}
+\symRot\PotQ_{\divDivS,\gat}^{k,k-1,0}
=\id_{\bH{k,k-1}{\S,\gat}(\divDiv,\om)}$.
In particular, weak and strong boundary conditions coincide also for the non-standard Sobolev spaces.
\end{cor}

Recall the Hilbert complexes and cohomology groups from Section \ref{bih:sec:bihcomplexes} and Section \ref{bih:sec:dirneu}.

\begin{theo}[closed and exact Hilbert complexes for extendable domains]
\label{bih:theo:closedhilcom}
Let $(\om,\gat)$ be an extendable bounded strong Lipschitz pair
and let $k\geq0$. Both biharmonic domain complexes 
\begin{equation*}
\footnotesize
\def\arrowlength{16ex}
\def\arrowdistance{0}
\begin{tikzcd}[column sep=\arrowlength]
\Pone_{\gat}
\arrow[r, rightarrow, shift left=\arrowdistance, "\iota_{\Pone_{\gat}}"] 
& 
[-4em]
\H{k+2}{\gat}
\ar[r, rightarrow, shift left=\arrowdistance, "\SGradgradgat^{k}"] 
&
[-1em]
\H{k}{\S,\gat}(\Rot)
\ar[r, rightarrow, shift left=\arrowdistance, "\TRotSgat^{k}"] 
& 
[-1em]
\H{k}{\T,\gat}(\Div)
\arrow[r, rightarrow, shift left=\arrowdistance, "\DivTgat^{k}"] 
& 
[-3em]
\H{k}{\gat}
\arrow[r, rightarrow, shift left=\arrowdistance, "\pi_{\RT_{\gan}}"] 
&
[-3em]
\RT_{\gan},
\end{tikzcd}
\end{equation*}
\vspace*{-4mm}
\begin{equation*}
\footnotesize
\def\arrowlength{16ex}
\def\arrowdistance{0}
\begin{tikzcd}[column sep=\arrowlength]
\Pone_{\gat}
\arrow[r, leftarrow, shift right=\arrowdistance, "\pi_{\Pone_{\gat}}"]
& 
[-4em]
\H{k}{\gan}
\ar[r, leftarrow, shift right=\arrowdistance, "\divDivSgan^{k}"]
&
[-2em]
\H{k}{\S,\gan}(\divDiv)
\ar[r, leftarrow, shift right=\arrowdistance, "\SRotTgan^{k}"]
& 
[-2em]
\H{k}{\T,\gan}(\symRot)
\arrow[r, leftarrow, shift right=\arrowdistance, "-\TGradgan^{k}"]
& 
[-2em]
\H{k+1}{\gan}
\arrow[r, leftarrow, shift right=\arrowdistance, "\iota_{\RT_{\gan}}"]
&
[-3em]
\RT_{\gan}
\end{tikzcd}
\end{equation*}
and, for $k\geq1$,
\begin{equation*}
\footnotesize
\def\arrowlength{16ex}
\def\arrowdistance{0}
\begin{tikzcd}[column sep=\arrowlength]
\Pone_{\gat}
\arrow[r, rightarrow, shift left=\arrowdistance, "\iota_{\Pone_{\gat}}"] 
& 
[-4em]
\H{k+1}{\gat}
\ar[r, rightarrow, shift left=\arrowdistance, "\SGradgradgat^{k,k-1}"] 
&
[-1em]
\H{k-1}{\S,\gat}(\Rot)
\ar[r, rightarrow, shift left=\arrowdistance, "\TRotSgat^{k-1}"] 
& 
[-2em]
\H{k-1}{\T,\gat}(\Div)
\arrow[r, rightarrow, shift left=\arrowdistance, "\DivTgat^{k-1}"] 
& 
[-3em]
\H{k-1}{\gat}
\arrow[r, rightarrow, shift left=\arrowdistance, "\pi_{\RT_{\gan}}"] 
&
[-3em]
\RT_{\gan},
\end{tikzcd}
\end{equation*}
\vspace*{-4mm}
\begin{equation*}
\footnotesize
\def\arrowlength{16ex}
\def\arrowdistance{0}
\begin{tikzcd}[column sep=\arrowlength]
\Pone_{\gat}
\arrow[r, leftarrow, shift right=\arrowdistance, "\pi_{\Pone_{\gat}}"]
& 
[-4em]
\H{k-1}{\gan}
\ar[r, leftarrow, shift right=\arrowdistance, "\divDivSgan^{k,k-1}"]
&
[-2em]
\H{k,k-1}{\S,\gan}(\divDiv)
\ar[r, leftarrow, shift right=\arrowdistance, "\SRotTgan^{k}"]
& 
[-2em]
\H{k}{\T,\gan}(\symRot)
\arrow[r, leftarrow, shift right=\arrowdistance, "-\TGradgan^{k}"]
& 
[-2em]
\H{k+1}{\gan}
\arrow[r, leftarrow, shift right=\arrowdistance, "\iota_{\RT_{\gan}}"]
&
[-4em]
\RT_{\gan}
\end{tikzcd}
\end{equation*}
are exact and closed Hilbert complexes.
In particular, all ranges are closed, 
all cohomology groups (Dirichlet/Neumann fields) are trivial,
and the operators from Theorem \ref{bih:highorderregpotextdombih}
are associated bounded regular potential operators.
\end{theo}

\subsubsection{General Strong Lipschitz Domains}
\label{bih:sec:regpotdecogendom}%

From now on we drop the additional condition ``extendable domain'', thus
$(\om,\gat)$ is a bounded strong Lipschitz pair.

\begin{lem}[cutting lemma]
\label{bih:lem:cutlem}
Let $\varphi\in\C{\infty}{}(\rt)$ and let $k\geq0$.
\begin{itemize}
\item[\bf(i)]
If $S\in\bH{k}{\S,\gat}(\Rot,\om)$, then
$\varphi S\in\bH{k}{\S,\gat}(\Rot,\om)$ and 
$\Rot(\varphi S)=\varphi\Rot S-S\spn\grad\varphi$.
\item[\bf(ii)]
If $T\in\bH{k}{\T,\gat}(\Div,\om)$, then
$\varphi T\in\bH{k}{\T,\gat}(\Div,\om)$ and 
$\Div(\varphi T)=\varphi\Div T+ T\grad\varphi$.
\item[\bf(iii)]
If $T\in\bH{k}{\T,\gat}(\symRot,\om)$, then
$\varphi T\in\bH{k}{\T,\gat}(\symRot,\om)$ and\\
$\symRot(\varphi T)=\varphi\symRot T-\sym(T\spn\grad\varphi)$.
\item[\bf(iv)]
If $k\geq1$ and $S\in\bH{k,k-1}{\S,\gat}(\divDiv,\om)$, then 
$\varphi S\in\bH{k,k-1}{\S,\gat}(\divDiv,\om)$ and 
\begin{align*}
\divDiv(\varphi S)
=\varphi\divDiv S
+2\grad\varphi\cdot\Div S
+\Grad\grad\varphi:S.
\end{align*}
In particular, this holds for $S\in\bH{k}{\S,\gat}(\divDiv,\om)$.
Note that $\,\cdot\,$ and $\,:\,$ denotes the point-wise scalar product 
for vectors fields and tensor (matrix) fields, respectively.
\end{itemize}
\end{lem}

We proceed by showing crucial regular decompositions for the biharmonic complexes
extending the results of Corollary \ref{bih:highorderregdecoextdombihcor}
and Corollary \ref{bih:highorderregdecoextdombihcornonstandard}
to our general setting.
The proof is based on Corollary \ref{bih:highorderregdecoextdombihcor}
together with a partition of unity.

\begin{lem}[regular decompositions]
\label{bih:lem:highorderregdecobih}
Let $k\geq0$. Then the bounded regular decompositions
\begin{align*}
\bH{k}{\S,\gat}(\Rot,\om)
&=\H{k+1}{\S,\gat}(\om)
+\Gradgrad\H{k+2}{\gat}(\om),\\
\bH{k}{\T,\gat}(\Div,\om)
&=\H{k+1}{\T,\gat}(\om)
+\Rot\H{k+1}{\S,\gat}(\om),\\
\bH{k}{\T,\gat}(\symRot,\om)
&=\H{k+1}{\T,\gat}(\om)
+\devGrad\H{k+1}{\gat}(\om),\\
\bH{k}{\S,\gat}(\divDiv,\om)
&=\H{k+2}{\S,\gat}(\om)
+\symRot\H{k+1}{\T,\gat}(\om)
\intertext{and, for $k\geq1$, the non-standard bounded regular decompositions}
\bH{k}{\S,\gat}(\divDiv,\om)
\subset\bH{k,k-1}{\S,\gat}(\divDiv,\om)
&=\H{k+1}{\S,\gat}(\om)
+\symRot\H{k+1}{\T,\gat}(\om)
\end{align*}
hold with bounded linear regular decomposition operators
\begin{align*}
\PotQ_{\TRotS,\gat}^{k,1}:
\bH{k}{\S,\gat}(\Rot,\om)&\to\H{k+1}{\S,\gat}(\om),
&
\PotQ_{\TRotS,\gat}^{k,0}:
\bH{k}{\S,\gat}(\Rot,\om)&\to\H{k+2}{\gat}(\om),\\
\PotQ_{\DivT,\gat}^{k,1}:
\bH{k}{\T,\gat}(\Div,\om)&\to\H{k+1}{\T,\gat}(\om),
&
\PotQ_{\DivT,\gat}^{k,0}:
\bH{k}{\T,\gat}(\Div,\om)&\to\H{k+1}{\S,\gat}(\om),\\
\PotQ_{\SRotT,\gat}^{k,1}:
\bH{k}{\T,\gat}(\symRot,\om)&\to\H{k+1}{\T,\gat}(\om),
&
\PotQ_{\SRotT,\gat}^{k,0}:
\bH{k}{\T,\gat}(\symRot,\om)&\to\H{k+1}{\gat}(\om),\\
\PotQ_{\divDivS,\gat}^{k,1}:
\bH{k}{\S,\gat}(\divDiv,\om)&\to\H{k+2}{\S,\gat}(\om),
&
\PotQ_{\divDivS,\gat}^{k,0}:
\bH{k}{\S,\gat}(\divDiv,\om)&\to\H{k+1}{\T,\gat}(\om),\\
\PotQ_{\divDivS,\gat}^{k,k-1,1}:
\bH{k,k-1}{\S,\gat}(\divDiv,\om)&\to\H{k+1}{\S,\gat}(\om),
&
\PotQ_{\divDivS,\gat}^{k,k-1,0}:
\bH{k,k-1}{\S,\gat}(\divDiv,\om)&\to\H{k+1}{\T,\gat}(\om)
\end{align*}
satisfying
\begin{align*}
\PotQ_{\TRotS,\gat}^{k,1}
+\Gradgrad\PotQ_{\TRotS,\gat}^{k,0}
&=\id_{\bH{k}{\S,\gat}(\Rot,\om)},\\
\PotQ_{\DivT,\gat}^{k,1}
+\Rot\PotQ_{\DivT,\gat}^{k,0}
&=\id_{\bH{k}{\T,\gat}(\Div,\om)},\\
\PotQ_{\SRotT,\gat}^{k,1}
+\devGrad\PotQ_{\SRotT,\gat}^{k,0}
&=\id_{\bH{k}{\T,\gat}(\symRot,\om)},\\
\PotQ_{\divDivS,\gat}^{k,1}
+\symRot\PotQ_{\divDivS,\gat}^{k,0}
&=\id_{\bH{k}{\S,\gat}(\divDiv,\om)},\\
\PotQ_{\divDivS,\gat}^{k,k-1,1}
+\symRot\PotQ_{\divDivS,\gat}^{k,k-1,0}
&=\id_{\bH{k,k-1}{\S,\gat}(\divDiv,\om)},\qquad
k\geq1.
\end{align*}
It holds 
$\Rot\PotQ_{\TRotS,\gat}^{k,1}=\TbRotSgat^{k}$,
$\Div\PotQ_{\DivT,\gat}^{k,1}=\bDivTgat^{k}$, and
$\symRot\PotQ_{\SRotT,\gat}^{k,1}=\SbRotTgat^{k}$
and thus 
$\bH{k}{\S,\gat,0}(\Rot,\om)$, 
$\bH{k}{\T,\gat,0}(\Div,\om)$, and
$\bH{k}{\T,\gat,0}(\symRot,\om)$
are invariant under 
$\PotQ_{\TRotS,\gat}^{k,1}$,
$\PotQ_{\DivT,\gat}^{k,1}$, and
$\PotQ_{\SRotT,\gat}^{k,1}$, respectively.
Analogously, we have
$\divDiv\PotQ_{\divDivS,\gat}^{k,1}=\bdivDivSgat^{k}$
and $\divDiv\PotQ_{\divDivS,\gat}^{k,k-1,1}=\bdivDivSgat^{k,k-1}$
and thus $\bH{k}{\S,\gat,0}(\divDiv,\om)$ 
is invariant under $\PotQ_{\divDivS,\gat}^{k,1}$
and $\PotQ_{\divDivS,\gat}^{k,k-1,1}$, respectively.
\end{lem}

Corollary \ref{bih:weakstrongextdombihcor} and \eqref{bih:Gradgradkkmoregformula}
are generalised to the following important result.

\begin{cor}[weak and strong partial boundary conditions coincide]
\label{bih:cor:weakstrongbih}
Let $k\geq0$. Weak and strong boundary conditions coincide, i.e.,
\begin{align*}
\bH{k}{\gat}(\Gradgrad,\om)
&=\H{k}{\gat}(\Gradgrad,\om)
=\H{k+2}{\gat}(\om)
=\bH{k+2}{\gat}(\om),\\
\bH{k,k-1}{\gat}(\Gradgrad,\om)
&=\H{k,k-1}{\gat}(\Gradgrad,\om)
=\H{k+1}{\gat}(\om)
=\bH{k+1}{\gat}(\om),\qquad
k\geq1,\\
\bH{k}{\S,\gat}(\Rot,\om)
&=\H{k}{\S,\gat}(\Rot,\om),\\
\bH{k}{\T,\gat}(\Div,\om)
&=\H{k}{\T,\gat}(\Div,\om),\\
\bH{k}{\gat}(\devGrad,\om)
&=\H{k}{\gat}(\devGrad,\om)
=\H{k+1}{\gat}(\om)
=\bH{k+1}{\gat}(\om),\\
\bH{k}{\T,\gat}(\symRot,\om)
&=\H{k}{\T,\gat}(\symRot,\om),\\
\bH{k}{\S,\gat}(\divDiv,\om)
&=\H{k}{\S,\gat}(\divDiv,\om),\\
\bH{k,k-1}{\S,\gat}(\divDiv,\om)
&=\H{k,k-1}{\S,\gat}(\divDiv,\om),\qquad
k\geq1.
\end{align*}
In particular, we have $\SbGradgradgat^{k}=\SGradgradgat^{k}$, $\TbRotSgat^{k}=\TRotSgat^{k}$,
$\bDivTgat^{k}=\DivTgat^{k}$, $\TbGradgat^{k}=\TGradgat^{k}$, $\SbRotTgat^{k}=\SRotTgat^{k}$,
$\bdivDivSgat^{k}=\divDivSgat^{k}$, as well as, for $k\geq1$, 
$\SbGradgradgat^{k,k-1}=\SGradgradgat^{k,k-1}$ and $\bdivDivSgat^{k,k-1}=\divDivSgat^{k,k-1}$.
\end{cor}

For a detailed proof of Lemma \ref{bih:lem:highorderregdecobih} and Corollary \ref{bih:cor:weakstrongbih}
see Appendix \ref{sec:someproof}.

\subsection{Mini FA-ToolBox}
\label{bih:sec:fatb}%

\subsubsection{Zero Order Mini FA-ToolBox}
\label{bih:sec:zofatb}%

Recall Section \ref{bih:sec:dirneu} and let $\eps$, $\mu$ be admissible.
In Section \ref{bih:sec:bihop} (for $\eps=\mu=\id$) we have seen that
the densely defined and closed linear operators
\begin{align*}
\A_{-1}=\iota_{\Pone_{\gat}}:\Pone_{\gat}&\to\L{2}{}(\om);
&
p&\mapsto p,\\
\A_{0}=\SGradgradgat:\H{2}{\gat}(\om)\subset\L{2}{}(\om)&\to\L{2}{\S,\eps}(\om);
&
u&\mapsto\Gradgrad u,\\
\A_{1}=\mu^{-1}\TRotSgat:\H{}{\S,\gat}(\Rot,\om)\subset\L{2}{\S,\eps}(\om)&\to\L{2}{\T,\mu}(\om);
&
S&\mapsto\mu^{-1}\Rot S,\\
\A_{2}=\DivTgat\mu:\mu^{-1}\H{}{\T,\gat}(\Div,\om)\subset\L{2}{\T,\mu}(\om)&\to\L{2}{}(\om);
&
T&\mapsto\Div\mu T,\\
\A_{3}=\iota_{\RT_{\gan}}^{*}:\L{2}{}(\om)&\to\RT_{\gan};
&
q&\mapsto\pi_{\RT_{\gan}}q,\\
\A_{-1}^{*}=\iota_{\Pone_{\gat}}^{*}:\L{2}{}(\om)&\to\Pone_{\gat};
&
p&\mapsto\pi_{\Pone_{\gat}}p,\\
\A_{0}^{*}=\divDivSgan\eps:\eps^{-1}\H{}{\S,\gan}(\divDiv,\om)\subset\L{2}{\S,\eps}(\om)&\to\L{2}{}(\om);
&
S&\mapsto\divDiv\eps S,\\
\A_{1}^{*}=\eps^{-1}\SRotTgan:\H{}{\T,\gan}(\symRot,\om)\subset\L{2}{\T,\mu}(\om)&\to\L{2}{\S,\eps}(\om);
&
T&\mapsto\eps^{-1}\symRot T,\\
\A_{2}^{*}=-\TGradgan:\H{1}{\gan}(\om)\subset\L{2}{}(\om)&\to\L{2}{\T,\mu}(\om);
&
v&\mapsto-\devGrad v,\\
\A_{3}^{*}=\iota_{\RT_{\gan}}:\RT_{\gan}&\to\L{2}{}(\om);
&
q&\mapsto q,
\end{align*}
where we have used Corollary \ref{bih:cor:weakstrongbih},
build the long primal and dual elasticity Hilbert complex
\begin{equation}
\label{bih:bihcomplex5}
\footnotesize
\def\arrowlength{16ex}
\def\arrowdistance{.8}
\begin{tikzcd}[column sep=\arrowlength]
\Pone_{\gat}
\arrow[r, rightarrow, shift left=\arrowdistance, "\A_{-1}=\iota_{\Pone_{\gat}}"] 
\arrow[r, leftarrow, shift right=\arrowdistance, "\A_{-1}^{*}\cong\pi_{\Pone_{\gat}}"']
& 
[-2em]
\L{2}{}(\om) 
\ar[r, rightarrow, shift left=\arrowdistance, "\A_{0}=\SGradgradgat"] 
\ar[r, leftarrow, shift right=\arrowdistance, "\A_{0}^{*}=\divDivSgan\eps"']
&
[-1em]
\L{2}{\S,\eps}(\om) 
\ar[r, rightarrow, shift left=\arrowdistance, "\A_{1}=\mu^{-1}\TRotSgat"] 
\ar[r, leftarrow, shift right=\arrowdistance, "\A_{1}^{*}=\eps^{-1}\SRotTgan"']
& 
[0em]
\L{2}{\T,\mu}(\om) 
\arrow[r, rightarrow, shift left=\arrowdistance, "\A_{2}=\DivTgat\mu"] 
\arrow[r, leftarrow, shift right=\arrowdistance, "\A_{2}^{*}=-\TGradgan"']
& 
[-1em]
\L{2}{}(\om) 
\arrow[r, rightarrow, shift left=\arrowdistance, "\A_{3}\cong\pi_{\RT_{\gan}}"] 
\arrow[r, leftarrow, shift right=\arrowdistance, "\A_{3}^{*}=\iota_{\RT_{\gan}}"']
&
[-2em]
\RT_{\gan},
\end{tikzcd}
\end{equation}
cf. \eqref{bih:bihcomplex2a}. Note that
\begin{align*}
\iota_{\Pone_{\gat}}\A_{-1}^{*}
=\iota_{\Pone_{\gat}}\iota_{\Pone_{\gat}}^{*}
=\pi_{\Pone_{\gat}}:\L{2}{}(\om)&\to\L{2}{}(\om),\\
\iota_{\RT_{\gan}}\A_{3}
=\iota_{\RT_{\gan}}\iota_{\RT_{\gan}}^{*}
=\pi_{\RT_{\gan}}:\L{2}{}(\om)&\to\L{2}{}(\om)
\end{align*}
are the actual projectors onto $\Pone_{\gat}$ and $\RT_{\gan}$, respectively.

\begin{theo}[compact embeddings]
\label{bih:theo:cptembzeroorder}
The embeddings
\begin{align*}
D(\A_{1})\cap D(\A_{0}^{*})
=\H{}{\S,\gat}(\Rot,\om)\cap\eps^{-1}\H{}{\S,\gan}(\divDiv,\om)
&\incl\L{2}{\S,\eps}(\om),\\
D(\A_{2})\cap D(\A_{1}^{*})
=\mu^{-1}\H{}{\T,\gat}(\Div,\om)\cap\H{}{\T,\gan}(\symRot,\om)
&\incl\L{2}{\T,\mu}(\om)
\end{align*}
are compact. Moreover, the compactness does not depend on $\eps$ or $\mu$.
\end{theo}

See Appendix \ref{sec:someproof} for a proof.

\begin{rem}[compact embeddings]
\label{bih:rem:cptembzeroorder}
The embeddings 
$$D(\A_{0})\cap D(\A_{-1}^{*})
=D(\A_{0})
=\H{2}{\gat}(\om)
\incl\L{2}{}(\om),\quad
D(\A_{3})\cap D(\A_{2}^{*})
=D(\A_{2}^{*})
=\H{1}{\gan}(\om)
\incl\L{2}{}(\om)$$
are compact by Rellich's selection theorem.
\end{rem}

\begin{theo}[compact biharmonic complex]
\label{bih:theo:cptembzeroorder2}
The long primal and dual biharmonic Hilbert complex \eqref{bih:bihcomplex5} is compact. 
In particular, the complex is closed.
\end{theo}

Let us recall for the densely defined and closed linear operators
$$\A_{n}:D(\A_{n})\subset\H{}{n}\to\H{}{n+1},\qquad
\A_{n}^{*}:D(\A_{n}^{*})\subset\H{}{n+1}\to\H{}{n}$$
the corresponding reduced operators
\begin{align*}
(\A_{n})_{\bot}:=\A_{n}\big|_{N(\A_{n})^{\bot_{\H{}{n}}}}:
D\big((\A_{n})_{\bot}\big)=D(\A_{n})\cap N(\A_{n})^{\bot_{\H{}{n}}}\subset
N(\A_{n})^{\bot_{\H{}{n}}}&\to R(\A_{n}),\\
(\A_{n}^{*})_{\bot}:=\A_{n}^{*}\big|_{N(\A_{n}^{*})^{\bot_{\H{}{n+1}}}}:
D\big((\A_{n}^{*})_{\bot}\big)=D(\A_{n}^{*})\cap N(\A_{n}^{*})^{\bot_{\H{}{n+1}}}\subset
N(\A_{n}^{*})^{\bot_{\H{}{n+1}}}&\to R(\A_{n}^{*}).
\end{align*}
Note that $R(\A_{n})=R\big((\A_{n})_{\bot}\big)=N(\A_{n}^{*})^{\bot_{\H{}{n+1}}}$
and $R(\A_{n}^{*})=R\big((\A_{n}^{*})_{\bot}\big)=N(\A_{n})^{\bot_{\H{}{n}}}$.
Here, we consider 
\begin{align*}
(\A_{0})_{\bot}&=(\SGradgradgat)_{\bot},
&
(\A_{1})_{\bot}&=(\mu^{-1}\TRotSgat)_{\bot},
&
(\A_{2})_{\bot}&=(\DivTgat\mu)_{\bot},\\
(\A_{0}^{*})_{\bot}&=(\divDivSgan\eps)_{\bot},
&
(\A_{1}^{*})_{\bot}&=(\eps^{-1}\SRotTgan)_{\bot},
&
(\A_{2}^{*})_{\bot}&=-(\TGradgan)_{\bot},
\end{align*}
and
\begin{align*}
(\A_{-1})_{\bot}=(\iota_{\Pone_{\gat}})_{\bot}
=\id_{\Pone_{\gat}}:\Pone_{\gat}&\to\Pone_{\gat},\\
(\A_{-1}^{*})_{\bot}\cong(\pi_{\Pone_{\gat}})_{\bot}
=\pi_{\Pone_{\gat}}\big|_{\Pone_{\gat}}
=\id_{\Pone_{\gat}}:\Pone_{\gat}&\to\Pone_{\gat},\\
(\A_{3})_{\bot}\cong(\pi_{\RT_{\gan}})_{\bot}
=\pi_{\RT_{\gan}}\big|_{\RT_{\gan}}
=\id_{\RT_{\gan}}:\RT_{\gan}&\to\RT_{\gan},\\
(\A_{3}^{*})_{\bot}=(\iota_{\RT_{\gan}})_{\bot}
=\id_{\RT_{\gan}}:\RT_{\gan}&\to\RT_{\gan}.
\end{align*}

\cite[Lemma 2.9]{PS2021b} shows:

\begin{theo}[mini FA-ToolBox]
\label{bih:theo:miniFATzeroorder}
For the zero order biharmonic complex it holds:
\begin{itemize}
\item[\bf(i)]
The ranges $R(\SGradgradgat)$, $R(\mu^{-1}\TRotSgat)$, and $R(\DivTgat\mu)$
are closed.
\item[\bf(i')]
The ranges $R(\divDivSgan\eps)$, $R(\eps^{-1}\SRotTgan)$, and $R(\TGradgan)$
are closed.
\item[\bf(ii)]
The inverse operators $(\SGradgradgat)_{\bot}^{-1}$, $(\mu^{-1}\TRotSgat)_{\bot}^{-1}$, 
and $(\DivTgat\mu)_{\bot}^{-1}$ are compact.
\item[\bf(ii')]
The inverse operators $(\divDivSgan\eps)_{\bot}^{-1}$, $(\eps^{-1}\SRotTgan)_{\bot}^{-1}$, 
and $(\TGradgan)_{\bot}^{-1}$ are compact.
\item[\bf(iii)]
The cohomology groups of generalised Dirichlet/Neumann tensor fields
$\Harm{}{\S,\gat,\gan,\eps}(\om)$ and $\Harm{}{\T,\gan,\gat,\mu}(\om)$
Are finite-dimensional.
Moreover, the dimensions do not depend on $\eps$ or $\mu$.
\item[\bf(iv)]
The orthonormal Helmholtz type decompositions
\begin{align*}
\L{2}{\S,\eps}(\om)
&=R(\SGradgradgat)
\oplus_{\L{2}{\S,\eps}(\om)}N(\divDivSgan\eps)\\
&=N(\mu^{-1}\TRotSgat)
\oplus_{\L{2}{\S,\eps}(\om)}R(\eps^{-1}\SRotTgan)\\
&=R(\SGradgradgat)
\oplus_{\L{2}{\S,\eps}(\om)}\Harm{}{\S,\gat,\gan,\eps}(\om)
\oplus_{\L{2}{\S,\eps}(\om)}R(\eps^{-1}\SRotTgan),\\
\L{2}{\T,\mu}(\om)
&=R(\TGradgan)
\oplus_{\L{2}{\T,\mu}(\om)}N(\DivTgat\mu)\\
&=N(\eps^{-1}\SRotTgan)
\oplus_{\L{2}{\T,\mu}(\om)}R(\mu^{-1}\TRotSgat)\\
&=R(\TGradgan)
\oplus_{\L{2}{\T,\mu}(\om)}\Harm{}{\T,\gan,\gat,\mu}(\om)
\oplus_{\L{2}{\T,\mu}(\om)}R(\mu^{-1}\TRotSgat)
\end{align*}
hold.
\item[\bf(v)]
There exist (optimal) $c_{0},c_{1},c_{2}>0$ such that
the Friedrichs/Poincar\'e type estimates 
\begin{align*}
\forall\,u&\in\H{2}{\gat}(\om)\cap(\Pone_{\gat})^{\bot_{\L{2}{}(\om)}}
&
\norm{u}_{\L{2}{}(\om)}&\leq c_{0}\norm{\Gradgrad u}_{\L{2}{\S,\eps}(\om)},\\
\forall\,S&\in\eps^{-1}\H{}{\S,\gan}(\divDiv,\om)\cap R(\SGradgradgat)
&
\norm{S}_{\L{2}{\S,\eps}(\om)}&\leq c_{0}\norm{\divDiv\eps S}_{\L{2}{}(\om)},\\
\forall\,S&\in\H{}{\S,\gat}(\Rot,\om)\cap R(\eps^{-1}\SRotTgan)
&
\norm{S}_{\L{2}{\S,\eps}(\om)}&\leq c_{1}\norm{\mu^{-1}\Rot S}_{\L{2}{\T,\mu}(\om)},\\
\forall\,T&\in\H{}{\T,\gan}(\symRot,\om)\cap R(\mu^{-1}\TRotSgat)
&
\norm{T}_{\L{2}{\T,\mu}(\om)}&\leq c_{1}\norm{\eps^{-1}\symRot T}_{\L{2}{\S,\eps}(\om)},\\
\forall\,T&\in\mu^{-1}\H{}{\T,\gat}(\Div,\om)\cap R(\TGradgan)
&
\norm{T}_{\L{2}{\T,\mu}(\om)}&\leq c_{2}\norm{\Div\mu T}_{\L{2}{}(\om)},\\
\forall\,v&\in\H{1}{\gan}(\om)\cap(\RT_{\gan})^{\bot_{\L{2}{}(\om)}}
&
\norm{v}_{\L{2}{}(\om)}&\leq c_{2}\norm{\devGrad v}_{\L{2}{\T,\mu}(\om)}
\end{align*}
hold.
\item[\bf(vi)]
For all 
$S\in\H{}{\S,\gat}(\Rot,\om)
\cap\eps^{-1}\H{}{\S,\gan}(\divDiv,\om)
\cap\Harm{}{\S,\gat,\gan,\eps}(\om)^{\bot_{\L{2}{\S,\eps}(\om)}}$
it holds
$$\norm{S}_{\L{2}{\S,\eps}(\om)}^2\leq 
c_{1}^2\norm{\mu^{-1}\Rot S}_{\L{2}{\T,\mu}(\om)}^2
+c_{0}^2\norm{\divDiv\eps S}_{\L{2}{}(\om)}^2.$$
\item[\bf(vi')]
For all 
$T\in\H{}{\T,\gan}(\symRot,\om)
\cap\mu^{-1}\H{}{\T,\gat}(\Div,\om)
\cap\Harm{}{\T,\gan,\gat,\mu}(\om)^{\bot_{\L{2}{\T,\mu}(\om)}}$
it holds
$$\norm{T}_{\L{2}{\T,\mu}(\om)}^2\leq 
c_{1}^2\norm{\eps^{-1}\symRot T}_{\L{2}{\S,\eps}(\om)}^2
+c_{2}^2\norm{\Div\mu T}_{\L{2}{}(\om)}^2.$$
\item[\bf(vii)]
$\Harm{}{\S,\gat,\gan,\eps}(\om)=\{0\}$
and $\Harm{}{\T,\gan,\gat,\mu}(\om)=\{0\}$, 
if $(\om,\gat)$ is extendable.
\end{itemize}
\end{theo}

\subsubsection{Higher Order Mini FA-ToolBox}
\label{bih:sec:hofatb}%

For simplicity, let $\eps=\mu=\id$.
From Section \ref{bih:sec:bihcomplexes} we recall the 
densely defined and closed higher Sobolev order operators
\begin{align}
\begin{aligned}
\label{bih:defhighorderop}
\SGradgradgat^{k}:\H{k+2}{\gat}(\om)\subset\H{k}{\gat}(\om)&\to\H{k}{\S,\gat}(\om),\\
\SGradgradgat^{k,k-1}:\H{k+1}{\gat}(\om)\subset\H{k}{\gat}(\om)&\to\H{k-1}{\S,\gat}(\om),\qquad
k\geq1,\\
\TRotSgat^{k}:\H{k}{\S,\gat}(\Rot,\om)\subset\H{k}{\S,\gat}(\om)&\to\H{k}{\T,\gat}(\om),\\
\DivTgat^{k}:\H{k}{\T,\gat}(\Div,\om)\subset\H{k}{\T,\gat}(\om)&\to\H{k}{\gat}(\om),\\
\TGradgan^{k}:\H{k+1}{\gan}(\om)\subset\H{k}{\gan}(\om)&\to\H{k}{\T,\gan}(\om),\\
\SRotTgan^{k}:\H{k}{\T,\gan}(\symRot,\om)\subset\H{k}{\T,\gan}(\om)&\to\H{k}{\S,\gan}(\om),\\
\divDivSgan^{k}:\H{k}{\S,\gan}(\divDiv,\om)\subset\H{k}{\S,\gan}(\om)&\to\H{k}{\gan}(\om),\\
\divDivSgan^{k,k-1}:\H{k,k-1}{\S,\gan}(\divDiv,\om)\subset\H{k}{\S,\gan}(\om)&\to\H{k-1}{\gan}(\om),\qquad
k\geq1,
\end{aligned}
\end{align}
building the long biharmonic Hilbert complexes
\begin{equation}
\label{bih:bihcomplex6a}
\scriptsize
\def\arrowlength{16ex}
\def\arrowdistance{0}
\begin{tikzcd}[column sep=\arrowlength]
\Pone_{\gat}
\arrow[r, rightarrow, shift left=\arrowdistance, "\iota_{\Pone_{\gat}}"] 
& 
[-4em]
\H{k}{\gat}(\om)
\ar[r, rightarrow, shift left=\arrowdistance, "\SGradgradgat^{k}"] 
&
[-0em]
\H{k}{\S,\gat}(\om)
\ar[r, rightarrow, shift left=\arrowdistance, "\TRotSgat^{k}"] 
& 
[-1em]
\H{k}{\T,\gat}(\om)
\arrow[r, rightarrow, shift left=\arrowdistance, "\DivTgat^{k}"] 
& 
[-3em]
\H{k}{\gat}(\om)
\arrow[r, rightarrow, shift left=\arrowdistance, "\pi_{\RT_{\gan}}"] 
&
[-3em]
\RT_{\gan},\quad k\geq0,
\end{tikzcd}
\end{equation}
\vspace*{-6mm}
\begin{equation}
\label{bih:bihcomplex6b}
\scriptsize
\def\arrowlength{16ex}
\def\arrowdistance{0}
\begin{tikzcd}[column sep=\arrowlength]
\Pone_{\gat}
\arrow[r, leftarrow, shift right=\arrowdistance, "\pi_{\Pone_{\gat}}"]
& 
[-4em]
\H{k}{\gan}(\om)
\ar[r, leftarrow, shift right=\arrowdistance, "\divDivSgan^{k}"]
&
[-1em]
\H{k}{\S,\gan}(\om)
\ar[r, leftarrow, shift right=\arrowdistance, "\SRotTgan^{k}"]
& 
[-2em]
\H{k}{\T,\gan}(\om)
\arrow[r, leftarrow, shift right=\arrowdistance, "-\TGradgan^{k}"]
& 
[-2em]
\H{k}{\gan}(\om)
\arrow[r, leftarrow, shift right=\arrowdistance, "\iota_{\RT_{\gan}}"]
&
[-3em]
\RT_{\gan},\quad k\geq0,
\end{tikzcd}
\end{equation}
\vspace*{-6mm}
\begin{equation}
\label{bih:bihcomplex7a}
\scriptsize
\def\arrowlength{16ex}
\def\arrowdistance{0}
\begin{tikzcd}[column sep=\arrowlength]
\Pone_{\gat}
\arrow[r, rightarrow, shift left=\arrowdistance, "\iota_{\Pone_{\gat}}"] 
& 
[-4em]
\H{k}{\gat}(\om)
\ar[r, rightarrow, shift left=\arrowdistance, "\SGradgradgat^{k,k-1}"] 
&
[-0em]
\H{k-1}{\S,\gat}(\om)
\ar[r, rightarrow, shift left=\arrowdistance, "\TRotSgat^{k-1}"] 
& 
[-1em]
\H{k-1}{\T,\gat}(\om)
\arrow[r, rightarrow, shift left=\arrowdistance, "\DivTgat^{k-1}"] 
& 
[-3em]
\H{k-1}{\gat}(\om)
\arrow[r, rightarrow, shift left=\arrowdistance, "\pi_{\RT_{\gan}}"] 
&
[-3em]
\RT_{\gan},\quad k\geq1,
\end{tikzcd}
\end{equation}
\vspace*{-6mm}
\begin{equation}
\label{bih:bihcomplex7b}
\scriptsize
\def\arrowlength{16ex}
\def\arrowdistance{0}
\begin{tikzcd}[column sep=\arrowlength]
\Pone_{\gat}
\arrow[r, leftarrow, shift right=\arrowdistance, "\pi_{\Pone_{\gat}}"]
& 
[-4em]
\H{k-1}{\gan}(\om)
\ar[r, leftarrow, shift right=\arrowdistance, "\divDivSgan^{k,k-1}"]
&
[-1em]
\H{k}{\S,\gan}(\om)
\ar[r, leftarrow, shift right=\arrowdistance, "\SRotTgan^{k}"]
& 
[-2em]
\H{k}{\T,\gan}(\om)
\arrow[r, leftarrow, shift right=\arrowdistance, "-\TGradgan^{k}"]
& 
[-2em]
\H{k}{\gan}(\om)
\arrow[r, leftarrow, shift right=\arrowdistance, "\iota_{\RT_{\gan}}"]
&
[-3em]
\RT_{\gan},\quad k\geq1,
\end{tikzcd}
\end{equation}

We start with regular representations
implied by Lemma \ref{bih:lem:highorderregdecobih} and Corollary \ref{bih:cor:weakstrongbih}.

\begin{theo}[regular representations and closed ranges]
\label{bih:theo:rangeselawithoutbdpot}
Let $k\geq0$. Then the regular potential representations 
\begin{align*}
R(\SGradgradgat^{k+1,k})
=R(\SGradgradgat^{k})
&=\Gradgrad\H{k}{\gat}(\Gradgrad,\om)
=\Gradgrad\H{k+2}{\gat}(\om)\\
&=\Gradgrad\H{k+1,k}{\gat}(\Gradgrad,\om)\\
&=\H{k}{\S,\gat}(\om)\cap R(\SGradgradgat)\\
&=\H{k}{\S,\gat}(\om)\cap\H{}{\S,\gat,0}(\Rot,\om)\cap\Harm{}{\S,\gat,\gan,\eps}(\om)^{\bot_{\L{2}{\S,\eps}(\om)}}\\
&=\H{k}{\S,\gat,0}(\Rot,\om)\cap\Harm{}{\S,\gat,\gan,\eps}(\om)^{\bot_{\L{2}{\S,\eps}(\om)}},\\
R(\TRotSgat^{k})
&=\Rot\H{k}{\S,\gat}(\Rot,\om)
=\Rot\H{k+1}{\S,\gat}(\om)\\
&=\H{k}{\T,\gat}(\om)\cap R(\TRotSgat)\\
&=\H{k}{\T,\gat}(\om)\cap\H{}{\T,\gat,0}(\Div,\om)\cap\Harm{}{\T,\gan,\gat,\mu}(\om)^{\bot_{\L{2}{\T}(\om)}},\\
&=\H{k}{\T,\gat,0}(\Div,\om)\cap\Harm{}{\T,\gan,\gat,\mu}(\om)^{\bot_{\L{2}{\T}(\om)}},\\
R(\DivTgat^{k})
&=\Div\H{k}{\T,\gat}(\Div,\om)
=\Div\H{k+1}{\T,\gat}(\om)\\
&=\H{k}{\gat}(\om)\cap R(\DivTgat)
=\H{k}{\gat}(\om)\cap(\RT_{\gan})^{\bot_{\L{2}{}(\om)}},\\
R(\TGradgat^{k})
&=\devGrad\H{k}{\gat}(\devGrad,\om)
=\devGrad\H{k+1}{\gat}(\om)\\
&=\H{k}{\T,\gat}(\om)\cap R(\TGradgat)\\
&=\H{k}{\T,\gat}(\om)\cap\H{}{\T,\gat,0}(\symRot,\om)\cap\Harm{}{\T,\gat,\gan,\mu}(\om)^{\bot_{\L{2}{\T,\mu}(\om)}}\\
&=\H{k}{\T,\gat,0}(\symRot,\om)\cap\Harm{}{\T,\gat,\gan,\mu}(\om)^{\bot_{\L{2}{\T,\mu}(\om)}},\\
R(\SRotTgat^{k})
&=\symRot\H{k}{\T,\gat}(\symRot,\om)
=\symRot\H{k+1}{\T,\gat}(\om)\\
&=\H{k}{\S,\gat}(\om)\cap R(\SRotTgat)\\
&=\H{k}{\S,\gat}(\om)\cap\H{}{\S,\gat,0}(\divDiv,\om)\cap\Harm{}{\S,\gan,\gat,\eps}(\om)^{\bot_{\L{2}{\S}(\om)}},\\
&=\H{k}{\S,\gat,0}(\divDiv,\om)\cap\Harm{}{\S,\gan,\gat,\eps}(\om)^{\bot_{\L{2}{\S}(\om)}},\\
R(\divDivSgat^{k+1,k})
=R(\divDivSgat^{k})
&=\divDiv\H{k}{\S,\gat}(\divDiv,\om)
=\divDiv\H{k+2}{\S,\gat}(\om)\\
&=\divDiv\H{k+1,k}{\S,\gat}(\divDiv,\om)\\
&=\H{k}{\gat}(\om)\cap R(\divDivSgat)
=\H{k}{\gat}(\om)\cap(\Pone_{\gan})^{\bot_{\L{2}{}(\om)}}
\end{align*}
hold. In particular, the latter spaces are closed subspaces of 
$\H{k}{\S}(\om)$, $\H{k}{\T}(\om)$, and $\H{k}{}(\om)$, respectively,
and all ranges of the higher Sobolev order operators 
in \eqref{bih:defhighorderop} are closed.
Moroever, the long biharmonic Hilbert complexes \eqref{bih:bihcomplex6a}--\eqref{bih:bihcomplex7b}
are closed.
\end{theo}

A proof is given in Appendix \ref{sec:someproof}.
Note that in Theorem \ref{bih:theo:rangeselawithoutbdpot} 
we claim nothing about bounded regular potential operators, 
leaving the question of bounded potentials to the next sections,
cf.~Theorem \ref{bih:theo:regpothigherorder}.

The reduced operators corresponding to \eqref{bih:defhighorderop} are
\begin{align*}
(\SGradgradgat^{k})_{\bot}:D\big((\SGradgradgat^{k})_{\bot}\big)\subset
(\Pone_{\gat})^{\bot_{\H{k}{\gat}(\om)}}&\to R(\SGradgradgat^{k}),\\
(\SGradgradgat^{k,k-1})_{\bot}:D\big((\SGradgradgat^{k,k-1})_{\bot}\big)\subset
(\Pone_{\gat})^{\bot_{\H{k}{\gat}(\om)}}&\to R(\SGradgradgat^{k-1}),\quad
k\geq1,\\
(\TRotSgat^{k})_{\bot}:D\big((\TRotSgat^{k})_{\bot}\big)\subset 
N(\TRotSgat^{k})^{\bot_{\H{k}{\S,\gat}(\om)}}&\to R(\TRotSgat^{k}),\\
(\DivTgat^{k})_{\bot}:D\big((\DivTgat^{k})_{\bot}\big)\subset
N(\DivTgat^{k})^{\bot_{\H{k}{\T,\gat}(\om)}}&\to R(\DivTgat^{k}),\\
(\TGradgan^{k})_{\bot}:D\big((\TGradgan^{k})_{\bot}\big)\subset
(\RT_{\gan})^{\bot_{\H{k}{\gan}(\om)}}&\to R(\TGradgan^{k}),\\
(\SRotTgan^{k})_{\bot}:D\big((\SRotTgan^{k})_{\bot}\big)\subset 
N(\SRotTgan^{k})^{\bot_{\H{k}{\T,\gan}(\om)}}&\to R(\SRotTgan^{k}),\\
(\divDivSgan^{k})_{\bot}:D\big((\divDivSgan^{k})_{\bot}\big)\subset
N(\divDivSgan^{k})^{\bot_{\H{k}{\S,\gan}(\om)}}&\to R(\divDivSgan^{k}),\\
(\divDivSgan^{k,k-1})_{\bot}:D\big((\divDivSgan^{k,k-1})_{\bot}\big)\subset
N(\divDivSgan^{k})^{\bot_{\H{k}{\S,\gan}(\om)}}&\to R(\divDivSgan^{k-1}),\quad
k\geq1.
\end{align*}

\cite[Lemma 2.1]{PS2021b} and Theorem \ref{bih:theo:rangeselawithoutbdpot} yield:

\begin{theo}[closed ranges and bounded inverse operators]
\label{bih:theo:clrangebdinvbih}
Let $k\geq0$. Then:
\begin{itemize}
\item[\bf(i)]
$R\big((\SGradgradgat^{k})_{\bot}\big)
=R(\SGradgradgat^{k})
=R(\SGradgradgat^{k+1,k})
=R\big((\SGradgradgat^{k+1,k})_{\bot}\big)$ 
are closed and, equivalently, the inverse operators
\begin{align*}
(\SGradgradgat^{k})_{\bot}^{-1}:R(\SGradgradgat^{k})&\to D\big((\SGradgradgat^{k})_{\bot}\big)\\
\text{resp.}\qquad
(\SGradgradgat^{k})_{\bot}^{-1}:R(\SGradgradgat^{k})&\to D(\SGradgradgat^{k}),\\
(\SGradgradgat^{k+1,k})_{\bot}^{-1}:R(\SGradgradgat^{k})&\to D\big((\SGradgradgat^{k+1,k})_{\bot}\big)\\
\text{resp.}\qquad
(\SGradgradgat^{k+1,k})_{\bot}^{-1}:R(\SGradgradgat^{k})&\to D(\SGradgradgat^{k+1,k})
\end{align*}
are bounded. Equivalently, there is $c>0$ such that for all
$u\in D\big((\SGradgradgat^{k})_{\bot}\big)$ resp. $u\in D\big((\SGradgradgat^{k+1,k})_{\bot}\big)$
$$\norm{u}_{\H{k}{}(\om)}\leq c\norm{\Gradgrad u}_{\H{k}{\S}(\om)}
\quad\text{resp.}\quad
\norm{u}_{\H{k+1}{}(\om)}\leq c\norm{\Gradgrad u}_{\H{k}{\S}(\om)}.$$
\item[\bf(ii)]
$R(\TRotSgat^{k})=R\big((\TRotSgat^{k})_{\bot}\big)$ are closed and, equivalently, 
the inverse operator 
\begin{align*}
(\TRotSgat^{k})_{\bot}^{-1}:R(\TRotSgat^{k})&\to D\big((\TRotSgat^{k})_{\bot}\big)\\
\text{resp.}\qquad
(\TRotSgat^{k})_{\bot}^{-1}:R(\TRotSgat^{k})&\to D(\TRotSgat^{k})
\end{align*}
is bounded. Equivalently, there is $c>0$ such that
for all $S\in D\big((\TRotSgat^{k})_{\bot}\big)$
$$\norm{S}_{\H{k}{\S}(\om)}\leq c\norm{\Rot S}_{\H{k}{\T}(\om)}.$$
\item[\bf(iii)]
$R(\DivTgat^{k})=R\big((\DivTgat^{k})_{\bot}\big)$ are closed and, equivalently, 
the inverse operator 
\begin{align*}
(\DivTgat^{k})_{\bot}^{-1}:R(\DivTgat^{k})&\to D\big((\DivTgat^{k})_{\bot}\big)\\
\text{resp.}\qquad
(\DivTgat^{k})_{\bot}^{-1}:R(\DivTgat^{k})&\to D(\DivTgat^{k})
\end{align*}
is bounded. Equivalently, there is $c>0$ such that
for all $T\in D\big((\DivTgat^{k})_{\bot}\big)$
$$\norm{T}_{\H{k}{\T}(\om)}\leq c\norm{\Div T}_{\H{k}{}(\om)}.$$
\item[\bf(iv)]
$R(\TGradgat^{k})=R\big((\TGradgat^{k})_{\bot}\big)$ are closed and, equivalently, 
the inverse operator 
\begin{align*}
(\TGradgat^{k})_{\bot}^{-1}:R(\TGradgat^{k})&\to D\big((\TGradgat^{k})_{\bot}\big)\\
\text{resp.}\qquad
(\TGradgat^{k})_{\bot}^{-1}:R(\TGradgat^{k})&\to D(\TGradgat^{k})
\end{align*}
is bounded. Equivalently, there is $c>0$ such that
for all $v\in D\big((\TGradgat^{k})_{\bot}\big)$
$$\norm{v}_{\H{k}{}(\om)}\leq c\norm{\devGrad v}_{\H{k}{\T}(\om)}.$$
\item[\bf(v)]
$R(\SRotTgat^{k})=R\big((\SRotTgat^{k})_{\bot}\big)$ are closed and, equivalently, 
the inverse operator 
\begin{align*}
(\SRotTgat^{k})_{\bot}^{-1}:R(\SRotTgat^{k})&\to D\big((\SRotTgat^{k})_{\bot}\big)\\
\text{resp.}\qquad
(\SRotTgat^{k})_{\bot}^{-1}:R(\SRotTgat^{k})&\to D(\SRotTgat^{k})
\end{align*}
is bounded. Equivalently, there is $c>0$ such that
for all $T\in D\big((\SRotTgat^{k})_{\bot}\big)$
$$\norm{T}_{\H{k}{\T}(\om)}\leq c\norm{\symRot T}_{\H{k}{\S}(\om)}.$$
\item[\bf(vi)]
$R\big((\divDivSgat^{k})_{\bot}\big)
=R(\divDivSgat^{k})
=R(\divDivSgat^{k+1,k})
=R\big((\divDivSgat^{k+1,k})_{\bot}\big)$ 
are closed and, equivalently, the inverse operators
\begin{align*}
(\divDivSgat^{k})_{\bot}^{-1}:R(\divDivSgat^{k})&\to D\big((\divDivSgat^{k})_{\bot}\big)\\
\text{resp.}\qquad
(\divDivSgat^{k})_{\bot}^{-1}:R(\divDivSgat^{k})&\to D(\divDivSgat^{k}),\\
(\divDivSgat^{k+1,k})_{\bot}^{-1}:R(\divDivSgat^{k})&\to D\big((\divDivSgat^{k+1,k})_{\bot}\big)\\
\text{resp.}\qquad
(\divDivSgat^{k+1,k})_{\bot}^{-1}:R(\divDivSgat^{k})&\to D(\divDivSgat^{k+1,k})
\end{align*}
are bounded. Equivalently, there is $c>0$ such that for all
$S\in D\big((\divDivSgat^{k})_{\bot}\big)$ resp. $S\in D\big((\divDivSgat^{k+1,k})_{\bot}\big)$
$$\norm{S}_{\H{k}{\S}(\om)}\leq c\norm{\divDiv S}_{\H{k}{}(\om)}
\quad\text{resp.}\quad
\norm{S}_{\H{k+1}{\S}(\om)}\leq c\norm{\divDiv S}_{\H{k}{}(\om)}.$$
\end{itemize}
\end{theo}

\begin{lem}[Schwarz' lemma]
\label{bih:lem:scharzlemma}
Let $0\leq|\alpha|\leq k$.
\begin{itemize}
\item[\bf(i)]
If $S\in\H{k}{\S,\gat}(\Rot,\om)$ then
$\p^{\alpha}S\in\H{}{\S,\gat}(\Rot,\om)$ 
and $\Rot\p^{\alpha}S=\p^{\alpha}\Rot S$.
\item[\bf(ii)]
If $T\in\H{k}{\T,\gat}(\Div,\om)$ then
$\p^{\alpha}T\in\H{}{\T,\gat}(\Div,\om)$ and $\Div\p^{\alpha}T=\p^{\alpha}\Div T$.
\item[\bf(iii)]
If $T\in\H{k}{\T,\gat}(\symRot,\om)$ then 
$\p^{\alpha}T\in\H{}{\T,\gat}(\symRot,\om)$ and $\symRot\p^{\alpha}T=\p^{\alpha}\symRot T$.
\item[\bf(iv)]
If $S\in\H{k}{\S,\gat}(\divDiv,\om)$ resp.
$S\in\H{k+1,k}{\S,\gat}(\divDiv,\om)$ then
$\p^{\alpha}S\in\H{}{\S,\gat}(\divDiv,\om)$ resp.
$\p^{\alpha}S\in\H{1,0}{\S,\gat}(\divDiv,\om)$
and $\divDiv\p^{\alpha}S=\p^{\alpha}\divDiv S$.
\end{itemize}
\end{lem}

\begin{theo}[compact embedding]
\label{bih:theo:cptembhigherorder}
Let $k\geq0$. Then the embeddings
\begin{align*}
\H{k}{\S,\gat}(\Rot,\om)\cap\H{k}{\S,\gan}(\divDiv,\om)
&\incl\H{k}{\S,\ga}(\om),\\
\H{k}{\T,\gat}(\Div,\om)\cap\H{k}{\T,\gan}(\symRot,\om)
&\incl\H{k}{\T,\ga}(\om)
\end{align*}
are compact.
\end{theo}

A proof is given in Appendix \ref{sec:someproof}.

\begin{rem}[compact embedding]
\label{bih:rem:cptembhigherorder}
For $k\geq1$, cf.~\cite[Remark 4.12]{PZ2021a}, there is
another and slightly more general proof of the first compact embedding
using a variant of \cite[Lemma 2.22]{PS2021b},
cf.~\cite[Theorem 3.19, Remark 3.20]{PS2021d},
see Appendix \ref{sec:someproof} for a proof. 
It utilises the decomposition
$\H{k,k-1}{\S,\gan}(\divDiv,\om)
=\H{k+1}{\S,\gan}(\om)
+\symRot\H{k+1}{\T,\gan}(\om)$
from Lemma \ref{bih:lem:highorderregdecobih}
and leads immediately to the next (stronger) result.
\end{rem}

\begin{theo}[compact embedding]
\label{bih:theo:cptembhigherorderkkmo}
Let $k\geq1$. Then the embedding
$$\H{k}{\S,\gat}(\Rot,\om)\cap\H{k,k-1}{\S,\gan}(\divDiv,\om)
\incl\H{k}{\S,\ga}(\om)$$
is compact.
\end{theo}

\begin{theo}[Friedrichs/Poincar\'e type estimate]
\label{bih:theo:friedrichspoincarehigherorder}
Let $k\geq0$. 
Then there exists $c>0$ such that for all 
\begin{align*}
S&\in\H{k}{\S,\gat}(\Rot,\om)
\cap\H{k}{\S,\gan}(\divDiv,\om)
\cap\Harm{}{\S,\gat,\gan,\id}(\om)^{\bot_{\L{2}{\S}(\om)}},\\
T&\in\H{k}{\T,\gat}(\symRot,\om)
\cap\H{k}{\T,\gan}(\Div,\om)
\cap\Harm{}{\T,\gat,\gan,\id}(\om)^{\bot_{\L{2}{\T}(\om)}}
\end{align*}
it holds
\begin{align*}
\norm{S}_{\H{k}{\S}(\om)}
&\leq c\big(\norm{\Rot S}_{\H{k}{\T}(\om)}
+\norm{\divDiv S}_{\H{k}{}(\om)}\big),\\
\norm{T}_{\H{k}{\T}(\om)}
&\leq c\big(\norm{\symRot T}_{\H{k}{\S}(\om)}
+\norm{\Div T}_{\H{k}{}(\om)}\big)
\end{align*}
respectively. The orthogonality condition $\Harm{}{\S,\gat,\gan,\id}(\om)^{\bot_{\L{2}{\S}(\om)}}$
and $\Harm{}{\T,\gat,\gan,\id}(\om)^{\bot_{\L{2}{\T}(\om)}}$
can be replaced by the weaker conditions 
$\Harm{k}{\S,\gat,\gan,\id}(\om)^{\bot_{\L{2}{\S}(\om)}}$ or
$\Harm{k}{\S,\gat,\gan,\id}(\om)^{\bot_{\H{k}{\S}(\om)}}$
and $\Harm{k}{\T,\gat,\gan,\id}(\om)^{\bot_{\L{2}{\T}(\om)}}$ or
$\Harm{k}{\T,\gat,\gan,\id}(\om)^{\bot_{\H{k}{\T}(\om)}}$,
respectively. In particular, 
\begin{align*}
\forall\,S&\in\H{k}{\S,\gat}(\Rot,\om)
\cap R(\SRotTgan^{k})
&\norm{S}_{\H{k}{\S}(\om)}
&\leq c\norm{\Rot S}_{\H{k}{\T}(\om)},\\
\forall\,S&\in\H{k}{\S,\gan}(\divDiv,\om)
\cap R(\SGradgradgat^{k})
&\norm{S}_{\H{k}{\S}(\om)}
&\leq c\norm{\divDiv S}_{\H{k}{}(\om)},\\
\forall\,T&\in\H{k}{\T,\gat}(\symRot,\om)
\cap R(\TRotSgan^{k})
&\norm{T}_{\H{k}{\T}(\om)}
&\leq c\norm{\symRot T}_{\H{k}{\S}(\om)},\\
\forall\,T&\in\H{k}{\T,\gan}(\Div,\om)
\cap R(\TGradgat^{k})
&\norm{T}_{\H{k}{\T}(\om)}
&\leq c\norm{\Div T}_{\H{k}{}(\om)}
\end{align*}
with
\begin{align*}
R(\SRotTgan^{k})
&=\H{k}{\S,\gan,0}(\divDiv,\om)
\cap\Harm{}{\S,\gat,\gan,\id}(\om)^{\bot_{\L{2}{\S}(\om)}},\\
R(\SGradgradgat^{k+1,k})=R(\SGradgradgat^{k})&=\H{k}{\S,\gat,0}(\Rot,\om)
\cap\Harm{}{\S,\gat,\gan,\id}(\om)^{\bot_{\L{2}{\S}(\om)}},\\
R(\TRotSgan^{k})
&=\H{k}{\T,\gan,0}(\Div,\om)
\cap\Harm{}{\T,\gat,\gan,\id}(\om)^{\bot_{\L{2}{\T}(\om)}},\\
R(\TGradgat^{k})&=\H{k}{\T,\gat,0}(\symRot,\om)
\cap\Harm{}{\T,\gat,\gan,\id}(\om)^{\bot_{\L{2}{\T}(\om)}}.
\end{align*}

Analogously, for $k\geq1$ there exists $c>0$ such that 
$$\norm{S}_{\H{k}{\S}(\om)}\leq 
c\big(\norm{\Rot S}_{\H{k}{\T}(\om)}
+\norm{\divDiv S}_{\H{k-1}{}(\om)}\big)$$
for all $S$ in $\H{k}{\S,\gat}(\Rot,\om)
\cap\H{k,k-1}{\S,\gan}(\divDiv,\om)
\cap\Harm{}{\S,\gat,\gan,\id}(\om)^{\bot_{\L{2}{\S}(\om)}}$.
Moreover, 
\begin{align*}
\forall\,S&\in\H{k,k-1}{\S,\gan}(\divDiv,\om)
\cap R(\SGradgradgat^{k})
&\norm{S}_{\H{k}{\S}(\om)}
&\leq c\norm{\divDiv S}_{\H{k-1}{}(\om)}.
\end{align*}
\end{theo}

The proof follows by a standard contradiction argument.

\begin{rem}[Friedrichs/Poincar\'e/Korn type estimate]
\label{bih:rem:friedrichspoincarekornhigherorder}
Let $k\geq0$. 
Similar to Theorem \ref{bih:theo:friedrichspoincarehigherorder}
and by Rellich's selection theorem, cf.~the estimates in Theorem \ref{bih:theo:clrangebdinvbih},
there exists $c>0$ such that for all 
$v\in\H{k+1}{\gat}(\om)\cap(\RT_{\gat})^{\bot_{\L{2}{}(\om)}}$
and for all $u\in\H{k+2}{\gat}(\om)\cap(\Pone_{\gat})^{\bot_{\L{2}{}(\om)}}$
$$\norm{v}_{\H{k}{}(\om)}\leq c\norm{\devGrad v}_{\H{k}{\T}(\om)},\qquad
\norm{u}_{\H{k}{}(\om)}\leq\norm{u}_{\H{k+1}{}(\om)}\leq c\norm{\Gradgrad u}_{\H{k}{\S}(\om)}.$$
As in Theorem \ref{bih:theo:clrangebdinvbih},
$(\RT_{\gat})^{\bot_{\L{2}{}(\om)}}$ and $(\Pone_{\gat})^{\bot_{\L{2}{}(\om)}}$ can be replaced 
by $(\RT_{\gat})^{\bot_{\H{k}{\gat}(\om)}}$ and $(\Pone_{\gat})^{\bot_{\H{k}{\gat}(\om)}}$, respectively.
\end{rem}

\subsection{Regular Potentials and Decompositions II}
\label{bih:sec:regpotdeco2}%

Let $k\geq0$.
According to Theorem \ref{bih:theo:clrangebdinvbih} the inverses
of the reduced operators  
\begin{align*}
(\SGradgradgat^{k})_{\bot}^{-1}:R(\SGradgradgat^{k})&\to D(\SGradgradgat^{k})=\H{k+2}{\gat}(\om),\\
(\SGradgradgat^{k+1,k})_{\bot}^{-1}:R(\SGradgradgat^{k})&\to D(\SGradgradgat^{k+1,k})=\H{k+2}{\gat}(\om),\\
(\TRotSgat^{k})_{\bot}^{-1}:R(\TRotSgat^{k})&\to D(\TRotSgat^{k})=\H{k}{\S,\gat}(\Rot,\om),\\
(\DivTgat^{k})_{\bot}^{-1}:R(\DivTgat^{k})&\to D(\DivTgat^{k})=\H{k}{\T,\gat}(\Div,\om),\\
(\TGradgat^{k})_{\bot}^{-1}:R(\TGradgat^{k})&\to D(\TGradgat^{k})=\H{k+1}{\gat}(\om),\\
(\SRotTgat^{k})_{\bot}^{-1}:R(\SRotTgat^{k})&\to D(\SRotTgat^{k})=\H{k}{\T,\gat}(\symRot,\om),\\
(\divDivSgat^{k})_{\bot}^{-1}:R(\divDivSgat^{k})&\to D(\divDivSgat^{k})=\H{k}{\S,\gat}(\divDiv,\om),\\
(\divDivSgat^{k+1,k})_{\bot}^{-1}:R(\divDivSgat^{k})&\to D(\divDivSgat^{k+1,k})=\H{k+1,k}{\S,\gat}(\divDiv,\om)
\end{align*}
are bounded and we recall the bounded linear regular decomposition operators
\begin{align*}
\PotQ_{\TRotS,\gat}^{k,1}:
\H{k}{\S,\gat}(\Rot,\om)&\to\H{k+1}{\S,\gat}(\om),
&
\PotQ_{\TRotS,\gat}^{k,0}:
\H{k}{\S,\gat}(\Rot,\om)&\to\H{k+2}{\gat}(\om),\\
\PotQ_{\DivT,\gat}^{k,1}:
\H{k}{\T,\gat}(\Div,\om)&\to\H{k+1}{\T,\gat}(\om),
&
\PotQ_{\DivT,\gat}^{k,0}:
\H{k}{\T,\gat}(\Div,\om)&\to\H{k+1}{\S,\gat}(\om),\\
\PotQ_{\SRotT,\gat}^{k,1}:
\H{k}{\T,\gat}(\symRot,\om)&\to\H{k+1}{\T,\gat}(\om),
&
\PotQ_{\SRotT,\gat}^{k,0}:
\H{k}{\T,\gat}(\symRot,\om)&\to\H{k+1}{\gat}(\om),\\
\PotQ_{\divDivS,\gat}^{k,1}:
\H{k}{\S,\gat}(\divDiv,\om)&\to\H{k+2}{\S,\gat}(\om),
&
\PotQ_{\divDivS,\gat}^{k,0}:
\H{k}{\S,\gat}(\divDiv,\om)&\to\H{k+1}{\T,\gat}(\om),\\
\PotQ_{\divDivS,\gat}^{k+1,k,1}:
\H{k+1,k}{\S,\gat}(\divDiv,\om)&\to\H{k+2}{\S,\gat}(\om),
&
\PotQ_{\divDivS,\gat}^{k+1,k,0}:
\H{k+1,k}{\S,\gat}(\divDiv,\om)&\to\H{k+2}{\T,\gat}(\om)
\end{align*}
from Lemma \ref{bih:lem:highorderregdecobih}.
Similar to \cite[Theorem 4.18, Theorem 5.2]{PS2021b}
and \cite[Theorem 3.24, Theorem 3.25]{PS2021d},
cf.~\cite[Lemma 2.22, Theorem 2.23]{PS2021b},
we obtain the following sequence of results:

\begin{theo}[bounded regular potentials from bounded regular decompositions]
\label{bih:theo:regpothigherorder}
For $k\geq0$ there exist bounded linear regular potential operators
\begin{align*}
\PotP_{\SGradgrad,\gat}^{k}:=(\SGradgradgat^{k})_{\bot}^{-1}:
\H{k}{\S,\gat,0}(\Rot,\om)\cap\Harm{}{\S,\gat,\gan,\eps}(\om)^{\bot_{\L{2}{\S,\eps}(\om)}}
&\to\H{k+2}{\gat}(\om),\\
\PotP_{\SGradgrad,\gat}^{k+1,k}:=(\SGradgradgat^{k+1,k})_{\bot}^{-1}:
\H{k}{\S,\gat,0}(\Rot,\om)\cap\Harm{}{\S,\gat,\gan,\eps}(\om)^{\bot_{\L{2}{\S,\eps}(\om)}}
&\to\H{k+2}{\gat}(\om),\\
\PotP_{\TRotS,\gat}^{k}:=\PotQ_{\TRotS,\gat}^{k,1}(\TRotSgat^{k})_{\bot}^{-1}:
\H{k}{\T,\gat,0}(\Div,\om)\cap\Harm{}{\T,\gan,\gat,\mu}(\om)^{\bot_{\L{2}{\T}(\om)}}
&\to\H{k+1}{\S,\gat}(\om),\\
\PotP_{\DivT,\gat}^{k}:=\PotQ_{\DivT,\gat}^{k,1}(\DivTgat^{k})_{\bot}^{-1}:
\H{k}{\gat}(\om)\cap(\RT_{\gan})^{\bot_{\L{2}{}(\om)}}
&\to\H{k+1}{\T,\gat}(\om),\\
\PotP_{\TGrad,\gat}^{k}:=(\TGradgat^{k})_{\bot}^{-1}:
\H{k}{\T,\gat,0}(\symRot,\om)\cap\Harm{}{\T,\gat,\gan,\mu}(\om)^{\bot_{\L{2}{\T,\mu}(\om)}}
&\to\H{k+1}{\gat}(\om),\\
\PotP_{\SRotT,\gat}^{k}:=\PotQ_{\SRotT,\gat}^{k,1}(\SRotTgat^{k})_{\bot}^{-1}:
\H{k}{\S,\gat,0}(\divDiv,\om)\cap\Harm{}{\S,\gan,\gat,\eps}(\om)^{\bot_{\L{2}{\S}(\om)}}
&\to\H{k+1}{\T,\gat}(\om),\\
\PotP_{\divDivS,\gat}^{k}:=\PotQ_{\divDivS,\gat}^{k,1}(\divDivSgat^{k})_{\bot}^{-1}:
\H{k}{\gat}(\om)\cap(\Pone_{\gan})^{\bot_{\L{2}{}(\om)}}
&\to\H{k+2}{\S,\gat}(\om),\\
\PotP_{\divDivS,\gat}^{k+1,k}:=\PotQ_{\divDivS,\gat}^{k+1,k,1}(\divDivSgat^{k+1,k})_{\bot}^{-1}:
\H{k}{\gat}(\om)\cap(\Pone_{\gan})^{\bot_{\L{2}{}(\om)}}
&\to\H{k+2}{\S,\gat}(\om),
\end{align*}
such that 
\begin{align*}
\Gradgrad\PotP_{\SGradgrad,\gat}^{k+1,k}
=\Gradgrad\PotP_{\SGradgrad,\gat}^{k}
&=\id|_{\H{k}{\S,\gat,0}(\Rot,\om)\cap\Harm{}{\S,\gat,\gan,\eps}(\om)^{\bot_{\L{2}{\S,\eps}(\om)}}},\\
\Rot\PotP_{\TRotS,\gat}^{k}
&=\id|_{\H{k}{\T,\gat,0}(\Div,\om)\cap\Harm{}{\T,\gan,\gat,\mu}(\om)^{\bot_{\L{2}{\T}(\om)}}},\\
\Div\PotP_{\DivT,\gat}^{k}
&=\id|_{\H{k}{\gat}(\om)\cap(\RT_{\gan})^{\bot_{\L{2}{}(\om)}}},\\
\devGrad\PotP_{\TGrad,\gat}^{k}
&=\id|_{\H{k}{\T,\gat,0}(\symRot,\om)\cap\Harm{}{\T,\gat,\gan,\mu}(\om)^{\bot_{\L{2}{\T,\mu}(\om)}}},\\
\symRot\PotP_{\SRotT,\gat}^{k}
&=\id|_{\H{k}{\S,\gat,0}(\divDiv,\om)\cap\Harm{}{\S,\gan,\gat,\eps}(\om)^{\bot_{\L{2}{\S}(\om)}}},\\
\divDiv\PotP_{\divDivS,\gat}^{k+1,k}
=\divDiv\PotP_{\divDivS,\gat}^{k}
&=\id|_{\H{k}{\gat}(\om)\cap(\Pone_{\gan})^{\bot_{\L{2}{}(\om)}}}.
\end{align*}
In particular, all potentials in Theorem \ref{bih:theo:rangeselawithoutbdpot}
can be chosen such that they depend continuously on the data.
$\PotP_{\SGradgrad,\gat}^{k+1,k}$, $\PotP_{\SGradgrad,\gat}^{k}$, $\PotP_{\SRotT,\gat}^{k}$, $\PotP_{\DivT,\gat}^{k}$,
$\PotP_{\TGrad,\gat}^{k}$, $\PotP_{\TRotS,\gat}^{k}$, $\PotP_{\divDivS,\gat}^{k}$, $\PotP_{\divDivS,\gat}^{k+1,k}$
are right inverses of $\Gradgrad$, $\Rot$, $\Div$, $\devGrad$, $\symRot$, $\divDiv$, respectively.
\end{theo}

\begin{theo}[bounded regular decompositions from bounded regular potentials]
\label{bih:theo:regdecohigherorder}
For $k\geq0$ the bounded regular decompositions
\begin{align*}
\H{k}{\S,\gat}(\Rot,\om)
&=\H{k+1}{\S,\gat}(\om)
+\H{k}{\S,\gat,0}(\Rot,\om)
=\H{k+1}{\S,\gat}(\om)
+\Gradgrad\H{k+2}{\gat}(\om)\\
&=R(\widetilde\PotQ_{\TRotS,\gat}^{k,1})
\dotplus\H{k}{\S,\gat,0}(\Rot,\om)\\
&=R(\widetilde\PotQ_{\TRotS,\gat}^{k,1})
\dotplus R(\widetilde\PotN_{\TRotS,\gat}^{k}),\\
\H{k}{\T,\gat}(\Div,\om)
&=\H{k+1}{\T,\gat}(\om)
+\H{k}{\T,\gat,0}(\Div,\om)
=\H{k+1}{\T,\gat}(\om)
+\Rot\H{k+1}{\S,\gat}(\om)\\
&=R(\widetilde\PotQ_{\DivT,\gat}^{k,1})
\dotplus\H{k}{\T,\gat,0}(\Div,\om)\\
&=R(\widetilde\PotQ_{\DivT,\gat}^{k,1})
\dotplus R(\widetilde\PotN_{\DivT,\gat}^{k}),\\
\H{k}{\T,\gat}(\symRot,\om)
&=\H{k+1}{\T,\gat}(\om)
+\H{k}{\T,\gat,0}(\symRot,\om)
=\H{k+1}{\T,\gat}(\om)
+\devGrad\H{k+1}{\gat}(\om)\\
&=R(\widetilde\PotQ_{\SRotT,\gat}^{k,1})
\dotplus\H{k}{\T,\gat,0}(\symRot,\om)\\
&=R(\widetilde\PotQ_{\SRotT,\gat}^{k,1})
\dotplus R(\widetilde\PotN_{\SRotT,\gat}^{k}),\\
\H{k}{\S,\gat}(\divDiv,\om)
&=\H{k+2}{\S,\gat}(\om)
+\H{k}{\S,\gat,0}(\divDiv,\om)
=\H{k+2}{\S,\gat}(\om)
+\symRot\H{k+1}{\T,\gat}(\om)\\
&=R(\widetilde\PotQ_{\divDivS,\gat}^{k,1})
\dotplus\H{k}{\S,\gat,0}(\divDiv,\om)\\
&=R(\widetilde\PotQ_{\divDivS,\gat}^{k,1})
\dotplus R(\widetilde\PotN_{\divDivS,\gat}^{k}),\\
\H{k+1,k}{\S,\gat}(\divDiv,\om)
&=\H{k+2}{\S,\gat}(\om)
+\H{k+1}{\S,\gat,0}(\divDiv,\om)
=\H{k+2}{\S,\gat}(\om)
+\symRot\H{k+2}{\T,\gat}(\om)\\
&=R(\widetilde\PotQ_{\divDivS,\gat}^{k+1,k,1})
\dotplus\H{k+1}{\S,\gat,0}(\divDiv,\om)\\
&=R(\widetilde\PotQ_{\divDivS,\gat}^{k+1,k,1})
\dotplus R(\widetilde\PotN_{\divDivS,\gat}^{k+1,k})
\end{align*}
hold with bounded linear regular decomposition operators
\begin{align*}
\widetilde\PotQ_{\TRotS,\gat}^{k,1}:=\PotP_{\TRotS,\gat}^{k}\TRotSgat^{k}:
\H{k}{\S,\gat}(\Rot,\om)&\to\H{k+1}{\S,\gat}(\om),\\
\widetilde\PotQ_{\DivT,\gat}^{k,1}:=\PotP_{\DivT,\gat}^{k}\DivTgat^{k}:
\H{k}{\T,\gat}(\Div,\om)&\to\H{k+1}{\T,\gat}(\om),\\
\widetilde\PotQ_{\SRotT,\gat}^{k,1}:=\PotP_{\SRotT,\gat}^{k}\SRotTgat^{k}:
\H{k}{\T,\gat}(\symRot,\om)&\to\H{k+1}{\T,\gat}(\om),\\
\widetilde\PotQ_{\divDivS,\gat}^{k,1}:=\PotP_{\divDivS,\gat}^{k}\divDivSgat^{k}:
\H{k}{\S,\gat}(\divDiv,\om)&\to\H{k+2}{\S,\gat}(\om),\\
\widetilde\PotQ_{\divDivS,\gat}^{k+1,k,1}:=\PotP_{\divDivS,\gat}^{k+1,k}\divDivSgat^{k+1,k}:
\H{k+1,k}{\S,\gat}(\divDiv,\om)&\to\H{k+2}{\S,\gat}(\om),\\
\widetilde\PotN_{\TRotS,\gat}^{k}:
\H{k}{\S,\gat}(\Rot,\om)&\to\H{k}{\S,\gat,0}(\Rot,\om),\\
\widetilde\PotN_{\DivT,\gat}^{k}:
\H{k}{\T,\gat}(\Div,\om)&\to\H{k}{\T,\gat,0}(\Div,\om),\\
\widetilde\PotN_{\SRotT,\gat}^{k}:
\H{k}{\T,\gat}(\symRot,\om)&\to\H{k}{\T,\gat,0}(\symRot,\om),\\
\widetilde\PotN_{\divDivS,\gat}^{k}:
\H{k}{\S,\gat}(\divDiv,\om)&\to\H{k}{\S,\gat,0}(\divDiv,\om),\\
\widetilde\PotN_{\divDivS,\gat}^{k+1,k}:
\H{k+1,k}{\S,\gat}(\divDiv,\om)&\to\H{k+1}{\S,\gat,0}(\divDiv,\om)
\end{align*}
satisfying
\begin{align*}
\id_{\H{k}{\S,\gat}(\Rot,\om)}
&=\widetilde\PotQ_{\TRotS,\gat}^{k,1}
+\widetilde\PotN_{\TRotS,\gat}^{k},\\
\id_{\H{k}{\T,\gat}(\Div,\om)}
&=\widetilde\PotQ_{\DivT,\gat}^{k,1}
+\widetilde\PotN_{\DivT,\gat}^{k},\\
\id_{\H{k}{\T,\gat}(\symRot,\om)}
&=\widetilde\PotQ_{\SRotT,\gat}^{k,1}
+\widetilde\PotN_{\SRotT,\gat}^{k},\\
\id_{\H{k}{\S,\gat}(\divDiv,\om)}
&=\widetilde\PotQ_{\divDivS,\gat}^{k,1}
+\widetilde\PotN_{\divDivS,\gat}^{k},\\
\id_{\H{k+1,k}{\S,\gat}(\divDiv,\om)}
&=\widetilde\PotQ_{\divDivS,\gat}^{k+1,k,1}
+\widetilde\PotN_{\divDivS,\gat}^{k+1,k}.
\end{align*}
\end{theo}

\begin{cor}[bounded regular kernel decompositions]
\label{bih:cor:regdecokernelhigherorder}
For $k\geq0$ the bounded regular kernel decompositions
\begin{align*}
\H{k}{\S,\gat,0}(\Rot,\om)
&=\H{k+1}{\S,\gat,0}(\Rot,\om)
+\Gradgrad\H{k+2}{\gat}(\om),\\
\H{k}{\T,\gat,0}(\Div,\om)
&=\H{k+1}{\T,\gat,0}(\Div,\om)
+\Rot\H{k+1}{\S,\gat}(\om),\\
\H{k}{\T,\gat,0}(\symRot\om)
&=\H{k+1}{\T,\gat,0}(\symRot,\om)
+\devGrad\H{k+1}{\gat}(\om),\\
\H{k}{\S,\gat,0}(\divDiv,\om)
&=\H{k+2}{\S,\gat,0}(\divDiv,\om)
+\symRot\H{k+1}{\T,\gat}(\om)
\end{align*}
hold.
\end{cor}

As in \cite[Remark 3.27, Theorem 3.28]{PS2021d}
and \cite[Theorem 4.18, Remark 4.19, Theorem 5.2, Remark 5.3]{PS2021b},
cf.~\cite[Sections 2.3 and 2.4]{PS2021b},
there is a collection of results about the 
bounded regular decomposition operators, 
see Remark \ref{bih:rem:regdecohigherorder} and Remark \ref{bih:rem:regdecohigherorderproj}
of Appendix \ref{sec:someremarks}.

Corollary \ref{bih:cor:regdecokernelhigherorder} shows:

\begin{cor}[bounded regular higher order kernel decompositions]
\label{bih:cor:regdecokernelhigherorder2}
For $k,\ell\geq0$ the bounded regular kernel decompositions
\begin{align*}
N(\TRotSgat^{k})
=\H{k}{\S,\gat,0}(\Rot,\om)
&=\H{\ell}{\S,\gat,0}(\Rot,\om)
+\Gradgrad\H{k+2}{\gat}(\om),\\
N(\DivTgat^{k})
=\H{k}{\T,\gat,0}(\Div,\om)
&=\H{\ell}{\T,\gat,0}(\Div,\om)
+\Rot\H{k+1}{\S,\gat}(\om),\\
N(\SRotTgat^{k})
=\H{k}{\T,\gat,0}(\symRot,\om)
&=\H{\ell}{\T,\gat,0}(\symRot,\om)
+\devGrad\H{k+1}{\gat}(\om),\\
N(\divDivSgat^{k})
=\H{k}{\S,\gat,0}(\divDiv,\om)
&=\H{\ell}{\S,\gat,0}(\divDiv,\om)
+\symRot\H{k+1}{\T,\gat}(\om)
\end{align*}
hold. In particular, for $k=0$ and all $\ell\geq0$
\begin{align*}
N(\TRotSgat)
=\H{}{\S,\gat,0}(\Rot,\om)
&=\H{\ell}{\S,\gat,0}(\Rot,\om)
+\Gradgrad\H{2}{\gat}(\om),\\
N(\DivTgat)
=\H{}{\T,\gat,0}(\Div,\om)
&=\H{\ell}{\T,\gat,0}(\Div,\om)
+\Rot\H{1}{\S,\gat}(\om),\\
N(\SRotTgat)
=\H{}{\T,\gat,0}(\symRot,\om)
&=\H{\ell}{\T,\gat,0}(\symRot,\om)
+\devGrad\H{1}{\gat}(\om),\\
N(\divDivSgat)
=\H{}{\S,\gat,0}(\divDiv,\om)
&=\H{\ell}{\S,\gat,0}(\divDiv,\om)
+\symRot\H{1}{\T,\gat}(\om).
\end{align*}
\end{cor}

\subsection{Dirichlet/Neumann Fields}
\label{bih:sec:dirneuB}%

From Theorem \ref{bih:theo:miniFATzeroorder} (iv) we recall the 
slightly modified orthonormal Helmholtz type decompositions
\begin{align}
\begin{aligned}
\label{bih:helmcoho1}
\L{2}{\S,\eps}(\om)
&=R(\SGradgradgat)
\oplus_{\L{2}{\S,\eps}(\om)}N(\divDivSgan\eps)\\
&=N(\TRotSgat)
\oplus_{\L{2}{\S,\eps}(\om)}R(\eps^{-1}\SRotTgan)\\
&=R(\SGradgradgat)
\oplus_{\L{2}{\S,\eps}(\om)}\Harm{}{\S,\gat,\gan,\eps}(\om)
\oplus_{\L{2}{\S,\eps}(\om)}R(\eps^{-1}\SRotTgan),\\
N(\TRotSgat)
&=R(\SGradgradgat)
\oplus_{\L{2}{\S,\eps}(\om)}\Harm{}{\S,\gat,\gan,\eps}(\om),\\
N(\divDivSgan\eps)
&=\Harm{}{\S,\gat,\gan,\eps}(\om)
\oplus_{\L{2}{\S,\eps}(\om)}R(\eps^{-1}\SRotTgan),\\
\L{2}{\T,\mu}(\om)
&=R(\TGradgat)
\oplus_{\L{2}{\T,\mu}(\om)}N(\DivTgan\mu)\\
&=N(\SRotTgat)
\oplus_{\L{2}{\T,\mu}(\om)}R(\mu^{-1}\TRotSgan)\\
&=R(\TGradgat)
\oplus_{\L{2}{\T,\mu}(\om)}\Harm{}{\T,\gat,\gan,\mu}(\om)
\oplus_{\L{2}{\T,\mu}(\om)}R(\mu^{-1}\TRotSgan),\\
N(\SRotTgat)
&=R(\TGradgat)
\oplus_{\L{2}{\T,\mu}(\om)}\Harm{}{\T,\gat,\gan,\mu}(\om),\\
N(\DivTgan\mu)
&=\Harm{}{\T,\gat,\gan,\mu}(\om)
\oplus_{\L{2}{\T,\mu}(\om)}R(\mu^{-1}\TRotSgan).
\end{aligned}
\end{align}
Let us denote the $\L{2}{\S,\eps}(\om)$- and $\L{2}{\T,\mu}(\om)$-orthonormal projectors onto 
$N(\divDivSgan\eps)$, $N(\TRotSgat)$
and $N(\DivTgan\mu)$, $N(\SRotTgat)$ by
\begin{align*}
\pi_{N(\divDivSgan\eps)}:\L{2}{\S,\eps}(\om)&\to N(\divDivSgan\eps),
&
\pi_{N(\DivTgan\mu)}:\L{2}{\T,\mu}(\om)&\to N(\DivTgan\mu),\\
\pi_{N(\TRotSgat)}:\L{2}{\S,\eps}(\om)&\to N(\TRotSgat),
&
\pi_{N(\SRotTgat)}:\L{2}{\T,\mu}(\om)&\to N(\SRotTgat),
\end{align*}
respectively. Then
\begin{align*}
\pi_{N(\divDivSgan\eps)}|_{N(\TRotSgat)}:N(\TRotSgat)&\to\Harm{}{\S,\gat,\gan,\eps}(\om),\\
\pi_{N(\TRotSgat)}|_{N(\divDivSgan\eps)}:N(\divDivSgan\eps)&\to\Harm{}{\S,\gat,\gan,\eps}(\om),\\
\pi_{N(\DivTgan\mu)}|_{N(\SRotTgat)}:N(\SRotTgat)&\to\Harm{}{\T,\gat,\gan,\mu}(\om),\\
\pi_{N(\SRotTgat)}|_{N(\DivTgan\mu)}:N(\DivTgan\mu)&\to\Harm{}{\T,\gat,\gan,\mu}(\om)
\end{align*}
are onto. Moreover, 
\begin{align*}
\pi_{N(\divDivSgan\eps)}|_{R(\SGradgradgat)}&=0,
&
\pi_{N(\DivTgan\mu)}|_{R(\TGradgat)}&=0,\\
\pi_{N(\TRotSgat)}|_{R(\eps^{-1}\SRotTgan)}&=0,
&
\pi_{N(\SRotTgat)}|_{R(\mu^{-1}\TRotSgan)}&=0,\\
\pi_{N(\divDivSgan\eps)}|_{\Harm{}{\S,\gat,\gan,\eps}(\om)}&=\id_{\Harm{}{\S,\gat,\gan,\eps}(\om)},
&
\pi_{N(\DivTgan\mu)}|_{\Harm{}{\T,\gat,\gan,\mu}(\om)}&=\id_{\Harm{}{\T,\gat,\gan,\mu}(\om)},\\
\pi_{N(\TRotSgat)}|_{\Harm{}{\S,\gat,\gan,\eps}(\om)}&=\id_{\Harm{}{\S,\gat,\gan,\eps}(\om)},
&
\pi_{N(\SRotTgat)}|_{\Harm{}{\T,\gat,\gan,\mu}(\om)}&=\id_{\Harm{}{\T,\gat,\gan,\mu}(\om)}.
\end{align*}
Therefore, by Corollary \ref{bih:cor:regdecokernelhigherorder2} and for all $\ell\geq0$
\begin{align*}
\Harm{}{\S,\gat,\gan,\eps}(\om)
&=\pi_{N(\divDivSgan\eps)}N(\TRotSgat)
=\pi_{N(\divDivSgan\eps)}\H{\ell}{\S,\gat,0}(\Rot,\om),\\
\Harm{}{\S,\gat,\gan,\eps}(\om)
&=\pi_{N(\TRotSgat)}N(\divDivSgan\eps)
=\pi_{N(\TRotSgat)}\eps^{-1}\H{\ell}{\S,\gan,0}(\divDiv,\om),\\
\Harm{}{\T,\gat,\gan,\mu}(\om)
&=\pi_{N(\DivTgan\mu)}N(\SRotTgat)
=\pi_{N(\DivTgan\mu)}\H{\ell}{\T,\gat,0}(\symRot,\om),\\
\Harm{}{\T,\gat,\gan,\mu}(\om)
&=\pi_{N(\SRotTgat)}N(\DivTgan\mu)
=\pi_{N(\SRotTgat)}\mu^{-1}\H{\ell}{\T,\gan,0}(\Div,\om),
\end{align*}
where we have used 
$$N(\divDivSgan\eps)=\eps^{-1}\H{}{\S,\gan,0}(\divDiv,\om),\qquad
N(\DivTgan\mu)=\mu^{-1}\H{}{\T,\gan,0}(\Div,\om).$$

Hence with
\begin{align*}
\H{\infty}{\S,\gat,0}(\Rot,\om)&:=\bigcap_{k\geq0}\H{k}{\S,\gat,0}(\Rot,\om),
&
\H{\infty}{\S,\gan,0}(\divDiv,\om)&:=\bigcap_{k\geq0}\H{k}{\S,\gan,0}(\divDiv,\om),\\
\H{\infty}{\T,\gat,0}(\symRot,\om)&:=\bigcap_{k\geq0}\H{k}{\T,\gat,0}(\symRot,\om),
&
\H{\infty}{\T,\gan,0}(\Div,\om)&:=\bigcap_{k\geq0}\H{k}{\T,\gan,0}(\Div,\om),
\end{align*}
and with the finite numbers
$$d_{\om,\S,\gat}:=\dim\Harm{}{\S,\gat,\gan,\eps}(\om),\qquad
d_{\om,\T,\gat}:=\dim\Harm{}{\T,\gat,\gan,\mu}(\om)$$
we get the following result:

\begin{theo}[smooth pre-bases of Dirichlet/Neumann fields]
\label{bih:theo:cohomologyinfty}
It holds 
\begin{align*}
\pi_{N(\divDivSgan\eps)}\H{\infty}{\S,\gat,0}(\Rot,\om)
&=\Harm{}{\S,\gat,\gan,\eps}(\om)
=\pi_{N(\TRotSgat)}\eps^{-1}\H{\infty}{\S,\gan,0}(\divDiv,\om),\\
\pi_{N(\DivTgan\mu)}\H{\infty}{\T,\gat,0}(\symRot,\om)
&=\Harm{}{\T,\gat,\gan,\mu}(\om)
=\pi_{N(\SRotTgat)}\mu^{-1}\H{\infty}{\T,\gan,0}(\Div,\om).
\end{align*}
Moreover, there exist smooth $\TRotSgat$ and $\divDivSgan$ \emph{pre-bases}
of $\Harm{}{\S,\gat,\gan,\eps}(\om)$
and smooth $\SRotTgat$ and $\DivTgan$ \emph{pre-bases}
of $\Harm{}{\T,\gat,\gan,\mu}(\om)$, i.e.,
there are linear independent smooth fields 
\begin{align*}
\B{\TRotSgat}(\om)
&:=\{\vB{\TRotSgat}{\ell}\}_{\ell=1}^{d_{\om,\S,\gat}}
\subset\H{\infty}{\S,\gat,0}(\Rot,\om),\\
\B{\divDivSgan}(\om)
&:=\{\vB{\divDivSgan}{\ell}\}_{\ell=1}^{d_{\om,\S,\gat}}
\subset\H{\infty}{\S,\gan,0}(\divDiv,\om),\\
\B{\SRotTgat}(\om)
&:=\{\vB{\SRotTgat}{\ell}\}_{\ell=1}^{d_{\om,\T,\gat}}
\subset\H{\infty}{\T,\gat,0}(\symRot,\om),\\
\B{\DivTgan}(\om)
&:=\{\vB{\DivTgan}{\ell}\}_{\ell=1}^{d_{\om,\T,\gat}}
\subset\H{\infty}{\T,\gan,0}(\Div,\om),
\end{align*}
such that $\pi_{N(\divDivSgan\eps)}\B{\TRotSgat}(\om)$ 
and $\pi_{N(\TRotSgat)}\eps^{-1}\B{\divDivSgan}(\om)$
are both bases of $\Harm{}{\S,\gat,\gan,\eps}(\om)$,
and $\pi_{N(\DivTgan\mu)}\B{\SRotTgat}(\om)$ 
and $\pi_{N(\SRotTgat)}\mu^{-1}\B{\DivTgan}(\om)$
are both bases of $\Harm{}{\T,\gat,\gan,\mu}(\om)$. 
In particular,
\begin{align*}
\Lin\pi_{N(\divDivSgan\eps)}\B{\TRotSgat}(\om)
&=\Harm{}{\S,\gat,\gan,\eps}(\om)
=\Lin\pi_{N(\TRotSgat)}\eps^{-1}\B{\divDivSgan}(\om),\\
\Lin\pi_{N(\DivTgan\mu)}\B{\SRotTgat}(\om)
&=\Harm{}{\T,\gat,\gan,\mu}(\om)
=\Lin\pi_{N(\SRotTgat)}\mu^{-1}\B{\DivTgan}(\om).
\end{align*}
\end{theo}

Note that, e.g., $(1-\pi_{N(\divDivSgan\eps)})$ is the $\L{2}{\S,\eps}(\om)$-orthonormal projector 
onto $R(\SGradgradgat)$.
By \eqref{bih:helmcoho1}, Theorem \ref{bih:theo:rangeselawithoutbdpot},
and Theorem \ref{bih:theo:cohomologyinfty} we compute
\begin{align}
\begin{aligned}
\label{bih:helmcoho2}
\H{}{\S,\gat,0}(\Rot,\om)
&=R(\SGradgradgat)
\oplus_{\L{2}{\S,\eps}(\om)}
\Harm{}{\S,\gat,\gan,\eps}(\om)\\
&=R(\SGradgradgat)
\oplus_{\L{2}{\S,\eps}(\om)}
\Lin\pi_{N(\divDivSgan\eps)}\B{\TRotSgat}(\om)\\
&=R(\SGradgradgat)
+(\pi_{N(\divDivSgan\eps)}-1)\Lin\B{\TRotSgat}(\om)
+\Lin\B{\TRotSgat}(\om)\\
&=R(\SGradgradgat)
+\Lin\B{\TRotSgat}(\om),\\
\H{k}{\S,\gat,0}(\Rot,\om)
&=R(\SGradgradgat)\cap\H{k}{\S,\gat,0}(\Rot,\om)
+\Lin\B{\TRotSgat}(\om),\\
&=R(\SGradgradgat^{k})
+\Lin\B{\TRotSgat}(\om).
\end{aligned}
\end{align}

Similarly, we obtain decompositions of 
$\H{k}{\S,\gan,0}(\divDiv,\om)$, $\H{}{\T,\gat,0}(\symRot,\om)$,
and $\H{k}{\T,\gan,0}(\Div,\om)$ 
using $\B{\divDivSgan}(\om)$, $\B{\SRotTgat}(\om)$, and $\B{\DivTgan}(\om)$, 
respectively. We conclude:

\begin{theo}[bounded regular direct decompositions]
\label{bih:theo:highorderregdecoinfty}
Let $k\geq0$. 
Then the bounded regular direct decompositions
\begin{align*}
\H{k}{\S,\gat}(\Rot,\om)
&=R(\widetilde\PotQ_{\TRotS,\gat}^{k,1})
\dotplus\H{k}{\S,\gat,0}(\Rot,\om),\\
\H{k}{\S,\gat,0}(\Rot,\om)
&=\Gradgrad\H{k+2}{\gat}(\om)
\dotplus\Lin\B{\TRotSgat}(\om),\\
\H{k}{\T,\gan}(\Div,\om)
&=R(\widetilde\PotQ_{\DivT,\gan}^{k,1})
\dotplus\H{k}{\T,\gan,0}(\Div,\om),\\
\H{k}{\T,\gan,0}(\Div,\om)
&=\Rot\H{k+1}{\S,\gan}(\om)
\dotplus\Lin\B{\DivTgan}(\om),\\
\H{k}{\T,\gat}(\symRot,\om)
&=R(\widetilde\PotQ_{\SRotT,\gat}^{k,1})
\dotplus\H{k}{\T,\gat,0}(\symRot,\om),\\
\H{k}{\T,\gat,0}(\symRot,\om)
&=\devGrad\H{k+1}{\gat}(\om)
\dotplus\Lin\B{\SRotTgat}(\om),\\
\H{k}{\S,\gan}(\divDiv,\om)
&=R(\widetilde\PotQ_{\divDivS,\gan}^{k,1})
\dotplus\H{k}{\S,\gan,0}(\divDiv,\om),\\
\H{k+1,k}{\S,\gan}(\divDiv,\om)
&=R(\widetilde\PotQ_{\divDivS,\gan}^{k+1,k,1})
\dotplus\H{k+1}{\S,\gan,0}(\divDiv,\om),\\
\H{k}{\S,\gan,0}(\divDiv,\om)
&=\symRot\H{k+1}{\T,\gan}(\om)
\dotplus\Lin\B{\divDivSgan}(\om)
\end{align*}
hold. Note that 
$R(\widetilde\PotQ_{\TRotS,\gat}^{k,1})\subset\H{k+1}{\S,\gat}(\om)$,
$R(\widetilde\PotQ_{\DivT,\gan}^{k,1})\subset\H{k+1}{\T,\gan}(\om)$,
$R(\widetilde\PotQ_{\SRotT,\gat}^{k,1})\subset\H{k+1}{\T,\gat}(\om)$,
and $R(\widetilde\PotQ_{\divDivS,\gan}^{k,1}),R(\widetilde\PotQ_{\divDivS,\gan}^{k+1,k,1})\subset\H{k+2}{\S,\gan}(\om)$.
\end{theo}

See Appendix \ref{sec:someproof} for a proof.

\begin{rem}[bounded regular direct decompositions]
\label{bih:rem:highorderregdecoinfty}
In particular, for $k=0$
\begin{align*}
\H{}{\S,\gat}(\Rot,\om)
&=R(\widetilde\PotQ_{\TRotS,\gat}^{0,1})
\dotplus\H{}{\S,\gat,0}(\Rot,\om),\\
\H{}{\S,\gat,0}(\Rot,\om)
&=\Gradgrad\H{2}{\gat}(\om)
\dotplus\Lin\B{\TRotSgat}(\om)\\
&=\Gradgrad\H{2}{\gat}(\om)
\oplus_{\L{2}{\S,\eps}(\om)}
\Harm{}{\S,\gat,\gan,\eps}(\om),\\
\H{}{\T,\gan}(\Div,\om)
&=R(\widetilde\PotQ_{\DivT,\gan}^{0,1})
\dotplus\H{}{\T,\gan,0}(\Div,\om),\\
\mu^{-1}\H{0}{\T,\gan,0}(\Div,\om)
&=\mu^{-1}\Rot\H{1}{\S,\gan}(\om)
\dotplus\mu^{-1}\Lin\B{\DivTgan}(\om)\\
&=\mu^{-1}\Rot\H{1}{\S,\gan}(\om)
\oplus_{\L{2}{\T,\mu}(\om)}
\Harm{}{\T,\gat,\gan,\mu}(\om),\\
\H{}{\T,\gat}(\symRot,\om)
&=R(\widetilde\PotQ_{\SRotT,\gat}^{0,1})
\dotplus\H{}{\T,\gat,0}(\symRot,\om),\\
\H{}{\T,\gat,0}(\symRot,\om)
&=\devGrad\H{1}{\gat}(\om)
\dotplus\Lin\B{\SRotTgat}(\om)\\
&=\devGrad\H{1}{\gat}(\om)
\oplus_{\L{2}{\T,\mu}(\om)}
\Harm{}{\T,\gat,\gan,\mu}(\om),\\
\H{}{\S,\gan}(\divDiv,\om)
&=R(\widetilde\PotQ_{\divDivS,\gan}^{0,1})
\dotplus\H{}{\S,\gan,0}(\divDiv,\om),\\
\eps^{-1}\H{}{\S,\gan,0}(\divDiv,\om)
&=\eps^{-1}\symRot\H{1}{\T,\gan}(\om)
\dotplus\eps^{-1}\Lin\B{\divDivSgan}(\om)\\
&=\eps^{-1}\symRot\H{1}{\T,\gan}(\om)
\oplus_{\L{2}{\S,\eps}(\om)}
\Harm{}{\S,\gat,\gan,\eps}(\om),
\intertext{and}
\L{2}{\S,\eps}(\om)
&=\H{}{\S,\gat,0}(\Rot,\om)
\oplus_{\L{2}{\S,\eps}(\om)}
\eps^{-1}\symRot\H{1}{\T,\gan}(\om)\\
&=\Gradgrad\H{2}{\gat}(\om)
\oplus_{\L{2}{\S,\eps}(\om)}
\eps^{-1}\H{}{\S,\gan,0}(\divDiv,\om),\\
\L{2}{\T,\mu}(\om)
&=\H{}{\T,\gat,0}(\symRot,\om)
\oplus_{\L{2}{\T,\mu}(\om)}
\mu^{-1}\Rot\H{1}{\S,\gan}(\om)\\
&=\devGrad\H{1}{\gat}(\om)
\oplus_{\L{2}{\T,\mu}(\om)}
\mu^{-1}\H{}{\T,\gan,0}(\Div,\om).
\end{align*}
\end{rem}

By the latter theorem we have bounded linear regular (direct) decompositions
\begin{align}
\label{bih:rem:highorderregdecoinfty2beforerem}
\begin{aligned}
\H{k}{\S,\gat}(\Rot,\om)
&=R(\widetilde\PotQ_{\TRotS,\gat}^{k,1})
\dotplus\Lin\B{\TRotSgat}(\om)
\dotplus\Gradgrad\H{k+2}{\gat}(\om)\\
&=\H{k+1}{\S,\gat}(\om)
+\Gradgrad\H{k+2}{\gat}(\om),\\
\H{k}{\T,\gan}(\Div,\om)
&=R(\widetilde\PotQ_{\DivT,\gan}^{k,1})
\dotplus\Lin\B{\DivTgan}(\om)
\dotplus\Rot\H{k+1}{\S,\gan}(\om)\\
&=\H{k+1}{\T,\gan}(\om)
+\Rot\H{k+1}{\S,\gan}(\om),\\
\H{k}{\T,\gat}(\symRot,\om)
&=R(\widetilde\PotQ_{\SRotT,\gat}^{k,1})
\dotplus\Lin\B{\SRotTgat}(\om)
\dotplus\devGrad\H{k+1}{\gat}(\om)\\
&=\H{k+1}{\T,\gat}(\om)
+\devGrad\H{k+1}{\gat}(\om),\\
\H{k}{\S,\gan}(\divDiv,\om)
&=R(\widetilde\PotQ_{\divDivS,\gan}^{k,1})
\dotplus\Lin\B{\divDivSgan}(\om)
\dotplus\symRot\H{k+1}{\T,\gan}(\om)\\
&=\H{k+2}{\S,\gan}(\om)
+\symRot\H{k+1}{\T,\gan}(\om),\\
\H{k+1,k}{\S,\gan}(\divDiv,\om)
&=R(\widetilde\PotQ_{\divDivS,\gan}^{k+1,k,1})
\dotplus\Lin\B{\divDivSgan}(\om)
\dotplus\symRot\H{k+2}{\T,\gan}(\om)\\
&=\H{k+2}{\S,\gan}(\om)
+\symRot\H{k+2}{\T,\gan}(\om),
\end{aligned}
\end{align}
see Remark \ref{bih:rem:highorderregdecoinfty2} for more details
on these decompositions and the corresponding 
bounded linear regular direct decomposition operators.
Noting
\begin{align}
\begin{aligned}
\label{bih:Bortho}
R(\eps^{-1}\SRotTgan)&\bot_{\L{2}{\S,\eps}(\om)}\B{\TRotSgat}(\om),
&
R(\SGradgradgat)&\bot_{\L{2}{\S}(\om)}\B{\divDivSgan}(\om),\\
R(\mu^{-1}\TRotSgan)&\bot_{\L{2}{\T,\mu}(\om)}\B{\SRotTgat}(\om),
&
R(\TGradgat)&\bot_{\L{2}{\T}(\om)}\B{\DivTgan}(\om)
\end{aligned}
\end{align}
we see:

\begin{theo}[alternative Dirichlet/Neumann projections]
\label{bih:theo:Bcoho1}
It holds
\begin{align*}
\Harm{}{\S,\gat,\gan,\eps}(\om)\cap\B{\TRotSgat}(\om)^{\bot_{\L{2}{\S,\eps}(\om)}}
&=\{0\},\\
N(\divDivSgan\eps)\cap\B{\TRotSgat}(\om)^{\bot_{\L{2}{\S,\eps}(\om)}}
&=R(\eps^{-1}\SRotTgan),\\
\Harm{}{\S,\gat,\gan,\eps}(\om)\cap\B{\divDivSgan}(\om)^{\bot_{\L{2}{\S}(\om)}}
&=\{0\},\\
N(\TRotSgat)\cap\B{\divDivSgan}(\om)^{\bot_{\L{2}{\S}(\om)}}
&=R(\SGradgradgat),\\
\Harm{}{\T,\gat,\gan,\mu}(\om)\cap\B{\SRotTgat}(\om)^{\bot_{\L{2}{\T,\mu}(\om)}}
&=\{0\},\\
N(\DivTgan\eps)\cap\B{\SRotTgat}(\om)^{\bot_{\L{2}{\T,\mu}(\om)}}
&=R(\mu^{-1}\TRotSgan),\\
\Harm{}{\T,\gat,\gan,\mu}(\om)\cap\B{\DivTgan}(\om)^{\bot_{\L{2}{\T}(\om)}}
&=\{0\},\\
N(\SRotTgat)\cap\B{\DivTgan}(\om)^{\bot_{\L{2}{\T}(\om)}}
&=R(\TGradgat).
\end{align*}
Moreover, for all $k\geq0$
\begin{align*}
N(\divDivSgan^{k}\eps)\cap\B{\TRotSgat}(\om)^{\bot_{\L{2}{\S,\eps}(\om)}}
&=R(\eps^{-1}\SRotTgan^{k})
=\eps^{-1}\symRot\H{k+1}{\T,\gan}(\om),\\
N(\TRotSgat^{k})\cap\B{\divDivSgan}(\om)^{\bot_{\L{2}{\S}(\om)}}
&=R(\SGradgradgat^{k})
=\Gradgrad\H{k+2}{\gat}(\om),\\
N(\DivTgan^{k}\eps)\cap\B{\SRotTgat}(\om)^{\bot_{\L{2}{\T,\mu}(\om)}}
&=R(\mu^{-1}\TRotSgan^{k})
=\mu^{-1}\Rot\H{k+1}{\S,\gan}(\om),\\
N(\SRotTgat^{k})\cap\B{\DivTgan}(\om)^{\bot_{\L{2}{\T}(\om)}}
&=R(\TGradgat^{k})
=\devGrad\H{k+1}{\gat}(\om).
\end{align*}
\end{theo}

See Appendix \ref{sec:someproof} for a proof.
Theorem \ref{bih:theo:highorderregdecoinfty} implies:

\begin{theo}[cohomology groups]
\label{bih:theo:Bcoho2}
It holds
\begin{align*}
\frac{N(\TRotSgat^{k})}{R(\SGradgradgat^{k})}
\cong\Lin\B{\TRotSgat}(\om)
&\cong\Harm{}{\S,\gat,\gan,\eps}(\om)
\cong\Lin\B{\divDivSgan}(\om)
\cong\frac{N(\divDivSgan^{k})}{R(\SRotTgan^{k})},\\
\frac{N(\SRotTgat^{k})}{R(\TGradgat^{k})}
\cong\Lin\B{\SRotTgat}(\om)
&\cong\Harm{}{\T,\gat,\gan,\mu}(\om)
\cong\Lin\B{\DivTgan}(\om)
\cong\frac{N(\DivTgan^{k})}{R(\TRotSgan^{k})}.
\end{align*}
In particular, the dimensions of the cohomology groups
(Dirichlet/Neumann fields) are independent of $k$ and $\eps$, $\mu$ and it holds
\begin{align*}
d_{\om,\S,\gat}
&=\dim\big(N(\TRotSgat^{k})/R(\SGradgradgat^{k})\big)
=\dim\big(N(\divDivSgan^{k})/R(\SRotTgan^{k})\big),\\
d_{\om,\T,\gat}
&=\dim\big(N(\SRotTgat^{k})/R(\TGradgat^{k})\big)
=\dim\big(N(\DivTgan^{k})/R(\TRotSgan^{k})\big).
\end{align*}
\end{theo}


\bibliographystyle{plain} 
\bibliography{/Users/paule/Documents/TeX/input/bibTex/psz}


%

\appendix

\section{Elementary Formulas}
\label{bih:app:sec:formulas}

From \cite{PZ2020a,PZ2021a} and \cite{PW2021a} we 
have the following collection of formulas related to the elasticity and 
the biharmonic complex.

\begin{lem}[{\cite[Lemma 12.10]{PW2021a}}]
\label{bih:lem:PZformulalem}
Let $u$, $v$, $w$, and $S$ belong to $\C{\infty}{}(\reals^{3})$.
\begin{itemize}
\item
$(\spn v)\,w=v\times w=-(\spn w)\,v$ 
\quad and  
\quad
$(\spn v)(\spn^{-1}S)=-Sv$, if $\sym S=0$
\item
$\sym\spn v=0$
\quad and \quad
$\dev(u\id)=0$
\item
$\tr\Grad v=\div v$
\quad and \quad
$2\skw\Grad v=\spn\rot v$
\item
$\Div(u\id)=\grad u$
\quad and \quad
$\Rot(u\id)=-\spn\grad u$,\\
in particular, \quad $\rot\Div(u\id)=0$ 
\quad and \quad $\rot\spn^{-1}\Rot(u\id)=0$\\
and \quad $\sym\Rot(u\id)=0$
\item
$\Div\spn v=-\rot v$
\quad and \quad
$\Div\skw S=-\rot\spn^{-1}\skw S$,\\
in particular, \quad $\div\Div\skw S=0$
\item
$\Rot\spn v=(\div v)\id-(\Grad v)^{\top}$\\
and \quad $\Rot\skw S=(\div\spn^{-1}\skw S)\id-(\Grad\spn^{-1}\skw S)^{\top}$
\item
$\dev\Rot\spn v=-(\dev\Grad v)^{\top}$
\item
$-2\Rot\sym\Grad v=2\Rot\skw\Grad v=-(\Grad\rot v)^{\top}$
\item
$2\spn^{-1}\skw\Rot S=\Div S^{\top}-\grad\tr S=\Div\big(S-(\tr S)\id\big)^{\top}$,\\
in particular, \quad $\rot\Div S^{\top}=2\rot\spn^{-1}\skw\Rot S$\\
and \quad $2\skw\Rot S=\spn\Div S^{\top}$, if $\tr S=0$
\item
$\tr\Rot S=2\div\spn^{-1}\skw S$,
\quad in particular, \quad $\tr\Rot S=0$, if $\skw S=0$,\\
and \quad $\tr\Rot\sym S=0$ \quad and \quad $\tr\Rot\skw S=\tr\Rot S$
\item
$2(\Grad\spn^{-1}\skw S)^{\top}=(\tr\Rot\skw S)\id-2\Rot\skw S$
\item
$3\Div(\dev\Grad v)^{\top}=2\grad\div v$
\item
$2\Rot\sym\Grad v=-2\Rot\skw\Grad v=-\Rot\spn\rot v=(\Grad\rot v)^{\top}$
\item
$2\Div\sym\Rot S=-2\Div\skw\Rot S=\rot\Div S^{\top}$
\item
$\Rot(\Rot\sym S)^{\top}=\sym\Rot(\Rot S)^{\top}$
\item
$\Rot(\Rot\skw S)^{\top}=\skw\Rot(\Rot S)^{\top}$
\end{itemize}
All formulas extend also to distributions.
\end{lem}

\section{Biharmonic Complex Operators Revisited}
\label{bih:app:sec:revop}

Let $\top$ denote the formal operator of matrix transposition, i.e.,
$$\top S:=S^{\top},$$
and define
$$\spn:\reals^{3}\to\reals^{3\times3}_{\skw};\,
\begin{bmatrix}a_{1}\\a_{2}\\a_{3}\end{bmatrix}
\mapsto\begin{bmatrix}0&-a_{3}&a_{2}\\a_{3}&0&-a_{1}\\-a_{2}&a_{1}&0\end{bmatrix}.$$

We recall the operators forming the de Rham complex (classical vector analysis)
$\grad$, $\rot$, and $\div$ acting on functions and vector fields, respectively,
as formal matrix operators
\begin{align*}
\grad&:=\begin{bmatrix}\p_{1}\\\p_{2}\\\p_{3}\end{bmatrix},
&
\rot&:=\spn\grad=\begin{bmatrix}0&-\p_{3}&\p_{2}\\\p_{3}&0&-\p_{1}\\-\p_{2}&\p_{1}&0\end{bmatrix},
&
\div&:=\top\grad=\begin{bmatrix}\p_{1}&\p_{2}&\p_{3}\end{bmatrix}.
\end{align*}
Moreover, we introduce their relatives from the vector de Rham complex
acting on vector and tensor fields, respectively, as formal matrix operators
\begin{align*}
\Grad&:=\top\grad\top,
&
\Rot&:=\top\rot\top,
&
\Div&:=\top\div\top.
\end{align*}
In words, $\Grad$, $\Rot$, and $\Div$ act row-wise as the operators 
$\grad$, $\rot$, and $\div$ from the classical de Rham complex.
Note that $\Grad v$ is just the Jacobian for a vector field $v$.

Let
$$\iota_{\S}:\reals^{3\times3}_{\sym}\to\reals^{3\times3},\qquad
\iota_{\T}:\reals^{3\times3}_{\dev}\to\reals^{3\times3}$$
denote the canonical embedding of symmetric and deviatoric (trace free) $(3\times3)$-matrices 
into the arbitrary $(3\times3)$-matrices, respectively.
Then the adjoints 
$$\iota_{\S}^{*}:\reals^{3\times3}\to\reals^{3\times3}_{\sym},\qquad
\iota_{\T}^{*}:\reals^{3\times3}\to\reals^{3\times3}_{\dev}$$
are almost the projectors onto symmetric and deviatoric $(3\times3)$-matrices, respectively, i.e.,
the actual projectors are given by 
$$\sym:=\iota_{\S}\iota_{\S}^{*}:\reals^{3\times3}\to\reals^{3\times3};\,S\mapsto\frac{1}{2}(S+S^{\top}),\quad
\dev:=\iota_{\T}\iota_{\T}^{*}:\reals^{3\times3}\to\reals^{3\times3};\,T\mapsto T-\frac{1}{3}(\tr T)\id.$$
We extend all the latter formal operators to $\L{2}{}(\om)$-tensor fields.

In the light of this, in the biharmonic complexes we are dealing with the operators
\begin{align*}
\SGradgrad
&:=\iota_{\S}^{*}\Grad\grad,
&
\TRotS
&:=\iota_{\T}^{*}\Rot\iota_{\S},
&
\DivT
&:=\Div\iota_{\T},\\
\divDivS
&:=\div\Div\iota_{\S},
&
\SRotT
&:=\iota_{\S}^{*}\Rot\iota_{\T},
&
\TGrad
&:=\iota_{\T}^{*}\Grad.
\end{align*}
Note that
\begin{align*}
\iota_{\S}\SGradgrad
&=\sym\Gradgrad
=\Gradgrad
=\top\grad\top\grad,\\
\iota_{\T}\TRotS
&=\dev\RotS
=\Rot\iota_{\S}
=\top\rot\top\iota_{\S}
=:\RotS,\\
\DivT
&=\top\div\top\iota_{\T},\\
\iota_{\T}\prescript{}{\T}{\Grad}
&=\devGrad
=\dev\top\grad\top,\\
\SRotT
&=\symRotT
=\sym\top\rot\top\iota_{\T},\\
\divDivS
&=\div\top\div\top\iota_{\S},
\end{align*}
in particular, on symmetric tensor fields we have
$\TRotS=\devRot=\Rot$, cf.~\cite[Lemma A.1]{PZ2021a}.
Using these formal operators we introduce 
their maximal $\L{2}{}(\om)$-realisations, i.e., 
\begin{align*}
\SGradgrad:D(\SGradgrad)\subset\L{2}{}(\om)
&\to\L{2}{\S}(\om),
&
u&\mapsto\Gradgrad u,\\
\TRotS:D(\TRotS)\subset\L{2}{\S}(\om)
&\to\L{2}{\T}(\om),
&
S&\mapsto\Rot S,\\
\DivT:D(\DivT)\subset\L{2}{\T}(\om)
&\to\L{2}{}(\om),
&
T&\mapsto\Div T,\\
\TGrad:D(\TGrad)\subset\L{2}{}(\om)
&\to\L{2}{\T}(\om),
&
v&\mapsto\devGrad v,\\
\SRotT:D(\SRotT)\subset\L{2}{\T}(\om)
&\to\L{2}{\S}(\om),
&
T&\mapsto\symRot T,\\
\divDivS:D(\divDivS)\subset\L{2}{\S}(\om)
&\to\L{2}{}(\om),
&
S&\mapsto\divDiv S,
\end{align*}
which are densely defined and closed (unbounded) linear operators 
and form the two (formally primal and dual) biharmonic complexes
\begin{equation*}
\def\arrowlength{5ex}
\def\arrowdistance{0}
\begin{tikzcd}[column sep=\arrowlength]
\cdots 
\arrow[r, rightarrow, shift left=\arrowdistance, "\cdots"] 
& 
\L{2}{}(\om) 
\ar[r, rightarrow, shift left=\arrowdistance, "\prescript{}{\S}{\Gradgrad}"] 
& 
[2.5em]
\L{2}{\S}(\om)
\arrow[r, rightarrow, shift left=\arrowdistance, "\prescript{}{\T}{\Rot}_{\S}"] 
& 
[2.5em]
\L{2}{\T}(\om)
\arrow[r, rightarrow, shift left=\arrowdistance, "\Div_{\T}"] 
& 
[1em]
\L{2}{}(\om)
\arrow[r, rightarrow, shift left=\arrowdistance, "\cdots"] 
&
\cdots,
\end{tikzcd}
\end{equation*}
\begin{equation*}
\def\arrowlength{5ex}
\def\arrowdistance{0}
\begin{tikzcd}[column sep=\arrowlength]
\cdots 
\arrow[r, leftarrow, shift left=\arrowdistance, "\cdots"] 
& 
\L{2}{}(\om) 
\ar[r, leftarrow, shift left=\arrowdistance, "\divDiv_{\S}"] 
& 
[2.5em]
\L{2}{\S}(\om)
\arrow[r, leftarrow, shift left=\arrowdistance, "\prescript{}{\S}{\Rot}_{\T}"] 
& 
[2.5em]
\L{2}{\T}(\om)
\arrow[r, leftarrow, shift left=\arrowdistance, "\prescript{}{\T}{\Grad}"] 
& 
[1em]
\L{2}{}(\om)
\arrow[r, leftarrow, shift left=\arrowdistance, "\cdots"] 
&
\cdots,
\end{tikzcd}
\end{equation*}
cf. \cite{PZ2020a} for the complex properties.

Finally, the operators 
$$\SGradgradgat,\quad
\TRotSgat,\quad
\DivTgat,\quad
\TGradgat,\quad
\SRotTgat,\quad
\divDivSgat$$
from Section \ref{bih:sec:bihop} are the restrictions of 
$$\SGradgrad,\quad
\TRotS,\quad
\DivT,\quad
\TGrad,\quad
\SRotT,\quad
\divDivS$$
to their domains of defintion 
$$D(\SGradgradgat),\quad
D(\TRotSgat),\quad
D(\DivTgat),\quad
D(\TGradgat),\quad
D(\SRotTgat),\quad
D(\divDivSgat)$$
which are the closures of $\C{\infty}{\gat}(\om)$, $\C{\infty}{\S,\gat}(\om)$, 
and $\C{\infty}{\T,\gat}(\om)$ in the corresponding graph norms, respectively.

\section{Some Proofs}
\label{sec:someproof}%

\begin{proof}[Proof of Theorem \ref{bih:highorderregpotextdombih}]
In \cite[Theorem 3.10]{PZ2020a} we have shown the stated results for $\gat=\ga$ and $\gat=\emptyset$,
which is also a crucial ingredient of this proof.
Note that in these two special cases always ``\emph{strong $=$ weak}'' holds
as $\A_{n}^{**}=\ol{\A_{n}}=\A_{n}$, and that this argument fails in the remaining cases
of mixed boundary conditions. Therefore, let $\emptyset\varsubsetneq\gat\varsubsetneq\ga$.
Moreover, recall the notion of an extendable domain from \cite[Section 3]{PS2021b}.
In particular, $\widehat{\om}$ and the extended domain 
$\widetilde{\om}$ are topologically trivial.
\begin{itemize}
\item
Let $S\in\bH{k}{\S,\gat,0}(\Rot,\om)$. 
By definition, $S$ can be extended through $\gat$
by zero to the larger domain $\widetilde{\om}$ yielding
$$\widetilde{S}\in\bH{k}{\S,\emptyset,0}(\Rot,\widetilde{\om})
=\bH{k}{\S,0}(\Rot,\widetilde{\om})
=\H{k}{\S,0}(\Rot,\widetilde{\om}).$$
By \cite[Theorem 3.10, Remark 3.11]{PZ2020a} 
and Stein's or Calderon's extension theorem
-- see also \cite[Lemma 4.3, Lemma 4.4]{PS2021b} for the fact that 
the respective potentials are already defined on the whole of $\rt$ --
there exists $\widetilde{u}\in\H{k+2}{}(\rt)$
such that $\Gradgrad\widetilde{u}=\widetilde{S}$ in $\widetilde{\om}$.
Since $\widetilde{S}=0$ in $\widehat{\om}$, $\widetilde{u}$ must be a polynomial
$p\in\Pone$ in $\widehat{\om}$. Far outside of $\widetilde{\om}$ we modify $p$
by a cut-off function such that the resulting function $\widetilde{p}$ 
is compactly supported and $\widetilde{p}|_{\widetilde{\om}}=p$.
Note that $\widetilde{p}$ depends continuously on $S$
by Poincar\'e's estimate.
Then $u:=\widetilde{u}-\widetilde{p}\in\H{k+2}{}(\rt)$
with $u|_{\widehat{\om}}=0$. Hence $u$ belongs to $\H{k+2}{\gat}(\om)$
and depends continuously on $S$. Moreover, $u$ satisfies
$\Gradgrad u=\Gradgrad\widetilde{u}=\widetilde{S}$ in $\widetilde{\om}$,
in particular $\Gradgrad u=S$ in $\om$.
We put $\PotP_{\SGradgrad,\gat}^{k}S:=u\in\H{k+2}{\gat}(\om)$.
\item
Let $T\in\bH{k}{\T,\gat,0}(\Div,\om)$. By definition, $T$ can be extended through $\gat$
by zero to $\widetilde{\om}$ giving
$$\widetilde{T}\in\bH{k}{\T,\emptyset,0}(\Div,\widetilde{\om})
=\bH{k}{\T,0}(\Div,\widetilde{\om})
=\H{k}{\T,0}(\Div,\widetilde{\om}).$$
By \cite[Theorem 3.10]{PZ2020a} there exists $\widetilde{S}\in\H{k+1}{\S}(\rt)$
such that $\Rot\widetilde{S}=\widetilde{T}$ in $\widetilde{\om}$.
Since $\widetilde{T}=0$ in $\widehat{\om}$, i.e.,
$\widetilde{S}|_{\widehat{\om}}\in\H{k+1}{\S,0}(\Rot,\widehat{\om})$,
we get again by \cite[Theorem 3.10]{PZ2020a}
(or the first part of this proof) $\widetilde{u}\in\H{k+3}{}(\rt)$
such that $\Gradgrad\widetilde{u}=\widetilde{S}$ in $\widehat{\om}$.
Then $S:=\widetilde{S}-\Gradgrad\widetilde{u}$
belongs to $\H{k+1}{\S}(\rt)$ and satisfies $S|_{\widehat{\om}}=0$.
Thus $S\in\H{k+1}{\S,\gat}(\om)$ and depends continuously on $T$.
Furthermore, $\Rot S=\Rot\widetilde{S}=\widetilde{T}$ in $\widetilde{\om}$,
in particular $\Rot S=T$ in $\om$.
We set $\PotP_{\TRotS,\gat}^{k}T:=S\in\H{k+1}{\S,\gat}(\om)$.
\item
Let $v\in\H{k}{\gat}(\om)$. By definition, $v$ can be extended through $\gat$
by zero to $\widetilde{\om}$ defining $\widetilde{v}\in\H{k}{}(\widetilde{\om})$.
\cite[Theorem 3.10]{PZ2020a} yields $\widetilde{T}\in\H{k+1}{\T}(\rt)$
such that $\Div\widetilde{T}=\widetilde{v}$ in $\widetilde{\om}$.
As $\widetilde{v}=0$ in $\widehat{\om}$, i.e.,
$\widetilde{T}|_{\widehat{\om}}\in\H{k+1}{\T,0}(\Div,\widehat{\om})$,
we get again by \cite[Theorem 3.10]{PZ2020a}
(or the second part of this proof) $\widetilde{S}\in\H{k+2}{\S}(\rt)$
such that $\Rot\widetilde{S}=\widetilde{T}$ holds in $\widehat{\om}$.
Then $T:=\widetilde{T}-\Rot\widetilde{S}$
belongs to $\H{k+1}{\T}(\rt)$ with $T|_{\widehat{\om}}=0$.
Hence $T$ belongs to $\H{k+1}{\T,\gat}(\om)$ and depends continuously on $v$.
Furthermore, $\Div T=\Div\widetilde{T}=\widetilde{v}$ in $\widetilde{\om}$,
in particular $\Div T=v$ in $\om$.
Finally, we define $\PotP_{\DivT,\gat}^{k}v:=T\in\H{k+1}{\T,\gat}(\om)$.
\item
Let $T\in\bH{k}{\T,\gat,0}(\symRot,\om)$. 
By definition, $T$ can be extended through $\gat$
by zero to $\widetilde{\om}$ yielding
$$\widetilde{T}\in\bH{k}{\T,\emptyset,0}(\symRot,\widetilde{\om})
=\bH{k}{\T,0}(\symRot,\widetilde{\om})
=\H{k}{\T,0}(\symRot,\widetilde{\om}).$$
By \cite[Theorem 3.10]{PZ2020a} there exists $\widetilde{v}\in\H{k+1}{}(\rt)$
such that $\devGrad\widetilde{v}=\widetilde{T}$ in $\widetilde{\om}$.
Since $\widetilde{T}=0$ in $\widehat{\om}$, $\widetilde{v}$ must be a Raviart-Thomas field
$r\in\RT$ in $\widehat{\om}$. Far outside of $\widetilde{\om}$ we modify $r$
by a cut-off function such that the resulting vector  field $\widetilde{r}$ 
is compactly supported and $\widetilde{r}|_{\widetilde{\om}}=r$.
Then $v:=\widetilde{v}-\widetilde{r}\in\H{k+1}{}(\rt)$
with $v|_{\widehat{\om}}=0$. Hence $v$ belongs to $\H{k+1}{\gat}(\om)$
and depends continuously on $T$. Moreover, $v$ satisfies
$\devGrad v=\devGrad\widetilde{v}=\widetilde{T}$ in $\widetilde{\om}$,
in particular $\devGrad v=T$ in $\om$.
We put $\PotP_{\TGrad,\gat}^{k}T:=v\in\H{k+1}{\gat}(\om)$.
\item
Let $S\in\bH{k}{\S,\gat,0}(\divDiv,\om)$. By definition, $S$ can be extended through $\gat$
by zero to $\widetilde{\om}$ giving
$$\widetilde{S}\in\bH{k}{\S,\emptyset,0}(\divDiv,\widetilde{\om})
=\bH{k}{\S,0}(\divDiv,\widetilde{\om})
=\H{k}{\S,0}(\divDiv,\widetilde{\om}).$$
By \cite[Theorem 3.10]{PZ2020a} there exists $\widetilde{T}\in\H{k+1}{\T}(\rt)$
such that $\symRot\widetilde{T}=\widetilde{S}$ in $\widetilde{\om}$.
Since $\widetilde{S}=0$ in $\widehat{\om}$, i.e.,
$\widetilde{T}|_{\widehat{\om}}\in\H{k+1}{\T,0}(\symRot,\widehat{\om})$,
we get again by \cite[Theorem 3.10]{PZ2020a}
(or the fourth part of this proof) $\widetilde{v}\in\H{k+2}{}(\rt)$
such that $\devGrad\widetilde{v}=\widetilde{T}$ in $\widehat{\om}$.
Then $T:=\widetilde{T}-\devGrad\widetilde{v}$
belongs to $\H{k+1}{\T}(\rt)$ and satisfies $T|_{\widehat{\om}}=0$.
Thus $T\in\H{k+1}{\T,\gat}(\om)$ and depends continuously on $S$.
Furthermore, $\symRot T=\symRot\widetilde{T}=\widetilde{S}$ in $\widetilde{\om}$,
in particular $\symRot T=S$ in $\om$.
We set $\PotP_{\SRotT,\gat}^{k}S:=T\in\H{k+1}{\T,\gat}(\om)$.
\item
Let $u\in\H{k}{\gat}(\om)$. By definition, $u$ can be extended through $\gat$
by zero to $\widetilde{\om}$ defining $\widetilde{u}\in\H{k}{}(\widetilde{\om})$.
\cite[Theorem 3.10]{PZ2020a} yields $\widetilde{S}\in\H{k+2}{\S}(\rt)$
such that $\divDiv\widetilde{S}=\widetilde{u}$ in $\widetilde{\om}$.
As $\widetilde{u}=0$ in $\widehat{\om}$, i.e.,
$\widetilde{S}|_{\widehat{\om}}\in\H{k+2}{\S,0}(\divDiv,\widehat{\om})$,
we get again by \cite[Theorem 3.10]{PZ2020a}
(or the fifth part of this proof) $\widetilde{T}\in\H{k+3}{\T}(\rt)$
such that $\symRot\widetilde{T}=\widetilde{S}$ holds in $\widehat{\om}$.
Then $S:=\widetilde{S}-\symRot\widetilde{T}$
belongs to $\H{k+2}{\S}(\rt)$ with $S|_{\widehat{\om}}=0$.
Hence $S$ belongs to $\H{k+2}{\S,\gat}(\om)$ and depends continuously on $u$.
Furthermore, $\divDiv S=\divDiv\widetilde{S}=\widetilde{u}$ in $\widetilde{\om}$,
in particular $\divDiv S=u$ in $\om$.
Finally, we define $\PotP_{\divDivS,\gat}^{k}u:=S\in\H{k+2}{\S,\gat}(\om)$.
\end{itemize}
The assertion about the compact supports is trivial.
\end{proof}

\begin{proof}[Proof of Lemma \ref{bih:lem:highorderregdecobih} and Corollary \ref{bih:cor:weakstrongbih}.]
According to \cite[Section 4.2]{BPS2016a},
cf.~\cite[Section 4.2]{BPS2019a}, \cite[Lemma 3.1]{PS2021b}, \cite{PS2021d}, or \cite{PZ2021a},
let $(U_{\ell},\varphi_{\ell})$ be a partition of unity for $\om$, such that
$$\om=\bigcup_{\ell=-L}^{L}\om_{\ell},\qquad
\om_{\ell}:=\om\cap U_{\ell},\qquad
\varphi_{\ell}\in\C{\infty}{\p U_{\ell}}(U_{\ell}),$$
and such that $(\om_{\ell},\widehat\ga_{t,\ell})$ are extendable bounded strong Lipschitz pairs.
Recall 
$$\Sigma_{\ell}:=\p\om_{\ell}\setminus\ga,\qquad
\ga_{t,\ell}:=\gat\cap U_{\ell},\qquad
\widehat\ga_{t,\ell}:=\mathrm{int}(\ga_{t,\ell}\cup\ol{\Sigma}_{\ell}).$$
\begin{itemize}
\item
Let $k\geq0$ and let $S\in\bH{k}{\S,\gat}(\Rot,\om)$. Then by definition
$S|_{\om_{\ell}}\in\bH{k}{\S,\ga_{t,\ell}}(\Rot,\om_{\ell})$ 
and we decompose by Corollary \ref{bih:highorderregdecoextdombihcor}
$$S|_{\om_{\ell}}=S_{\ell,1}
+\Gradgrad u_{\ell,0}$$
with 
$S_{\ell,1}:=\PotQ_{\TRotS,\ga_{t,\ell}}^{k,1}S|_{\om_{\ell}}\in\H{k+1}{\S,\ga_{t,\ell}}(\om_{\ell})$
and 
$u_{\ell,0}:=\PotQ_{\TRotS,\ga_{t,\ell}}^{k,0}S|_{\om_{\ell}}\in\H{k+2}{\ga_{t,\ell}}(\om_{\ell})$.
Lemma \ref{bih:lem:cutlem} yields
\begin{align*}
\varphi_{\ell}S|_{\om_{\ell}}
&=\varphi_{\ell}S_{\ell,1}
+\varphi_{\ell}\Gradgrad u_{\ell,0}\\
&=\overbrace{\varphi_{\ell}S_{\ell,1}
-2\sym\big((\grad\varphi_{\ell})(\grad u_{\ell,0})^{\top}\big)
-u_{\ell,0}\Gradgrad\varphi_{\ell}}^{=:S_{\ell}}\\
&\qquad+\Gradgrad(\underbrace{\varphi_{\ell}u_{\ell,0})}_{=:u_{\ell}}
\end{align*}
with $S_{\ell}\in\H{k+1}{\S,\widehat\ga_{t,\ell}}(\om_{\ell})$
and $u_{\ell}\in\H{k+2}{\widehat\ga_{t,\ell}}(\om_{\ell})$.
Extending $S_{\ell}$ and $u_{\ell}$ by zero to $\om$
gives tensor fields $\widetilde{S}_{\ell}\in\H{k+1}{\S,\gat}(\om)$ and 
$\widetilde{u}_{\ell}\in\H{k+2}{\gat}(\om)$ as well as 
\begin{align*}
S=\sum_{\ell=-L}^{L}\varphi_{\ell}S|_{\om_{\ell}}
&=\sum_{\ell=-L}^{L}\widetilde{S}_{\ell}+\Gradgrad\sum_{\ell=-L}^{L}\widetilde{u}_{\ell}\\
&\in\H{k+1}{\S,\gat}(\om)+\Gradgrad\H{k+2}{\gat}(\om)
\subset\H{k}{\S,\gat}(\Rot,\om).
\end{align*}
As all operations have been linear and continuous we set
$$\PotQ_{\TRotS,\gat}^{k,1}S:=\sum_{\ell=-L}^{L}\widetilde{S}_{\ell}\in\H{k+1}{\S,\gat}(\om),\qquad
\PotQ_{\TRotS,\gat}^{k,0}S:=\sum_{\ell=-L}^{L}\widetilde{u}_{\ell}\in\H{k+2}{\gat}(\om).$$
\item
Let $k\geq0$ and let $T\in\bH{k}{\T,\gat}(\Div,\om)$. Then by definition
$T|_{\om_{\ell}}\in\bH{k}{\T,\ga_{t,\ell}}(\Div,\om_{\ell})$ 
and we decompose by Corollary \ref{bih:highorderregdecoextdombihcor}
$$T|_{\om_{\ell}}=T_{\ell,1}
+\Rot S_{\ell,0}$$
with 
$T_{\ell,1}:=\PotQ_{\DivT,\ga_{t,\ell}}^{k,1}T|_{\om_{\ell}}\in\H{k+1}{\T,\ga_{t,\ell}}(\om_{\ell})$
and 
$S_{\ell,0}:=\PotQ_{\DivT,\ga_{t,\ell}}^{k,0}T|_{\om_{\ell}}\in\H{k+1}{\S,\ga_{t,\ell}}(\om_{\ell})$.
Lemma \ref{bih:lem:cutlem} yields
\begin{align*}
\varphi_{\ell}T|_{\om_{\ell}}
&=\varphi_{\ell}T_{\ell,1}
+\varphi_{\ell}\Rot S_{\ell,0}
=\underbrace{\varphi_{\ell}T_{\ell,1}
+S_{\ell,0}\spn\grad\varphi_{\ell}}_{=:T_{\ell}}
+\Rot(\underbrace{\varphi_{\ell}S_{\ell,0})}_{=:S_{\ell}}
\end{align*}
with $T_{\ell}\in\H{k+1}{\T,\widehat\ga_{t,\ell}}(\om_{\ell})$
and $S_{\ell}\in\H{k+1}{\S,\widehat\ga_{t,\ell}}(\om_{\ell})$.
Extending $T_{\ell}$ and $S_{\ell}$ by zero to $\om$
gives tensor fields $\widetilde{T}_{\ell}\in\H{k+1}{\T,\gat}(\om)$ and 
$\widetilde{S}_{\ell}\in\H{k+1}{\S,\gat}(\om)$ as well as 
\begin{align*}
T=\sum_{\ell=-L}^{L}\varphi_{\ell}T|_{\om_{\ell}}
&=\sum_{\ell=-L}^{L}\widetilde{T}_{\ell}+\Rot\sum_{\ell=-L}^{L}\widetilde{S}_{\ell}\\
&\in\H{k+1}{\T,\gat}(\om)+\Rot\H{k+1}{\S,\gat}(\om)
\subset\H{k}{\T,\gat}(\Div,\om).
\end{align*}
As all operations have been linear and continuous we set
$$\PotQ_{\DivT,\gat}^{k,1}T:=\sum_{\ell=-L}^{L}\widetilde{T}_{\ell}\in\H{k+1}{\T,\gat}(\om),\qquad
\PotQ_{\DivT,\gat}^{k,0}T:=\sum_{\ell=-L}^{L}\widetilde{S}_{\ell}\in\H{k+1}{\S,\gat}(\om).$$
\item
Let $k\geq0$ and let $T\in\bH{k}{\T,\gat}(\symRot,\om)$. Then by definition
$T|_{\om_{\ell}}\in\bH{k}{\T,\ga_{t,\ell}}(\symRot,\om_{\ell})$ 
and we decompose by Corollary \ref{bih:highorderregdecoextdombihcor}
$$T|_{\om_{\ell}}=T_{\ell,1}
+\devGrad v_{\ell,0}$$
with 
$T_{\ell,1}:=\PotQ_{\SRotT,\ga_{t,\ell}}^{k,1}T|_{\om_{\ell}}\in\H{k+1}{\T,\ga_{t,\ell}}(\om_{\ell})$
and 
$v_{\ell,0}:=\PotQ_{\SRotT,\ga_{t,\ell}}^{k,0}T|_{\om_{\ell}}\in\H{k+1}{\ga_{t,\ell}}(\om_{\ell})$.
Lemma \ref{bih:lem:cutlem} yields
\begin{align*}
\varphi_{\ell}T|_{\om_{\ell}}
&=\varphi_{\ell}T_{\ell,1}
+\varphi_{\ell}\devGrad v_{\ell,0}\\
&=\underbrace{\varphi_{\ell}T_{\ell,1}
+\dev\big(v_{\ell,0}(\grad\varphi_{\ell})^{\top}\big)}_{=:T_{\ell}}
+\devGrad(\underbrace{\varphi_{\ell}v_{\ell,0})}_{=:v_{\ell}}
\end{align*}
with $T_{\ell}\in\H{k+1}{\T,\widehat\ga_{t,\ell}}(\om_{\ell})$
and $v_{\ell}\in\H{k+1}{\widehat\ga_{t,\ell}}(\om_{\ell})$.
Extending $T_{\ell}$ and $v_{\ell}$ by zero to $\om$
gives tensor fields $\widetilde{T}_{\ell}\in\H{k+1}{\T,\gat}(\om)$ and 
$\widetilde{v}_{\ell}\in\H{k+1}{\gat}(\om)$ as well as 
\begin{align*}
T=\sum_{\ell=-L}^{L}\varphi_{\ell}T|_{\om_{\ell}}
&=\sum_{\ell=-L}^{L}\widetilde{T}_{\ell}+\devGrad\sum_{\ell=-L}^{L}\widetilde{v}_{\ell}\\
&\in\H{k+1}{\T,\gat}(\om)+\devGrad\H{k+1}{\gat}(\om)
\subset\H{k}{\T,\gat}(\symRot,\om).
\end{align*}
As all operations have been linear and continuous we set
$$\PotQ_{\SRotT,\gat}^{k,1}T:=\sum_{\ell=-L}^{L}\widetilde{T}_{\ell}\in\H{k+1}{\T,\gat}(\om),\qquad
\PotQ_{\SRotT,\gat}^{k,0}T:=\sum_{\ell=-L}^{L}\widetilde{v}_{\ell}\in\H{k+1}{\gat}(\om).$$
\item
Let $k\geq1$ and let  $S\in\bH{k,k-1}{\S,\gat}(\divDiv,\om)$. Then by definition
$S|_{\om_{\ell}}\in\bH{k,k-1}{\S,\ga_{t,\ell}}(\divDiv,\om_{\ell})$ 
and we decompose by Corollary \ref{bih:highorderregdecoextdombihcornonstandard}
$$S|_{\om_{\ell}}=S_{\ell,1}
+\symRot T_{\ell,0}$$
with 
$S_{\ell,1}:=\PotQ_{\divDivS,\ga_{t,\ell}}^{k,k-1,1}S|_{\om_{\ell}}\in\H{k+1}{\S,\ga_{t,\ell}}(\om_{\ell})$
and 
$T_{\ell,0}:=\PotQ_{\divDivS,\ga_{t,\ell}}^{k,k-1,0}S|_{\om_{\ell}}\in\H{k+1}{\T,\ga_{t,\ell}}(\om_{\ell})$.
Thus
\begin{align}
\begin{aligned}
\label{bih:decovarphiS}
\varphi_{\ell}S|_{\om_{\ell}}
&=\varphi_{\ell}S_{\ell,1}
+\varphi_{\ell}\symRot T_{\ell,0}\\
&=\underbrace{\varphi_{\ell}S_{\ell,1}
+\sym(T_{\ell,0}\spn\grad\varphi_{\ell})}_{=:S_{\ell}}
+\symRot(\underbrace{\varphi_{\ell}T_{\ell,0})}_{=:T_{\ell}}
\end{aligned}
\end{align}
with $S_{\ell}\in\H{k+1}{\S,\widehat\ga_{t,\ell}}(\om_{\ell})$
and $T_{\ell}\in\H{k+1}{\T,\widehat\ga_{t,\ell}}(\om_{\ell})$.
Extending $S_{\ell}$ and $T_{\ell}$ by zero to $\om$
gives fields $\widetilde{S}_{\ell}\in\H{k+1}{\S,\gat}(\om)$ and 
$\widetilde{T}_{\ell}\in\H{k+1}{\T,\gat}(\om)$ as well as 
\begin{align*}
S=\sum_{\ell=-L}^{L}\varphi_{\ell}S|_{\om_{\ell}}
&=\sum_{\ell=-L}^{L}\widetilde{S}_{\ell}+\symRot\sum_{\ell=-L}^{L}\widetilde{T}_{\ell}\\
&\in\H{k+1}{\S,\gat}(\om)+\symRot\H{k+1}{\T,\gat}(\om)
\subset\H{k,k-1}{\S,\gat}(\divDiv,\om).
\end{align*}
As all operations have been linear and continuous we set
$$\PotQ_{\divDivS,\gat}^{k,k-1,1}S:=\sum_{\ell=-L}^{L}\widetilde{S}_{\ell}\in\H{k+1}{\S,\gat}(\om),\qquad
\PotQ_{\divDivS,\gat}^{k,k-1,0}S:=\sum_{\ell=-L}^{L}\widetilde{T}_{\ell}\in\H{k+1}{\T,\gat}(\om).$$
\item
Let $k\geq0$ and let  $S\in\bH{k}{\S,\gat}(\divDiv,\om)$. Then by definition
$S|_{\om_{\ell}}\in\bH{k}{\S,\ga_{t,\ell}}(\divDiv,\om_{\ell})$ 
and we decompose by Corollary \ref{bih:highorderregdecoextdombihcor}
$$S|_{\om_{\ell}}=S_{\ell,1}
+\symRot T_{\ell,0}$$
with 
$S_{\ell,1}:=\PotQ_{\divDivS,\ga_{t,\ell}}^{k,k-1,1}S|_{\om_{\ell}}\in\H{k+2}{\S,\ga_{t,\ell}}(\om_{\ell})$
and 
$T_{\ell,0}:=\PotQ_{\divDivS,\ga_{t,\ell}}^{k,k-1,0}S|_{\om_{\ell}}\in\H{k+1}{\T,\ga_{t,\ell}}(\om_{\ell})$.
Now we follow the arguments from \eqref{bih:decovarphiS} on.
Note that still only $S_{\ell}\in\H{k+1}{\S,\widehat\ga_{t,\ell}}(\om_{\ell})$ holds,
i.e., we have lost one order of regularity for $S_{\ell}$.
Nevertheless, we get 
$$S\in\H{k+1}{\S,\gat}(\om)+\symRot\H{k+1}{\T,\gat}(\om),$$
and all operations have been linear and continuous.
But this implies by the previous step
$$S\in\H{k+1,k}{\S,\gat}(\divDiv,\om)+\symRot\H{k+1}{\T,\gat}(\om).$$
Again by the previous step we obtain 
\begin{align*}
S&\in\H{k+2}{\S,\gat}(\om)+\symRot\H{k+2}{\T,\gat}(\om)+\symRot\H{k+1}{\T,\gat}(\om)\\
&=\H{k+2}{\S,\gat}(\om)+\symRot\H{k+1}{\T,\gat}(\om)
\subset\H{k}{\S,\gat}(\divDiv,\om),
\end{align*}
and all operations have been linear and continuous.
\end{itemize}
It remains to prove the assertions on the operators $\devGrad$ and $\Gradgrad$.
\begin{itemize}
\item
Let $v\in\bH{k}{\gat}(\devGrad,\om)$. Then by Corollary \ref{bih:weakstrongextdombihcor}
$$\varphi_{\ell}v\in\bH{k}{\widehat\ga_{t,\ell}}(\devGrad,\om_{\ell})
=\H{k}{\widehat\ga_{t,\ell}}(\devGrad,\om_{\ell})
=\H{k+1}{\widehat\ga_{t,\ell}}(\om_{\ell}).$$
Extending $\varphi_{\ell}v$ by zero to $\om$
yields $v_{\ell}\in\H{k+1}{\gat}(\om)$
and $v=\sum_{\ell}\varphi_{\ell}v=\sum_{\ell}v_{\ell}\in\H{k+1}{\gat}(\om)$.
\item
Let $u\in\bH{k}{\gat}(\Gradgrad,\om)$. Then by Corollary \ref{bih:weakstrongextdombihcor}
$$\varphi_{\ell}u\in\bH{k}{\widehat\ga_{t,\ell}}(\Gradgrad,\om_{\ell})
=\H{k}{\widehat\ga_{t,\ell}}(\Gradgrad,\om_{\ell})
=\H{k+2}{\widehat\ga_{t,\ell}}(\om_{\ell}).$$
Extending $\varphi_{\ell}u$ by zero to $\om$
yields $u_{\ell}\in\H{k+2}{\gat}(\om)$
and $u=\sum_{\ell}\varphi_{\ell}u=\sum_{\ell}u_{\ell}\in\H{k+2}{\gat}(\om)$.
\item
Let $u\in\bH{k,k-1}{\gat}(\Gradgrad,\om)$. Then 
$\varphi_{\ell}u\in\bH{k,k-1}{\widehat\ga_{t,\ell}}(\Gradgrad,\om_{\ell})
=\H{k+1}{\widehat\ga_{t,\ell}}(\om_{\ell})$
by \eqref{bih:Gradgradkkmoregformula}.
Extending $\varphi_{\ell}u$ by zero to $\om$
yields $u_{\ell}\in\H{k+1}{\gat}(\om)$
and $u=\sum_{\ell}\varphi_{\ell}u=\sum_{\ell}u_{\ell}\in\H{k+1}{\gat}(\om)$.
\end{itemize}
The proof is finished.
\end{proof}

\begin{proof}[Proof of Theorem \ref{bih:theo:cptembzeroorder}.]
Note that these types of compact embeddings are independent of $\eps$ and $\mu$,
cf.~\cite[Lemma 5.1]{PW2021a}. So, let $\eps=\mu=\id$.
Lemma \ref{bih:lem:highorderregdecobih} (for $k=0$)
yields, e.g., the bounded regular decomposition
$$D(\A_{1})=\H{}{\S,\gat}(\Rot,\om)
=\H{1}{\S,\gat}(\om)
+\Gradgrad\H{2}{\gat}(\om)$$
with $\H{+}{1}=\H{1}{\S,\gat}(\om)$ and $\H{+}{0}=\H{2}{\gat}(\om)$ and 
$\H{}{1}=\L{2}{\S}(\om)$, $\H{}{0}=\L{2}{}(\om)$.
Rellich's selection theorem and 
\cite[Corollary 2.12]{PZ2021a}, cf.~\cite[Lemma 2.22]{PS2021b},
yield that $D(\A_{1})\cap D(\A_{0}^{*})\incl\H{}{1}$ is compact.
Analogously, we show the compactness of $D(\A_{2})\cap D(\A_{1}^{*})\incl\H{}{2}$
using, e.g., the bounded regular decomposition
$D(\A_{2})=\H{}{\T,\gat}(\Div,\om)
=\H{1}{\T,\gat}(\om)
+\Rot\H{1}{\S,\gat}(\om)$.
\end{proof}

\begin{proof}[Proof of Theorem \ref{bih:theo:rangeselawithoutbdpot}.]
We only show the representations for $R(\TRotSgat^{k})$ and $R(\divDivSgat^{k})$.
The others follow analogously.
\begin{itemize}
\item
By Lemma \ref{bih:lem:highorderregdecobih} and Corollary \ref{bih:cor:weakstrongbih} we have
\begin{align}
\label{bih:rangeRotkp2}
R(\TRotSgat^{k})
&=\Rot\H{k}{\S,\gat}(\Rot,\om)
=\Rot\H{k+1}{\S,\gat}(\om).
\end{align}
Moreover,
\begin{align*}
R(\TRotSgat^{k})
&\subset\H{k}{\T,\gat,0}(\Div,\om)\cap\Harm{}{\T,\gan,\gat,\mu}(\om)^{\bot_{\L{2}{\T}(\om)}}\\
&=\H{k}{\T,\gat}(\om)\cap\H{}{\T,\gat,0}(\Div,\om)\cap\Harm{}{\T,\gan,\gat,\mu}(\om)^{\bot_{\L{2}{\T}(\om)}}
=\H{k}{\T,\gat}(\om)\cap R(\TRotSgat),
\end{align*}
since by Theorem \ref{bih:theo:miniFATzeroorder} (iv)
\begin{align}
\label{bih:helmdecobih1}
R(\TRotSgat)
=\H{}{\T,\gat,0}(\Div,\om)\cap\Harm{}{\T,\gan,\gat,\mu}(\om)^{\bot_{\L{2}{\T}(\om)}}.
\end{align}
Thus it remains to show
$$\H{k}{\T,\gat,0}(\Div,\om)\cap\Harm{}{\T,\gan,\gat,\mu}(\om)^{\bot_{\L{2}{\T}(\om)}}
\subset\Rot\H{k}{\S,\gat}(\Rot,\om),\qquad
k\geq1.$$
For this, let $k\geq1$ and
$T\in\H{k}{\T,\gat,0}(\Div,\om)\cap\Harm{}{\T,\gan,\gat,\mu}(\om)^{\bot_{\L{2}{\T}(\om)}}$.
By \eqref{bih:helmdecobih1} and \eqref{bih:rangeRotkp2} we have 
$$T\in R(\TRotSgat)=\Rot\H{1}{\S,\gat}(\om)$$
and hence there is $S_{1}\in\H{1}{\S,\gat}(\om)$ such that $\Rot S_{1}=T$.
We see $S_{1}\in\H{1}{\S,\gat}(\Rot,\om)$.
Hence we are done for $k=1$.
For $k\geq2$ we have 
$T\in\Rot\H{1}{\S,\gat}(\Rot,\om)=\Rot\H{2}{\S,\gat}(\om)$ 
by \eqref{bih:rangeRotkp2}.
Thus there is $S_{2}\in\H{2}{\S,\gat}(\om)$ such that $\Rot S_{2}=T$.
Then $S_{2}\in\H{2}{\S,\gat}(\Rot,\om)$, and we are done for $k=2$.
After finitely many steps, we observe that
$T$ belongs to $\Rot\H{k}{\S,\gat}(\Rot,\om)$.
\item
By Lemma \ref{bih:lem:highorderregdecobih} and Corollary \ref{bih:cor:weakstrongbih} we have
\begin{align}
\nonumber
\divDiv\H{k+2}{\S,\gat}(\om)
&\subset\divDiv\H{k+1,k}{\S,\gat}(\divDiv,\om)
=R(\divDivSgat^{k+1,k})\\
\nonumber
&\subset\divDiv\H{k}{\S,\gat}(\divDiv,\om)
=R(\divDivSgat^{k})
=\divDiv\H{k+2}{\S,\gat}(\om).
\intertext{In particular,}
\label{bih:rangedivDivkp2}
R(\divDivSgat^{k})
&=\divDiv\H{k}{\S,\gat}(\divDiv,\om)
=\divDiv\H{k+2}{\S,\gat}(\om).
\intertext{Moreover,}
\nonumber
R(\divDivSgat^{k})
&\subset\H{k}{\gat}(\om)\cap(\Pone_{\gan})^{\bot_{\L{2}{}(\om)}}
=\H{k}{\gat}(\om)\cap R(\divDivSgat),
\end{align}
since
\begin{align}
\label{bih:helmdecobih2}
R(\divDivSgat)
=\L{2}{}(\om)\cap(\Pone_{\gan})^{\bot_{\L{2}{}(\om)}}.
\end{align}
Thus it remains to show
$$\H{k}{\gat}(\om)\cap(\Pone_{\gan})^{\bot_{\L{2}{}(\om)}}
\subset\divDiv\H{k}{\S,\gat}(\divDiv,\om),\qquad
k\geq1.$$
For this, let $k\geq1$ and
$u\in\H{k}{\gat}(\om)\cap(\Pone_{\gan})^{\bot_{\L{2}{}(\om)}}$.
By \eqref{bih:helmdecobih2} and \eqref{bih:rangedivDivkp2} we have 
$$u\in R(\divDivSgat)=\divDiv\H{2}{\S,\gat}(\om)$$
and hence there is $S_{1}\in\H{2}{\S,\gat}(\om)$ such that $\divDiv S_{1}=u$.
We see $S_{1}\in\H{2}{\S,\gat}(\divDiv,\om)$ resp. $S_{1}\in\H{1}{\S,\gat}(\divDiv,\om)$ if $k=1$.
Hence we are done for $k=1$ and $k=2$.
For $k\geq2$ we have 
$u\in\divDiv\H{2}{\S,\gat}(\divDiv,\om)=\divDiv\H{4}{\S,\gat}(\om)$ 
by \eqref{bih:rangedivDivkp2}.
Thus there is $S_{2}\in\H{4}{\S,\gat}(\om)$ such that $\divDiv S_{2}=u$.
Then $S_{2}\in\H{4}{\S,\gat}(\divDiv,\om)$ resp. $S_{2}\in\H{3}{\S,\gat}(\divDiv,\om)$ if $k=3$,
and we are done for $k=3$ and $k=4$.
After finitely many steps, we observe that
$u$ belongs to $\divDiv\H{k}{\S,\gat}(\divDiv,\om)$,
finishing the proof.
\end{itemize}
\end{proof}

\begin{proof}[Proof of Theorem \ref{bih:theo:cptembhigherorder}.]
We follow in close lines the proof of \cite[Theorem 4.11]{PZ2021a},
cf.~\cite[Theorem 4.16]{PS2021b} and \cite[Theorem 3.19]{PS2021d}, using induction.
The case $k=0$ is given by Theorem \ref{bih:theo:cptembzeroorder}.
Let $k\geq1$ and let $(S_{\ell})$ be a bounded sequence in 
$\H{k}{\S,\gat}(\Rot,\om)\cap\H{k}{\S,\gan}(\divDiv,\om)$.
Note that 
$$\H{k}{\S,\gat}(\Rot,\om)\cap\H{k}{\S,\gan}(\divDiv,\om)
\subset\H{k}{\S,\gat}(\om)\cap\H{k}{\S,\gan}(\om)
=\H{k}{\S,\ga}(\om).$$
By assumption and w.l.o.g. we have that $(S_{\ell})$ 
is a Cauchy sequence in $\H{k-1}{\S,\ga}(\om)$.
Moreover, for all $|\alpha|=k$ we have
$\p^{\alpha}S_{\ell}\in\H{}{\S,\gat}(\Rot,\om)\cap\H{}{\S,\gan}(\divDiv,\om)$
with $\Rot\p^{\alpha}S_{\ell}=\p^{\alpha}\Rot S_{\ell}$
and $\divDiv\p^{\alpha}S_{\ell}=\p^{\alpha}\divDiv S_{\ell}$
by Lemma \ref{bih:lem:scharzlemma}.
Hence $(\p^{\alpha}S_{\ell})$ is a bounded sequence in the zero order space
$\H{}{\S,\gat}(\Rot,\om)\cap\H{}{\S,\gan}(\divDiv,\om)$.
Thus, w.l.o.g. $(\p^{\alpha}S_{\ell})$ is a Cauchy sequence in $\L{2}{\S}(\om)$
by Theorem \ref{bih:theo:cptembzeroorder}.
Finally, $(S_{\ell})$ is a Cauchy sequence in $\H{k}{\S,\ga}(\om)$.
Analogously, we show the assertion for the second compact embedding.
\end{proof}

\begin{proof}[Proof of Remark \ref{bih:rem:cptembhigherorder}.]
Let $(S_{\ell})$ be a bounded sequence in 
$\H{k}{\S,\gat}(\Rot,\om)\cap\H{k}{\S,\gan}(\divDiv,\om)$.
In particular, $(S_{\ell})$ is bounded in 
$\H{k}{\S,\gat}(\Rot,\om)\cap\H{k,k-1}{\S,\gan}(\divDiv,\om)$.
According to Lemma \ref{bih:lem:highorderregdecobih}, i.e.,
$$\H{k,k-1}{\S,\gat}(\divDiv,\om)
=\H{k+1}{\S,\gat}(\om)
+\symRot\H{k+1}{\T,\gat}(\om)$$
we decompose 
$S_{\ell}=\widetilde{S}_{\ell}+\symRot T_{\ell}$
with $\widetilde{S}_{\ell}\in\H{k+1}{\S,\gat}(\om)$ and $T_{\ell}\in\H{k+1}{\T,\gat}(\om)$.
By the boundedness of the regular decomposition operators,
$(\widetilde{S}_{\ell})$ and $(T_{\ell})$ are bounded in 
$\H{k+1}{\S,\gat}(\om)$ and $\H{k+1}{\T,\gat}(\om)$, respectively.
W.l.o.g. $(\widetilde{S}_{\ell})$ and $(T_{\ell})$ converge in 
$\H{k}{\S,\gat}(\om)$ and $\H{k}{\T,\gat}(\om)$, respectively.
For all $0\leq|\alpha|\leq k$ Lemma \ref{bih:lem:scharzlemma} yields
$(\p^{\alpha}S_{\ell})\subset\H{}{\S,\gat}(\Rot,\om)$
and $\Rot\p^{\alpha}S_{\ell}=\p^{\alpha}\Rot S_{\ell}$. With the notations
$S_{\ell,l}:=S_{\ell}-S_{l}$, $\widetilde{S}_{\ell,l}:=\widetilde{S}_{\ell}-\widetilde{S}_{l}$,
and $T_{\ell,l}:=T_{\ell}-T_{l}$, we get
\begin{align*}
\norm{S_{\ell,l}}_{\H{k}{\S}(\om)}^2
&=\scp{S_{\ell,l}}{\widetilde{S}_{\ell,l}}_{\H{k}{\S}(\om)}
+\scp{S_{\ell,l}}{\symRot T_{\ell,l}}_{\H{k}{\S}(\om)}\\
&=\scp{S_{\ell,l}}{\widetilde{S}_{\ell,l}}_{\H{k}{\S}(\om)}
+\scp{\Rot S_{\ell,l}}{T_{\ell,l}}_{\H{k}{\T}(\om)}
\leq c\big(\norm{\widetilde{S}_{\ell,l}}_{\H{k}{\S}(\om)}+\norm{T_{\ell,l}}_{\H{k}{\T}(\om)}\big)
\to0,
\end{align*}
completing the proof.
\end{proof}

\begin{proof}[Proof of Theorem \ref{bih:theo:highorderregdecoinfty}.]
Theorem \ref{bih:theo:regdecohigherorder} and \eqref{bih:helmcoho2} show
\begin{align*}
\H{k}{\S,\gat}(\Rot,\om)
&=R(\widetilde\PotQ_{\TRotS,\gat}^{k,1})
\dotplus\H{k}{\S,\gat,0}(\Rot,\om),\\
\H{k}{\S,\gat,0}(\Rot,\om)
&=\Gradgrad\H{k+2}{\gat}(\om)
+\Lin\B{\TRotSgat}(\om)
\end{align*}
To prove the directness of the second sum, let
$$\sum_{\ell=1}^{d_{\om,\S,\gat}}\lambda_{\ell}\vB{\TRotSgat}{\ell}
\in\Gradgrad\H{k+2}{\gat}(\om)
\cap\Lin\B{\TRotSgat}(\om).$$
Then $0=\sum_{\ell}\lambda_{\ell}\pi_{N(\divDivSgan\eps)}\vB{\TRotSgat}{\ell}
\in\Lin\pi_{N(\divDivSgan\eps)}\B{\TRotSgat}$
and therefore $\lambda_{\ell}=0$ for all $\ell$
as $\pi_{N(\divDivSgan\eps)}\B{\TRotSgat}$ is a basis of 
$\Harm{}{\S,\gat,\gan,\eps}(\om)$
by Theorem \ref{bih:theo:cohomologyinfty}.
Concerning the boundedness of the decompositions, let 
$$\H{k}{\S,\gat,0}(\Rot,\om)\ni S=\Gradgrad u+B,\qquad
u\in\H{k+2}{\gat}(\om),\quad
B\in\Lin\B{\TRotSgat}(\om).$$
By Theorem \ref{bih:theo:regpothigherorder}
$\Gradgrad u\in R(\SGradgradgat^{k})$ and 
$\widetilde{u}:=\PotP_{\SGradgrad,\gat}^{k}\Gradgrad u\in\H{k+2}{\gat}(\om)$ solves
$\Gradgrad\widetilde{u}=\Gradgrad u$ with 
$\norm{\widetilde{u}}_{\H{k+2}{}(\om)}\leq c\norm{\Gradgrad u}_{\H{k}{\S}(\om)}$.
Therefore, 
$$\norm{\widetilde{u}}_{\H{k+2}{}(\om)}
+\norm{B}_{\H{k}{\S}(\om)}
\leq c\big(\norm{\Gradgrad u}_{\H{k}{\S}(\om)}
+\norm{B}_{\H{k}{\S}(\om)}\big)
\leq c\big(\norm{S}_{\H{k}{\S}(\om)}
+\norm{B}_{\H{k}{\S}(\om)}\big).$$
Note that the mapping
$$\Abb{I_{\pi_{N(\divDivSgan\eps)}}}
{\Lin\B{\TRotSgat}(\om)}
{\Lin\pi_{N(\divDivSgan\eps)}\B{\TRotSgat}(\om)=\Harm{}{\S,\gat,\gan,\eps}(\om)}
{\vB{\TRotSgat}{\ell}}
{\pi_{N(\divDivSgan\eps)}\vB{\TRotSgat}{\ell}}$$
is a topological isomorphism 
(between finite dimensional spaces and with arbitrary norms). Thus 
$$\norm{B}_{\H{k}{\S}(\om)}
\leq c\norm{B}_{\L{2}{\S}(\om)}
\leq c\norm{\pi_{N(\divDivSgan\eps)}B}_{\L{2}{\S}(\om)}
=c\norm{\pi_{N(\divDivSgan\eps)}S}_{\L{2}{\S}(\om)}
\leq c\norm{S}_{\L{2}{\S}(\om)}
\leq c\norm{S}_{\H{k}{\S}(\om)}.$$
Finally, we see 
$S=\Gradgrad\widetilde{u}+B
\in\Gradgrad\H{k+2}{\gat}(\om)
\dotplus\Lin\B{\TRotSgat}(\om)$
and 
$$\norm{\widetilde{u}}_{\H{k+2}{}(\om)}
+\norm{B}_{\H{k}{\S}(\om)}
\leq c\norm{S}_{\H{k}{\S}(\om)}.$$
The other assertions for 
$\H{k}{\T,\gan}(\Div,\om)$, $\H{k}{\T,\gat}(\symRot,\om)$, $\H{k}{\S,\gan}(\divDiv,\om)$, 
and $\H{k+1,k}{\S,\gan}(\divDiv,\om)$
follow analogously.
\end{proof}

\begin{proof}[Proof of Theorem \ref{bih:theo:Bcoho1}.]
For $k=0$ and 
$S\in\Harm{}{\S,\gat,\gan,\eps}(\om)\cap\B{\TRotSgat}(\om)^{\bot_{\L{2}{\S,\eps}(\om)}}$
we have
\begin{align*}
0=\scp{S}{\vB{\TRotSgat}{\ell}}_{\L{2}{\S,\eps}(\om)}
&=\scp{\pi_{N(\divDivSgan\eps)}S}{\vB{\TRotSgat}{\ell}}_{\L{2}{\S,\eps}(\om)}\\
&=\scp{S}{\pi_{N(\divDivSgan\eps)}\vB{\TRotSgat}{\ell}}_{\L{2}{\S,\eps}(\om)}
\end{align*}
and hence $S=0$ by Theorem \ref{bih:theo:cohomologyinfty}.
Analogously, we see for 
$S\in\Harm{}{\S,\gat,\gan,\eps}(\om)\cap\B{\divDivSgan}(\om)^{\bot_{\L{2}{\S}(\om)}}$
\begin{align*}
0=\scp{S}{\vB{\divDivSgan}{\ell}}_{\L{2}{\S}(\om)}
&=\scp{\pi_{N(\TRotSgat)}S}{\eps^{-1}\vB{\divDivSgan}{\ell}}_{\L{2}{\S,\eps}(\om)}\\
&=\scp{S}{\pi_{N(\TRotSgat)}\eps^{-1}\vB{\divDivSgan}{\ell}}_{\L{2}{\S,\eps}(\om)}
\end{align*}
and thus $S=0$ again by Theorem \ref{bih:theo:cohomologyinfty}.
According to \eqref{bih:helmcoho1} we can decompose
\begin{align*}
N(\divDivSgan\eps)
&=R(\eps^{-1}\SRotTgan)
\oplus_{\L{2}{\S,\eps}(\om)}
\Harm{}{\S,\gat,\gan,\eps}(\om),\\
N(\TRotSgat)
&=R(\SGradgradgat)
\oplus_{\L{2}{\S,\eps}(\om)}
\Harm{}{\S,\gat,\gan,\eps}(\om),
\end{align*}
which shows by \eqref{bih:Bortho} the other two assertions.
Let $k\geq0$. The case $k=0$ and Theorem \ref{bih:theo:rangeselawithoutbdpot} show
\begin{align*}
N(\divDivSgan^{k}\eps)\cap\B{\TRotSgat}(\om)^{\bot_{\L{2}{\S,\eps}(\om)}}
&=\eps^{-1}\H{k}{\S,\gan}(\om)\cap N(\divDivSgan\eps)\cap\B{\TRotSgat}(\om)^{\bot_{\L{2}{\S,\eps}(\om)}}\\
&=\eps^{-1}\H{k}{\S,\gan}(\om)\cap R(\eps^{-1}\SRotTgan)\\
&=R(\eps^{-1}\SRotTgan^{k})
=\eps^{-1}\symRot\H{k+1}{\T,\gan}(\om),\\
N(\TRotSgat^{k})\cap\B{\divDivSgan}(\om)^{\bot_{\L{2}{\S}(\om)}}
&=\H{k}{\S,\gat}(\om)\cap N(\TRotSgat)\cap\B{\divDivSgan}(\om)^{\bot_{\L{2}{\S}(\om)}}\\
&=\H{k}{\S,\gat}(\om)\cap R(\SGradgradgat)\\
&=R(\SGradgradgat^{k})
=\Gradgrad\H{k+2}{\gat}(\om).
\end{align*}
Analogously we prove the assertions for the remaining 
$\L{2}{\T,\mu}(\om)$-related spaces.
\end{proof}

\section{Some Technical Remarks}
\label{sec:someremarks}%

\begin{rem}[bounded regular decompositions from bounded regular potentials]
\label{bih:rem:regdecohigherorder}
It holds 
\begin{align*}
\Rot\widetilde\PotQ_{\TRotS,\gat}^{k,1}
=\Rot\PotQ_{\TRotS,\gat}^{k,1}
&=\TRotSgat^{k},\\
\Div\widetilde\PotQ_{\DivT,\gat}^{k,1}
=\Div\PotQ_{\DivT,\gat}^{k,1}
&=\DivTgat^{k},\\
\symRot\widetilde\PotQ_{\SRotT,\gat}^{k,1}
=\symRot\PotQ_{\SRotT,\gat}^{k,1}
&=\SRotTgat^{k},\\
\divDiv\widetilde\PotQ_{\divDivS,\gat}^{k,1}
=\divDiv\PotQ_{\divDivS,\gat}^{k,1}
&=\divDivSgat^{k},\\
\divDiv\widetilde\PotQ_{\divDivS,\gat}^{k+1,k,1}
=\divDiv\PotQ_{\divDivS,\gat}^{k+1,k,1}
&=\divDivSgat^{k+1,k}.
\end{align*}
Therefore, the kernels $\H{k}{\S,\gat,0}(\Rot,\om)$, $\H{k}{\T,\gat,0}(\Div,\om)$,
$\H{k}{\T,\gat,0}(\symRot,\om)$, 
and $\H{k}{\S,\gat,0}(\divDiv,\om)$, $\H{k+1}{\S,\gat,0}(\divDiv,\om)$ 
are invariant under 
$\PotQ_{\TRotS,\gat}^{k,1}$, $\widetilde\PotQ_{\TRotS,\gat}^{k,1}$,
$\PotQ_{\DivT,\gat}^{k,1}$, $\widetilde\PotQ_{\DivT,\gat}^{k,1}$,
$\PotQ_{\SRotT,\gat}^{k,1}$, $\widetilde\PotQ_{\SRotT,\gat}^{k,1}$,
$\PotQ_{\divDivS,\gat}^{k,1}$, $\widetilde\PotQ_{\divDivS,\gat}^{k,1}$,
and $\PotQ_{\divDivS,\gat}^{k+1,k,1}$, $\widetilde\PotQ_{\divDivS,\gat}^{k+1,k,1}$, 
respectively. Moreover,
\begin{align*}
R(\widetilde\PotQ_{\TRotS,\gat}^{k,1})
&=R(\PotP_{\TRotS,\gat}^{k}),
&
\widetilde\PotQ_{\TRotS,\gat}^{k,1}
&=\PotQ_{\TRotS,\gat}^{k,1}(\TRotSgat^{k})_{\bot}^{-1}\TRotSgat^{k},\\
R(\widetilde\PotQ_{\DivT,\gat}^{k,1})
&=R(\PotP_{\DivT,\gat}^{k}),
&
\widetilde\PotQ_{\DivT,\gat}^{k,1}
&=\PotQ_{\DivT,\gat}^{k,1}(\DivTgat^{k})_{\bot}^{-1}\DivTgat^{k},\\
R(\widetilde\PotQ_{\SRotT,\gat}^{k,1})
&=R(\PotP_{\SRotT,\gat}^{k}),
&
\widetilde\PotQ_{\SRotT,\gat}^{k,1}
&=\PotQ_{\SRotT,\gat}^{k,1}(\SRotTgat^{k})_{\bot}^{-1}\SRotTgat^{k},\\
R(\widetilde\PotQ_{\divDivS,\gat}^{k,1})
&=R(\PotP_{\divDivS,\gat}^{k}),
&
\widetilde\PotQ_{\divDivS,\gat}^{k,1}
&=\PotQ_{\divDivS,\gat}^{k,1}(\divDivSgat^{k})_{\bot}^{-1}\divDivSgat^{k},\\
R(\widetilde\PotQ_{\divDivS,\gat}^{k+1,k,1})
&=R(\PotP_{\divDivS,\gat}^{k+1,k}),
&
\widetilde\PotQ_{\divDivS,\gat}^{k+1,k,1}
&=\PotQ_{\divDivS,\gat}^{k+1,k,1}(\divDivSgat^{k+1,k})_{\bot}^{-1}\divDivSgat^{k+1,k}.
\end{align*}
Hence, $\widetilde\PotQ_{\TRotS,\gat}^{k,1}$, $\widetilde\PotQ_{\DivT,\gat}^{k,1}$, 
$\widetilde\PotQ_{\SRotT,\gat}^{k,1}$, $\widetilde\PotQ_{\divDivS,\gat}^{k,1}$, 
$\widetilde\PotQ_{\divDivS,\gat}^{k+1,k,1}$
coincide with $\PotQ_{\TRotS,\gat}^{k,1}$, $\PotQ_{\DivT,\gat}^{k,1}$, 
$\PotQ_{\SRotT,\gat}^{k,1}$, $\PotQ_{\divDivS,\gat}^{k,1}$, 
$\PotQ_{\divDivS,\gat}^{k+1,k,1}$
on the reduced domains of definition 
$$D((\TRotSgat^{k})_{\bot}),\quad
D((\DivTgat^{k})_{\bot}),\quad
D((\SRotTgat^{k})_{\bot}),\quad
D((\divDivSgat^{k})_{\bot}),\quad
D((\divDivSgat^{k+1,k})_{\bot}),$$
respectively. Thus $\widetilde\PotQ_{\TRotS,\gat}^{k,1}$, $\widetilde\PotQ_{\DivT,\gat}^{k,1}$, 
$\widetilde\PotQ_{\SRotT,\gat}^{k,1}$, $\widetilde\PotQ_{\divDivS,\gat}^{k,1}$, 
$\widetilde\PotQ_{\divDivS,\gat}^{k+1,k,1}$
may differ from $\PotQ_{\TRotS,\gat}^{k,1}$, $\PotQ_{\DivT,\gat}^{k,1}$, 
$\PotQ_{\SRotT,\gat}^{k,1}$, $\PotQ_{\divDivS,\gat}^{k,1}$, 
$\PotQ_{\divDivS,\gat}^{k+1,k,1}$
only on the kernels 
$$N(\TRotSgat^{k})=\H{k}{\S,\gat,0}(\Rot,\om),\;
N(\DivTgat^{k})=\H{k}{\T,\gat,0}(\Div,\om),\;
N(\SRotTgat^{k})=\H{k}{\T,\gat,0}(\symRot,\om),$$
and $N(\divDivSgat^{k})=\H{k}{\S,\gat,0}(\divDiv,\om)$, 
$N(\divDivSgat^{k+1,k})=\H{k+1}{\S,\gat,0}(\divDiv,\om)$, 
respectively.
\end{rem}

\begin{rem}[projections]
\label{bih:rem:regdecohigherorderproj}
Recall Theorem \ref{bih:theo:regdecohigherorder}, e.g.,
for $\divDivSgat^{k}$
$$\H{k}{\S,\gat}(\divDiv,\om)
=R(\widetilde\PotQ_{\divDivS,\gat}^{k,1})
\dotplus R(\widetilde\PotN_{\divDivS,\gat}^{k}).$$
\begin{itemize}
\item[\bf(i)]
$\widetilde\PotQ_{\divDivS,\gat}^{k,1}$ and
$\widetilde\PotN_{\divDivS,\gat}^{k}=1-\widetilde\PotQ_{\divDivS,\gat}^{k,1}$
are projections.
\item[\bf(i')] 
$\widetilde\PotQ_{\divDivS,\gat}^{k,1}\widetilde\PotN_{\divDivS,\gat}^{k}
=\widetilde\PotN_{\divDivS,\gat}^{k}\widetilde\PotQ_{\divDivS,\gat}^{k,1}=0$.
\item[\bf(ii)]
For $I_{\pm}:=\widetilde\PotQ_{\divDivS,\gat}^{k,1}\pm\widetilde\PotN_{\divDivS,\gat}^{k}$
it holds $I_{+}=I_{-}^{2}=\id_{\H{k}{\S,\gat}(\divDiv,\om)}$.
Therefore, $I_{+}$, $I_{-}^{2}$, as well as
$I_{-}=2\widetilde\PotQ_{\divDivS,\gat}^{k,1}-\id_{\H{k}{\S,\gat}(\divDiv,\om)}$
are topological isomorphisms on $\H{k}{\S,\gat}(\divDiv,\om)$.
\item[\bf(iii)]
There exists $c>0$ such that for all $S\in\H{k}{\S,\gat}(\divDiv,\om)$
\begin{align*}
c\norm{\widetilde\PotQ_{\divDivS,\gat}^{k,1}S}_{\H{k+2}{\S}(\om)}
&\leq\norm{\divDiv S}_{\H{k}{}(\om)}
\leq\norm{S}_{\H{k}{\S}(\divDiv,\om)},\\
\norm{\widetilde\PotN_{\divDivS,\gat}^{k}S}_{\H{k}{\S}(\om)}
&\leq\norm{S}_{\H{k}{\S}(\om)}
+\norm{\widetilde\PotQ_{\divDivS,\gat}^{k,1}S}_{\H{k}{\S}(\om)}.
\end{align*}
\item[\bf(iii')]
For $S\in\H{k}{\S,\gat,0}(\divDiv,\om)$ we have $\widetilde\PotQ_{\divDivS,\gat}^{k,1}S=0$
and $\widetilde\PotN_{\divDivS,\gat}^{k}S=S$.
In particular, $\widetilde\PotN_{\divDivS,\gat}^{k}$ is onto.
\end{itemize}
Similar results to (i)-(iii') hold also for 
$\TRotSgat^{k}$, $\DivTgat^{k}$, $\SRotTgat^{k}$, and $\divDivSgat^{k+1,k}$.
In particular, $\widetilde\PotQ_{\TRotS,\gat}^{k,1}$, $\widetilde\PotQ_{\DivT,\gat}^{k,1}$, 
$\widetilde\PotQ_{\SRotT,\gat}^{k,1}$, $\widetilde\PotQ_{\divDivS,\gat}^{k+1,k,1}$, 
and $\widetilde\PotN_{\TRotS,\gat}^{k,1}$, $\widetilde\PotN_{\DivT,\gat}^{k,1}$, 
$\widetilde\PotN_{\SRotT,\gat}^{k,1}$, $\widetilde\PotN_{\divDivS,\gat}^{k+1,k,1}$
are projections and there exists $c>0$ such that for all 
$S\in\H{k}{\S,\gat}(\Rot,\om)$, $T\in\H{k}{\T,\gat}(\Div,\om)$,
$\widehat{T}\in\H{k}{\T,\gat}(\symRot,\om)$, 
and $\widehat{S}\in\H{k+1,k}{\S,\gat}(\divDiv,\om)$
\begin{align*}
\norm{\widetilde\PotQ_{\TRotS,\gat}^{k,1}S}_{\H{k+1}{\S}(\om)}
&\leq c\norm{\Rot S}_{\H{k}{\T}(\om)},
&
\norm{\widetilde\PotQ_{\SRotT,\gat}^{k,1}\widehat{T}}_{\H{k+1}{\T}(\om)}
&\leq c\norm{\symRot\widehat{T}}_{\H{k}{\S}(\om)},\\
\norm{\widetilde\PotQ_{\DivT,\gat}^{k,1}T}_{\H{k+1}{\T}(\om)}
&\leq c\norm{\Div T}_{\H{k}{}(\om)},
&
\norm{\widetilde\PotQ_{\divDivS,\gat}^{k+1,k,1}\widehat{S}}_{\H{k+2}{\S}(\om)}
&\leq c\norm{\divDiv\widehat{S}}_{\H{k}{}(\om)}.
\end{align*}
\end{rem}

\begin{rem}[bounded regular direct decompositions]
\label{bih:rem:highorderregdecoinfty2}
By Theorem \ref{bih:theo:highorderregdecoinfty} we have, e.g.,
\begin{align*}
\H{k}{\S,\gat}(\Rot,\om)
&=R(\widetilde\PotQ_{\TRotS,\gat}^{k,1})
\dotplus\Lin\B{\TRotSgat}(\om)
\dotplus\Gradgrad\H{k+2}{\gat}(\om)\\
&=\H{k+1}{\S,\gat}(\om)
+\Gradgrad\H{k+2}{\gat}(\om)
\end{align*}
with bounded linear regular direct decomposition operators
\begin{align*}
\widehat\PotQ_{\TRotS,\gat}^{k,1}:
\H{k}{\S,\gat}(\Rot,\om)&\to R(\widetilde\PotQ_{\TRotS,\gat}^{k,1})
\subset\H{k+1}{\S,\gat}(\om),\\
\widehat\PotQ_{\TRotS,\gat}^{k,\infty}:
\H{k}{\S,\gat}(\Rot,\om)&\to\Lin\B{\TRotSgat}(\om)
\subset\H{\infty}{\S,\gat,0}(\Rot,\om)
\subset\H{k+1}{\S,\gat}(\om),\\
\widehat\PotQ_{\TRotS,\gat}^{k,0}:
\H{k}{\S,\gat}(\Rot,\om)&\to
\H{k+2}{\gat}(\om)
\end{align*}
satisfying 
$\widehat\PotQ_{\TRotS,\gat}^{k,1}
+\widehat\PotQ_{\TRotS,\gat}^{k,\infty}
+\Gradgrad\widehat\PotQ_{\TRotS,\gat}^{k,0}
=\id_{\H{k}{\S,\gat}(\Rot,\om)}$.

A closer inspection of the proof 
allows for a more precise description of these bounded decomposition operators.
For this, let $S\in\H{k}{\S,\gat}(\Rot,\om)$. 
According to Theorem \ref{bih:theo:regdecohigherorder} and 
Remark \ref{bih:rem:regdecohigherorderproj} we decompose 
$$S=S_{R}+S_{N}
:=\widetilde\PotQ_{\TRotS,\gat}^{k,1}S+\widetilde\PotN_{\TRotS,\gat}^{k}S
\in 
R(\widetilde\PotQ_{\TRotS,\gat}^{k,1})
\dotplus 
R(\widetilde\PotN_{\TRotS,\gat}^{k})$$
with 
$R(\widetilde\PotN_{\TRotS,\gat}^{k})
=\H{k}{\S,\gat,0}(\Rot,\om)
=N(\TRotSgat^{k})$.
By Theorem \ref{bih:theo:highorderregdecoinfty} we further decompose
\begin{align*}
\H{k}{\S,\gat,0}(\Rot,\om)\ni 
S_{N}
=\Gradgrad\widetilde{u}+B
\in\Gradgrad\H{k+2}{\gat}(\om)
\dotplus\Lin\B{\TRotSgat}(\om).
\end{align*}
Then 
$\pi_{N(\divDivSgan\eps)}S_{N}
=\pi_{N(\divDivSgan\eps)}B\in\Harm{}{\S,\gat,\gan,\eps}(\om)$
and thus 
$$B=I_{\pi_{N(\divDivSgan\eps)}}^{-1}\pi_{N(\divDivSgan\eps)}S_{N}\in\Lin\B{\TRotSgat}(\om).$$
Therefore, 
\begin{align*}
\widetilde{u}
=\PotP_{\SGradgrad,\gat}^{k}\Gradgrad\widetilde{u}
&=\PotP_{\SGradgrad,\gat}^{k}(S_{N}-B)\\
&=\PotP_{\SGradgrad,\gat}^{k}(1-I_{\pi_{N(\divDivSgan\eps)}}^{-1}\pi_{N(\divDivSgan\eps)})S_{N}.
\end{align*}
Finally we see
\begin{align*}
\widehat\PotQ_{\TRotS,\gat}^{k,1}
&=\widetilde\PotQ_{\TRotS,\gat}^{k,1}
=\PotP_{\TRotS,\gat}^{k}\TRotSgat^{k}
=\PotQ_{\TRotS,\gat}^{k,1}(\TRotSgat^{k})_{\bot}^{-1}\TRotSgat^{k},\\
\widehat\PotQ_{\TRotS,\gat}^{k,\infty}
&=I_{\pi_{N(\divDivSgan\eps)}}^{-1}\pi_{N(\divDivSgan\eps)}\widetilde\PotN_{\TRotS,\gat}^{k},\\
\widehat\PotQ_{\TRotS,\gat}^{k,0}
&=\PotP_{\SGradgrad,\gat}^{k}(1-I_{\pi_{N(\divDivSgan\eps)}}^{-1}\pi_{N(\divDivSgan\eps)})\widetilde\PotN_{\TRotS,\gat}^{k}
\end{align*}
with 
$\widetilde\PotN_{\TRotS,\gat}^{k}
=1-\widetilde\PotQ_{\TRotS,\gat}^{k,1}$.
Analogously, we have for the other spaces
\begin{align*}
\H{k}{\T,\gan}(\Div,\om)
&=R(\widetilde\PotQ_{\DivT,\gan}^{k,1})
\dotplus\Lin\B{\DivTgan}(\om)
\dotplus\Rot\H{k+1}{\S,\gan}(\om)\\
&=\H{k+1}{\T,\gan}(\om)
+\Rot\H{k+1}{\S,\gan}(\om),\\
\H{k}{\T,\gat}(\symRot,\om)
&=R(\widetilde\PotQ_{\SRotT,\gat}^{k,1})
\dotplus\Lin\B{\SRotTgat}(\om)
\dotplus\devGrad\H{k+1}{\gat}(\om)\\
&=\H{k+1}{\T,\gat}(\om)
+\devGrad\H{k+1}{\gat}(\om),\\
\H{k}{\S,\gan}(\divDiv,\om)
&=R(\widetilde\PotQ_{\divDivS,\gan}^{k,1})
\dotplus\Lin\B{\divDivSgan}(\om)
\dotplus\symRot\H{k+1}{\T,\gan}(\om)\\
&=\H{k+2}{\S,\gan}(\om)
+\symRot\H{k+1}{\T,\gan}(\om),\\
\H{k+1,k}{\S,\gan}(\divDiv,\om)
&=R(\widetilde\PotQ_{\divDivS,\gan}^{k+1,k,1})
\dotplus\Lin\B{\divDivSgan}(\om)
\dotplus\symRot\H{k+2}{\T,\gan}(\om)\\
&=\H{k+2}{\S,\gan}(\om)
+\symRot\H{k+2}{\T,\gan}(\om)
\end{align*}
with bounded linear regular direct decomposition operators
\begin{align*}
\widehat\PotQ_{\DivT,\gan}^{k,1}:
\H{k}{\T,\gan}(\Div,\om)&\to R(\widetilde\PotQ_{\DivT,\gan}^{k,1})
\subset\H{k+1}{\T,\gan}(\om),\\
\widehat\PotQ_{\DivT,\gan}^{k,\infty}:
\H{k}{\T,\gan}(\Div,\om)&\to\Lin\B{\DivTgan}(\om)
\subset\H{\infty}{\T,\gan,0}(\Div,\om)
\subset\H{k+1}{\T,\gan}(\om),\\
\widehat\PotQ_{\DivT,\gan}^{k,0}:
\H{k}{\T,\gan}(\Div,\om)&\to
\H{k+1}{\S,\gan}(\om),\\
\widehat\PotQ_{\SRotT,\gat}^{k,1}:
\H{k}{\T,\gat}(\symRot,\om)&\to R(\widetilde\PotQ_{\SRotT,\gat}^{k,1})
\subset\H{k+1}{\T,\gat}(\om),\\
\widehat\PotQ_{\SRotT,\gat}^{k,\infty}:
\H{k}{\T,\gat}(\symRot,\om)&\to\Lin\B{\SRotTgat}(\om)
\subset\H{\infty}{\T,\gat,0}(\symRot,\om)
\subset\H{k+1}{\T,\gat}(\om),\\
\widehat\PotQ_{\SRotT,\gat}^{k,0}:
\H{k}{\T,\gat}(\symRot,\om)&\to
\H{k+1}{\gat}(\om),\\
\widehat\PotQ_{\divDivS,\gan}^{k,1}:
\H{k}{\S,\gan}(\divDiv,\om)&\to R(\widetilde\PotQ_{\divDivS,\gan}^{k,1})
\subset\H{k+2}{\S,\gan}(\om),\\
\widehat\PotQ_{\divDivS,\gan}^{k,\infty}:
\H{k}{\S,\gan}(\divDiv,\om)&\to\Lin\B{\divDivSgan}(\om)
\subset\H{\infty}{\S,\gan,0}(\divDiv,\om)
\subset\H{k+2}{\S,\gan}(\om),\\
\widehat\PotQ_{\divDivS,\gan}^{k,0}:
\H{k}{\S,\gan}(\divDiv,\om)&\to
\H{k+1}{\T,\gan}(\om),\\
\widehat\PotQ_{\divDivS,\gan}^{k+1,k,1}:
\H{k+1,k}{\S,\gan}(\divDiv,\om)&\to R(\widetilde\PotQ_{\divDivS,\gan}^{k+1,k,1})
\subset\H{k+2}{\S,\gan}(\om),\\
\widehat\PotQ_{\divDivS,\gan}^{k+1,k,\infty}:
\H{k+1,k}{\S,\gan}(\divDiv,\om)&\to\Lin\B{\divDivSgan}(\om)
\subset\H{\infty}{\S,\gan,0}(\divDiv,\om)
\subset\H{k+2}{\S,\gan}(\om),\\
\widehat\PotQ_{\divDivS,\gan}^{k+1,k,0}:
\H{k+1,k}{\S,\gan}(\divDiv,\om)&\to
\H{k+2}{\T,\gan}(\om)
\end{align*}
satisfying 
\begin{align*}
\widehat\PotQ_{\DivT,\gan}^{k,1}
+\widehat\PotQ_{\DivT,\gan}^{k,\infty}
+\Rot\widehat\PotQ_{\DivT,\gan}^{k,0}
&=\id_{\H{k}{\T,\gan}(\Div,\om)},\\
\widehat\PotQ_{\SRotT,\gat}^{k,1}
+\widehat\PotQ_{\SRotT,\gat}^{k,\infty}
+\devGrad\widehat\PotQ_{\SRotT,\gat}^{k,0}
&=\id_{\H{k}{\T,\gat}(\symRot,\om)},\\
\widehat\PotQ_{\divDivS,\gan}^{k,1}
+\widehat\PotQ_{\divDivS,\gan}^{k,\infty}
+\symRot\widehat\PotQ_{\divDivS,\gan}^{k,0}
&=\id_{\H{k}{\S,\gan}(\divDiv,\om)},\\
\widehat\PotQ_{\divDivS,\gan}^{k+1,k,1}
+\widehat\PotQ_{\divDivS,\gan}^{k+1,k,\infty}
+\symRot\widehat\PotQ_{\divDivS,\gan}^{k+1,k,0}
&=\id_{\H{k+1,k}{\S,\gan}(\divDiv,\om)}
\end{align*}
and 
\begin{align*}
\widehat\PotQ_{\DivT,\gan}^{k,1}
&=\widetilde\PotQ_{\DivT,\gan}^{k,1}
=\PotP_{\DivT,\gan}^{k}\DivTgan^{k}
=\PotQ_{\DivT,\gan}^{k,1}(\DivTgan^{k})_{\bot}^{-1}\DivTgan^{k},\\
\widehat\PotQ_{\DivT,\gan}^{k,\infty}
&=I_{\pi_{N(\SRotTgat)}}^{-1}\pi_{N(\SRotTgat)}\widetilde\PotN_{\DivT,\gan}^{k},\\
\widehat\PotQ_{\DivT,\gan}^{k,0}
&=\PotP_{\TRotS,\gan}^{k}(1-I_{\pi_{N(\SRotTgat)}}^{-1}\pi_{N(\SRotTgat)})\widetilde\PotN_{\DivT,\gan}^{k},\\
\widehat\PotQ_{\SRotT,\gat}^{k,1}
&=\widetilde\PotQ_{\SRotT,\gat}^{k,1}
=\PotP_{\SRotT,\gat}^{k}\SRotTgat^{k}
=\PotQ_{\SRotT,\gat}^{k,1}(\SRotTgat^{k})_{\bot}^{-1}\SRotTgat^{k},\\
\widehat\PotQ_{\SRotT,\gat}^{k,\infty}
&=I_{\pi_{N(\DivTgan\mu)}}^{-1}\pi_{N(\DivTgan\mu)}\widetilde\PotN_{\SRotT,\gat}^{k},\\
\widehat\PotQ_{\SRotT,\gat}^{k,0}
&=\PotP_{\TGrad,\gat}^{k}(1-I_{\pi_{N(\DivTgan\mu)}}^{-1}\pi_{N(\DivTgan\mu)})\widetilde\PotN_{\SRotT,\gat}^{k},\\
\widehat\PotQ_{\divDivS,\gan}^{k,1}
&=\widetilde\PotQ_{\divDivS,\gan}^{k,1}
=\PotP_{\divDivS,\gan}^{k}\divDivSgan^{k}
=\PotQ_{\divDivS,\gan}^{k,1}(\divDivSgan^{k})_{\bot}^{-1}\divDivSgan^{k},\\
\widehat\PotQ_{\divDivS,\gan}^{k,\infty}
&=I_{\pi_{N(\TRotSgat)}}^{-1}\pi_{N(\TRotSgat)}\widetilde\PotN_{\divDivS,\gan}^{k},\\
\widehat\PotQ_{\divDivS,\gan}^{k,0}
&=\PotP_{\SRotT,\gan}^{k}(1-I_{\pi_{N(\TRotSgat)}}^{-1}\pi_{N(\TRotSgat)})\widetilde\PotN_{\divDivS,\gan}^{k},\\
\widehat\PotQ_{\divDivS,\gan}^{k+1,k,1}
&=\widetilde\PotQ_{\divDivS,\gan}^{k+1,k,1}
=\PotP_{\divDivS,\gan}^{k+1,k}\divDivSgan^{k+1,k}
=\PotQ_{\divDivS,\gan}^{k+1,k,1}(\divDivSgan^{k+1,k})_{\bot}^{-1}\divDivSgan^{k+1,k},\\
\widehat\PotQ_{\divDivS,\gan}^{k+1,k,\infty}
&=I_{\pi_{N(\TRotSgat)}}^{-1}\pi_{N(\TRotSgat)}\widetilde\PotN_{\divDivS,\gan}^{k+1,k},\\
\widehat\PotQ_{\divDivS,\gan}^{k+1,k,0}
&=\PotP_{\SRotT,\gan}^{k+1}(1-I_{\pi_{N(\TRotSgat)}}^{-1}\pi_{N(\TRotSgat)})\widetilde\PotN_{\divDivS,\gan}^{k+1,k}
\end{align*}
with 
\begin{align*}
\widetilde\PotN_{\DivT,\gan}^{k}
&=1-\widetilde\PotQ_{\DivT,\gan}^{k,1},
&
\widetilde\PotN_{\divDivS,\gan}^{k}
&=1-\widetilde\PotQ_{\divDivS,\gan}^{k,1},\\
\widetilde\PotN_{\SRotT,\gat}^{k}
&=1-\widetilde\PotQ_{\SRotT,\gat}^{k,1},
&
\widetilde\PotN_{\divDivS,\gan}^{k+1,k}
&=1-\widetilde\PotQ_{\divDivS,\gan}^{k+1,k,1},\\
\end{align*}
and
$$\begin{array}{ccccc}
I_{\pi_{N(\divDivSgan\eps)}}
&:
&\Lin\B{\TRotSgat}(\om)
&\longrightarrow&\Lin\pi_{N(\divDivSgan\eps)}\B{\TRotSgat}(\om)=\Harm{}{\S,\gat,\gan,\eps}(\om)\\
{}
&{}
&\vB{\TRotSgat}{\ell}
&\longmapsto
&\pi_{N(\divDivSgan\eps)}\vB{\TRotSgat}{\ell},\\[3ex]
I_{\pi_{N(\TRotSgat)}}
&:
&\Lin\B{\divDivSgan}(\om)
&\longrightarrow
&\Lin\pi_{N(\TRotSgat)}\eps^{-1}\B{\divDivSgan}(\om)=\Harm{}{\S,\gat,\gan,\eps}(\om)\\
{}
&{}
&\vB{\divDivSgan}{\ell}
&\longmapsto
&\pi_{N(\TRotSgat)}\eps^{-1}\vB{\divDivSgan}{\ell},\\[3ex]
I_{\pi_{N(\DivTgan\mu)}}
&:
&\Lin\B{\SRotTgat}(\om)
&\longrightarrow
&\Lin\pi_{N(\DivTgan\mu)}\B{\SRotTgat}(\om)=\Harm{}{\T,\gat,\gan,\mu}(\om)\\
{}
&{}
&\vB{\SRotTgat}{\ell}
&\longmapsto
&\pi_{N(\DivTgan\mu)}\vB{\SRotTgat}{\ell},\\[3ex]
I_{\pi_{N(\SRotTgat)}}
&:
&\Lin\B{\DivTgan}(\om)
&\longrightarrow
&\Lin\pi_{N(\SRotTgat)}\mu^{-1}\B{\DivTgan}(\om)=\Harm{}{\T,\gat,\gan,\mu}(\om)\\
{}
&{}
&\vB{\DivTgan}{\ell}
&\longmapsto
&\pi_{N(\SRotTgat)}\mu^{-1}\vB{\DivTgan}{\ell}.
\end{array}$$
\end{rem}


\vspace*{5mm}
\hrule
\vspace*{3mm}


\end{document}